\documentclass[12pt]{report}
\usepackage{amsthm, amssymb, mythesis}
  \usepackage[pdftex]{graphicx}
\usepackage{epstopdf}
\usepackage{epsfig}
\usepackage{amscd}
\usepackage{amsfonts}
\usepackage{fullpage}
\usepackage{graphicx}
\usepackage{subfigure}
\usepackage{amstext}
\usepackage{setspace}
\usepackage{amsmath}
\usepackage{dsfont}
\usepackage{fancyhdr}
\usepackage[notref,notcite]{}
\usepackage{uwthesis}
\usepackage{hyperref}
\newtheorem{thm}{Theorem}[section]
\numberwithin{equation}{section}

\newcommand{\abs}[1]{\lvert#1\rvert}

\newcommand{\e}{\varepsilon}

  \def\cD{\mathcal{D}}
\def\cF{\mathcal{F}}

\def\cN{\mathcal{N}}

\def\bE{\mathbb{E}}
\def\bN{\mathbb{N}}
\def\bP{\mathbb{P}}
\def\bQ{\mathbb{Q}}
\def\bR{\mathbb{R}}
\def\bZ{\mathbb{Z}}

\def\P{\bP} \def\E{\bE}

\def\mE{\mathbf{E}}

\def\mP{\mathbf{P}}
\def\m1{\mathbf{1}}

\newcommand{\be}{\begin{equation}}
\newcommand{\ee}{\end{equation}}
\newcommand{\nn}{\nonumber}
\newcommand{\fl}[1]{\lfloor{#1}\rfloor}
\newcommand{\ce}[1]{\lceil{#1}\rceil}
\DeclareMathOperator{\Var}{Var}   
\DeclareMathOperator{\esssup}{ess\,sup}
\DeclareMathOperator{\essinf}{ess\,inf}
\def\wt{\widetilde}   
\def\ppa{p} 
\def\ppb{s} 

\begin{document}
\title{Properties of the limit shape for some last passage    growth models in   random
environments}
\author{Hao Lin}
\degree{Doctor of Philosophy}
\dept{Mathematics}
\thesistype{dissertation}

\beforepreface
\prefacesection{Abstract}

We study directed   last-passage percolation on the planar square
lattice whose weights have general distributions, or equivalently,
queues in series with general service distributions. Each row of the
last-passage  model has its own randomly chosen  weight
distribution. We first show the existence of the limiting time
constant and list its properties. Next we study the problem for
models with Bernoulli and exponential weights, for which we already
have more precise results. We then present some universality results
about the limiting time constant close to the boundary of the
quadrant. Close to the $y$-axis, where the number of random
distributions averaged over stays large, the limiting time constant
takes the same universal form as in the homogeneous model. But close
to the $x$-axis we see the effect of the tail of the distribution of
the random environment. In particular we will give some estimates of
the upper bound in this case.

\prefacesection{Acknowledgements}
I would like to express my deepest gratitude to my advisor, Prof.
Timo Sepp\"al\"ainen, for his guidance and patience throughout the
years. He not only has taught me how to make progress in
mathematical research, but also demonstrated me a good example of
combining both rigorous scholarship and accessibility to audience in
his teaching, academic talks and papers. This dissertation would not
have been possible without his continued feedback and encouragement.

I also appreciate all valuable suggestions from the committee
members, Prof. Benedek Valko, Prof. David Anderson, Prof. Jordan
Ellenberg and Prof. Gregorio Moreno-Flores. Great thanks for your
time and comments on my work.

I would like to thank Prof. Tom Kurtz and Prof. David Griffeath,
from whom I took several probability courses. You are the first
teachers who showed me a beautiful picture of probability theory and
motivated me to major in it eventually.

I would like to thank my colleagues Matthew Joseph, Rohini Kumar and
Nicos Georgiou for sharing their experience and wisdom with me. I
learned a lot from you all. Also thanks for my dear friends in the
department: Jingwei Guo, Anakewit Boonkasame, Hwan Lee, Gabriel
Pretel and many many other names. You guys have made my life
colorful here!

I am grateful to the department staff: Sharon Paulson, Mary Rice,
Vicky Whelan, Joan Wendt Yvonne Nagel and Mike Grenie. Thanks for
tolerating my endless questions and requests for assistance.

At last, indescribable thanks to my parents. Your love is the
meaning of my life!

%

\afterpreface
\chapter{Introduction}\label{introduction}

This paper studies the limit shapes of some last-passage percolation
models in random environments. Specifically, we will first derive
the hydrodynamic limit of the last-passage time for the corner
growth model with exponential weights and for two Bernoulli models
with different rules for admissible paths. Next, we will present
some universality results for the limit shape for a broader range of
underlying distributions.

We begin by introducing the corner growth model through its queueing
interpretation.  Consider service stations in series, labeled $0, 1,
2, \dotsc, \ell$, each with unbounded waiting room and first-in
first-out (FIFO) service discipline. Initially customers $0, 1, 2,
\dotsc, k$ are queued up at server $0$. At time $t=0$ customer $0$
begins service with server $0$. Each customer moves through the
system of servers  in order, joining the queue at server $j+1$ as
soon as service with server $j$ is complete.
After customer $i$ departs server $j$, server $j$ starts serving
customer $i+1$ immediately if $i+1$ has been waiting in the queue,
or then waits for customer $i+1$ to arrive from   station $j-1$.
Customers stay ordered throughout the process. Let $X(i,j)$ be the
service time that customer $i$ needs at station $j$, and $T(k,\ell)$
the time when customer $k$ completes service with server $\ell$.

Asymptotics for $T(k,\ell)$  as $k$ and $\ell$ get large   have been
investigated a great deal in the past two decades. A seminal paper by  Glynn-Whitt \cite{glyn-whit} studied the case of
i.i.d.\ $\{X(i,j)\}$.   They took advantage of the connection with
directed  last-passage percolation given by the identity \be
T(k,\ell)=\max_\pi \sum_{(i,j)\in\pi} X(i,j). \label{last-pass-1}\ee
In this model, $X(i,j)$ is a random weight assigned to the point $(i,j)$.
 The maximum is taken over non-decreasing nearest-neighbor lattice
paths $\pi\subseteq\bZ_+^2$  from $(0,0)$ to $(k,\ell)$ that are  of
the form $\pi=\{(0,0)=(x_0,y_0), (x_1,y_1),\dotsc,
(x_{k+\ell},y_{k+\ell})=(k,\ell)\}$ where
$(x_i,y_i)-(x_{i-1},y_{i-1})=(1,0)$ or $(0,1)$.  Below is a picture
of an admissible path from $(0,0)$ to $(4,3)$:
\begin{figure}[ht]
\begin{center}
\begin{picture}(220,140)(5,20)
\thinlines \put(40,20){\vector(1,0){150}}
\put(40,20){\vector(0,1){130}}
\multiput(40,40)(0,20){6}{\line(1,0){138}}
\multiput(63,20)(23,0){6}{\line(0,1){120}}
\multiput(40,20)(23,0){3}{\circle*{6}}
\multiput(86,40)(0,20){2}{\circle*{6}}
\multiput(109,60)(23,0){2}{\circle*{6}} \put(132,80){\circle*{6}}
\linethickness{0.25mm} \multiput(40,20)(23,0){2}{\vector(1,0){20}}
\multiput(86,20)(0,20){2}{\vector(0,1){20}}
\multiput(86,60)(23,0){2}{\vector(1,0){20}}
\put(132,60){\vector(0,1){20}} \put(194.7,17){$i$}
\put(30,146.5){$j$} \put(32,12){\small 0} \put(58,10){\small 1}
\put(81,10){\small 2} \put(104,10){\small 3} \put(127,10){\small 4}
\put(32,40){\small 1} \put(32,60){\small 2} \put(32,80){\small 3}
\put(32,100){\small 4}
\end{picture}
\end{center}
\caption{An admissible path to $(4,3)$.}  \label{fig1}
\end{figure}
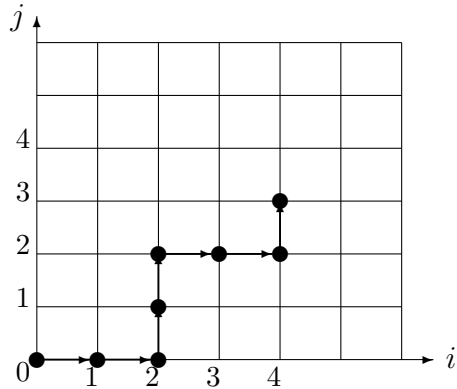%

It is easy to see that both the queueing setting and
\eqref{last-pass-1} satisfy the following recursive relationship for positive $k$
and $\ell$:
 \be T(k,\ell) =
\max\{ T(k-1,\ell) , T(k, \ell-1) \} + X(k,\ell).
\label{eqn:recursion}\ee
 Therefore if the process $\{X(i,j)\}$ in the last-passage model has the same distribution as
$\{X(i,j)\}$ in the queueing model, then $\{T(k,\ell)\}$ defined in
\eqref{last-pass-1} and in the queueing model have the same
distribution, too. \eqref{last-pass-1} together with earlier
references to this observation  can be found in \cite{glyn-whit}
(see Prop.~2.1). This particular  last-passage model is also known
as the {\sl corner growth model}.

 Next we add   a random environment to  both the queueing and last-passage percolation models.  The environment is a
sequence $\{F_j\}_{ j\in\bZ_+}$ of probability distributions,
generated by a probability measure-valued ergodic or i.i.d.\ process
with distribution $\bP$. Given the sequence $\{F_j\}_{j\in\bZ_+}$,
we assume that the
 variables  $\{X(i,j)\}$  are independent and  $X(i,j)$ has distribution $F_j$.
In the queueing picture this means  that for each $j\in\bZ_+$ the
service times $\{X(i,j): i\in\bZ_+\}$ at service station  $j$ have
common distribution $F_j$,  and at the outset
 the distributions $\{F_j\}_{j\in\bZ_+}$ themselves are chosen randomly according
 to some given law $\bP$.    Obviously  the  labels ``customer'' and ``server''
 are interchangeable because we can switch around the roles of the indices $i$
 and $j$. In the last-passage percolation model, the random
 environment means that weights assigned to points on the $j-$th row
 follow $F_j$.


Although \eqref{eqn:recursion} is simple and clear, it does not
suffice to provide much information about $T(k,\ell)$ when $k$ and
$\ell$ are large. In fact, it is not very realistic to ask what is
the distribution of $T(k,\ell)$. Instead, we let $k$ and $\ell$
go to infinity and scale $T(k,\ell)$ in a proper way. The asymptotic
regime we consider for $T(k,\ell)$ is the {\sl hydrodynamic}  one
where
 $k$ and $\ell$ are both of order $n$ and $n$ is taken to $\infty$.  Under some moment  assumptions
 standard subadditive considerations and approximations
 imply the existence  of the deterministic limit for all positive real numbers $x$ and $y$:
 \[\Psi(x,y)=\lim_{n\to\infty} n^{-1} T(\fl{nx},\fl{ny}).\]
We will also verify some properties of $\Psi(x,y)$ in Section
\ref{sec:existence and properties}: homogeneity, concavity and
continuity.

 Only in the case where the distributions $F_j$ are exponential or
 geometric has it been possible to describe explicitly the limit $\Psi$.
 This is the case of $\cdot\,/M/1$ queues in series,  which in terms of
 interacting particle systems is the same as studying either the
 totally asymmetric simple exclusion process or the zero-range process
 with constant jump rate.  For  i.i.d.\ exponential $\{X(i,j)\}$
 with rate 1, the limit $\Psi(x,y)=(\sqrt x+\sqrt y\,)^2$ was first derived by
Rost \cite{Rost1981} in a seminal paper on hydrodynamic limits of
asymmetric exclusion processes.

The random environment model with exponential $F_j$'s was studied in
\cite{andj-etal, krug-ferr, TimoKrug}. The exact $\Psi(x,y)$ can be
described implicitly. It depends on the specific
distribution of the exponential rates. In Section \ref{sec:exp} we
will see some explicit estimates of $\Psi(\alpha, 1)$ and
$\Psi(1,\alpha)$ when $\alpha$ is small. These two quantities have
different behaviors and will be discussed in further details.

Let us now set aside the queueing motivation and consider    the
last-passage model on the first quadrant   $\bZ_+^2$ of the planar
integer lattice, defined by the nondecreasing lattice paths and the
random  weights $\{X(i,j)\}$.    For the queueing application it is
natural to assume the weights nonnegative, but in the general
last-passage situation there is no reason to restrict   to
nonnegative weights.

The ideal limit shape result would have some degree of universality,
that is, apply to a broad class of distributions.  Such
  results have been obtained only close to the
boundary:  in    \cite{Martin2004}  Martin showed that in the
i.i.d.\ case, under suitable moment hypotheses and
 as  $\alpha \searrow 0$,
\be \Psi(1, \alpha) = \mu + 2\sigma \sqrt{\alpha}+ o(\sqrt{\alpha}),
\label{martin-1} \ee where $\mu$ and $\sigma^2$ are the common mean
and variance of the weights  $X(i,j)$. The $o(\sqrt\alpha)$
term in the statement means that   $\lim_{\alpha \searrow 0}
\alpha^{-1/2}\bigl[ \Psi(\alpha,1) - \mu - 2\sigma \sqrt{\alpha} \,
\bigr] =0. $ In the i.i.d.\ case $\Psi$ is symmetric so the same
holds for $\Psi(\alpha, 1)$.

Our goal is to find the form Martin's result takes in the random
environment  setting.  $\Psi$ is no longer necessarily symmetric
since the distribution  of the array $\{X(i,j)\}$ is not invariant
under transposition.
So we must ask the question separately for $\Psi(1, \alpha)$ and
$\Psi(\alpha, 1)$.

It turns out that for $\Psi(\alpha, 1)$, where the number of rows
stays large relative to the number of columns,
   the fluctuations  of the
  environment average out to the degree that our result in Theorem \ref{thm:main}  is
 essentially identical to Martin's  result in the homogeneous environment.
 We  still have
$\Psi(\alpha,1) = \mu + 2\sigma \sqrt{\alpha}+ o(\sqrt{\alpha})$ as
$\alpha \searrow 0$, where now $\mu$ is the average of the
``quenched" mean and $\sigma^2$ is the average of the ``quenched''
variance. That is, if we let $\mu_0=\int x\,dF_0(x)$ and
 $\sigma^2_0=\int (x-\mu_0)^2\,dF_0(x)$ denote the mean and  variance of the random distribution
$F_0$, and $\bE$   expectation under $\bP$, then $\mu=\bE(\mu_0)$
and $\sigma^2 = \bE( \sigma^2_0)$.

There is some evidence that we can do better than $o(\sqrt{\alpha})$
for the error term. If $\{F_j\}$ is a sequence of exponential
distributions, the result
\[\Psi(\alpha,1) = \mu + 2\sigma \sqrt{\alpha}+ O(\alpha) \] can be
shown. For general distributions with uniform boundedness, one can
achieve $o(\alpha^{\frac{3}{5}-\e})$ for any $\e>0$. Although not
yet proved, we conjecture that $O(\alpha)$ should be the answer even
for general $\{F_j\}$.  The first step should be  to prove this for
the homogeneous case.

The case $\Psi(1, \alpha)$ does not possess a clean result such as
the one above. Even though we are studying the deterministic  limit
obtained {\sl after} $n$ has been taken to infinity,  we see an
effect from the tail of the distribution of the quenched mean
$\mu_0$.   We illustrate this with the case of exponential
$\{F_j\}$. Now the number $n\alpha$ of   distributions $F_j$ is
small compared to the number $n$ of weights $X(i,j)$ in each row,
hence   the fluctuations among the $F_j$'s become prominent.  The
effect comes in two forms:  first, the leading term is no longer the
averaged mean $\mu$ but the maximal mean $\mu^*$. Second,
  if large values among the row means are rare,
  the order of the $\alpha$-dependent correction is smaller than the $\sqrt\alpha$
  seen above and this order of magnitude  depends on the
tail of the distribution of $\mu_0$.  As an exponent characterizing
this tail changes, we can see a phase transition of sorts in the
power of $\alpha$, with a logarithmic correction at the transition
point.

Intuitively, the above two phenomena suggest that when very few rows
compared to columns are available, the optimal path makes most of
its horizontal movement along the rows with large means very close
to $\mu^*$. Therefore when large means  are rare, there are not many
candidates for the optimal path, so $\Psi(1,\alpha)-\mu^*$ tends to
be smaller. The other extreme is that all means are $\mu^*$, i.e.
they are equal. In this case we can recall what happens in the
homogeneous case and guess the first $\alpha$-dependent term may be
$\sqrt{\alpha}$. We will verify this idea in the exponential model,
and derive an upper bound on $\Psi(1, \alpha)$ that gives the
correction of order $\sqrt\alpha$ as well for general distributions
under sufficient conditions.

The key idea in proving universality results in this paper is to
compare limiting time constant in models with general distributions
to that in models with normal distributions. For this purpose we
need to quantify the difference between $\Psi_F(x,y)$ and
$\Psi_G(x,y)$ for two processes $\{F_j\}_{j\in \bZ_+}$ and
$\{G_j\}_{j\in \bZ_+}$. An example of this is Lemma \ref{lem:diff}.
In the proof we use as auxiliary results bounds on the limits of
last-passage models with Bernoulli weights .

It is worth noting that with Bernoulli weights the limiting time constant $\Psi(x,y )$
has not been derived for the standard corner growth
model. $\Psi(x,y)$ can be solved in a model with  Bernoulli weights
 when the path geometry is altered suitably.  The model  we
take up is the one where the paths are weakly increasing in one
coordinate but strictly in the other.  There are two cases,
depending on which coordinate is required to increase strictly.   If
we require the  $x$-coordinate to increase strictly then an
admissible path $\{(x_0,y_0), (x_1,y_1),\dotsc, (x_m,y_m)\}$
satisfies
\be \text{$x_{i+1}-x_i =1$ 
and $y_0\le y_1\le  \dotsm\le y_m$. } \label{path-2}\ee We give a
figure below showing a possible path:
\begin{figure}[ht]
\begin{center}
\begin{picture}(220,140)(5,20)
\thinlines \put(40,20){\vector(1,0){150}}
\put(40,20){\vector(0,1){130}}
\multiput(40,40)(0,20){6}{\line(1,0){138}}
\multiput(60,20)(20,0){6}{\line(0,1){120}}
\multiput(40,20)(20,0){2}{\circle*{6}}
\multiput(80,40)(0,20){1}{\circle*{6}}
\multiput(100,60)(20,0){2}{\circle*{6}} \put(140,100){\circle*{6}}
\thicklines \multiput(40,20)(23,0){1}{\vector(1,0){20}}
\multiput(60,20)(20,20){1}{\vector(1,1){20}}
\multiput(80,40)(20,20){1}{\vector(1,1){20}}
\multiput(100,60)(20,0){1}{\vector(1,0){20}}
\multiput(120,60)(20,40){1}{\vector(1,2){20}} \put(194.7,17){$i$}
\put(30,146.5){$j$} \put(32,12){\small 0} \put(58,10){\small 1}
\put(78,10){\small 2} \put(98,10){\small 3} \put(118,10){\small 4}
\put(138,10){\small 5} \put(32,40){\small 1} \put(32,60){\small 2}
\put(32,80){\small 3} \put(32,100){\small 4}
\end{picture}
\caption{An admissible path to $(5,4)$} with $x$-coordinate
strictly increasing. \label{fig2}
\end{center}
\end{figure}
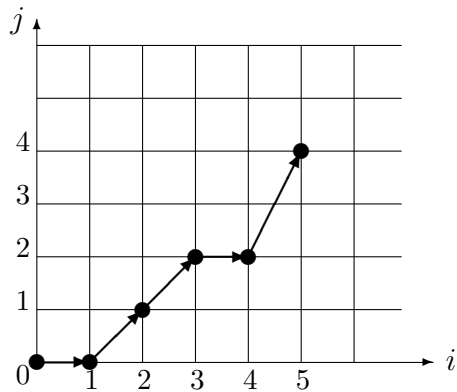%

The other case   interchanges $x$ and $y$. These cases have to be
addressed separately because the random environment attached to rows
makes the model asymmetric, i.e. the value of $\Psi(x,y)$ changes when we interchange the two coordinates.
 The sum of these two last-passage values
gives   a bound for the case where neither coordinate is required to
increase strictly in each step.

We derive   the exact limit constants for   Bernoulli models with
both types of \newline ``strict/weak" paths.   For one of them this has been done
  before by Gravner, Tracy and Widom \cite{GTW2002}.  Their proof utilizes the
  fact that the distribution of $T(k,\ell)$ is a symmetric function of the environment
in the sense that it is not affected if we interchange the
distributions in any two rows (at least for the particular Bernoulli
case they study). The proof here is completely different. It is
based on the idea   in \cite{Timo1998}  where the limit for
   the homogeneous case was derived: the last-passage model is coupled
with a particle system whose invariant distributions can be written
down explicitly, and then through some convex analysis the speed of
a tagged particle yields the explicit limit of the last-passage
model. This same approach can be adapted to the random environment
case so that results in \cite{Timo1998} can be generalized.

\medskip

{\sl Further remarks on the literature.} In \cite{glyn-whit}, a
different asymptotic regime given by $\frac{1}{n^{\frac{1+a}{2}}}
T(\fl{n^ax}, n)$ with $0<a<1$ was studied in the homogeneous model.
They derived an asymptotic result for the above quantity when the
underlying distribution has an exponentially decaying tail. It would
be interesting to see whether the result can be generalized to the
random environment case.

 Many papers also addressed
questions of fluctuations. For the last-passage model with i.i.d.\
exponential or geometric weights,  the distributional limit with
fluctuations of order $n^{1/3}$ and limit given by the Tracy-Widom
GUE distribution was proved by Johansson \cite{joha}. As for the
shape,  universality has been achieved only close to the boundary,
by  Baik-Suidan \cite{baik-suid}  and  Bodineau-Martin
\cite{bodi-mart}.

Fluctuations of the Bernoulli model with strict/weak paths and
homogeneous weights were derived first in \cite{joha01} and later
also in \cite{grav-trac-wido-01}.  For the model in a random
environment fluctuation limits appear  in \cite{GTW2002,
grav-trac-wido-02b}.

On the lattice $\bZ_+^2$ we can imagine three types of nondecreasing
paths: (i) weak-weak: both coordinates required to increase weakly,
the type used in \eqref{last-pass-1};  (ii) strict-weak:  one
coordinate increases strictly, as above in \eqref{path-2};  and
(iii) strict-strict:
 both coordinates increase
strictly so an admissible path $\{(x_0,y_0), (x_1,y_1),\dotsc,
(x_m,y_m)\}$ satisfies $x_0<\dotsm<x_m$ and $y_0< \dotsm< y_m$.   As
mentioned, with Bernoulli weights the strict-weak case is solvable
but the weak-weak case appears harder. The third case,
strict-strict, is also solvable
 with Bernoulli weights.   The shape was derived in
\cite{sepp97incr}  and recent work on this model appears in
\cite{georgiou-10}.

\medskip

{\sl Organization of the paper.} We begin by introducing the
last-passage time percolation model in Chapter \ref{chap:existence}
and verify the existence of the limiting time constant in Section
\ref{sec:existence and properties}. Next we present some results
specifically for Bernoulli models (Section \ref{sec:bernoulli}) and
exponential models (Section \ref{sec:exp}). Then we will show
universality theorems on the shape close to the boundary in Chapter
\ref{chap:universality}: in Section \ref{sec:thm:main} we present
Theorem \ref{thm:main} on $\Psi(\alpha, 1)$ and in Section
\ref{sec:1,a-thm} we have some estimates on $\Psi(1, \alpha)$.

\medskip

{\sl Some frequently used notation.}  We write \[
\underset{\bP}{\esssup}\,f=\inf\{ s\in\bR:  \bP(f>s)=0\}\] for the
essential supremum of a function $f$ under a measure $\bP$.
 $\bZ_+=\{0,1,2,\dotsc\}$,  $\bN=\{1,2,3, \dotsc\}$.
  $I(A)$ is the indicator function of event $A$.

\chapter{The existence and properties of the limiting time constant}
\label{chap:existence}
\section{The last-passage percolation model}

We give a precise definition of the  last-passage model in a random
environment. Let $\bP$ be a stationary, ergodic probability measure
on the space $\mathcal{M}_1(\bR)^{\bZ_+}$ of sequences of Borel
probability distributions on $\bR$. $\bE$ denotes expectation under
$\bP$. For some of the results in the following chapters $\P$ will
be further assumed to be an i.i.d.\ product measure.   A realization
of the distribution-valued process under $\P$ is denoted by
$\{F_j\}_{j\in\bZ_+}$. This is the environment.   Given $\{F_j\}$,
the weights  $\{X(z): z \in \bZ_+^2\}$ are independent real-valued
random variables with marginal distributions $X(i,j)\sim F_j$ for
$(i,j)\in\bZ_+^2$. Let $(\Omega, \mathcal{F}, \mathbf{P})$ be the
probability space on which all variables $\{ F_j, X(i,j)\}$ are
defined, and denote expectation under $\mP$ by $\mE$.

A (weakly)  nondecreasing path is a sequence of points
$z_0=(x_0,y_0), z_1=(x_1,y_1),\dotsc, \newline z_m=(x_m, y_m)$ in
$\bZ_+^2$
 that satisfy $x_0\leq  x_1 \leq
\dotsm \leq x_m$, $y_0 \leq y_1 \leq \dotsm \leq y_m$, and $|x_{i+1}
- x_i| + |y_{i+1} -y_i| =1$.
 For $z_1, z_2 \in \bZ^2_+$ with $z_1 \leq
z_2$ (coordinatewise ordering), let $\Pi[z_1, z_2)$ be the set of
nondecreasing paths from $z_1$ to $z_2$.   Whether   the endpoints
$z_1$ and  $z_2$ are  included in the path  makes no difference to
the limit results below, but to be precise let us include $z_1$ and
exclude $z_2$, so that we can run a subadditive argument later.

\begin{rem}\label{rem:endp}
We included the endpoint in the definition above Figure \ref{fig1}
because we wanted it to be consistent with \eqref{eqn:recursion}.

Hereafter we will always exclude the endpoint when talking about a
path between two points.
\end{rem}

 The last-passage time $T(z_1, z_2)$
 from $z_1$ to $z_2$ is defined  by
$$T(z_1, z_2) = \max_{\pi \in \Pi[z_1,z_2)}  \sum_{ z\in \pi} X(z).$$
When $z_1 = 0$ abbreviate
 $ \Pi(z) = \Pi[0, z)$ and $ T(z) = T(0, z)$.

$T(z)$ is a random variable that depends on the underlying
distributions, and will be quite complicated as $z$ moves far away
from the origin. However, the following quantity, known as the
 limiting time constant, exists under proper conditions and provide
information about the last-passage time $$\Psi(x,y) =
\lim_{n\rightarrow \infty}\frac{1}{n} T(\lfloor nx \rfloor, \lfloor
ny \rfloor).$$

\section{The existence and properties}
\label{sec:existence and properties}

We now give a set of sufficient conditions for the aforementioned
 limit to exist. Put these three assumptions on the
model: \be \mE |X(z)| < \infty, \label{momass}\ee \be
\int_0^{\infty} \Bigl\{1-\bE( F_0(x))\Bigr\}^{1/2} dx < \infty,
\label{tailass1}\ee and \be \int_0^{\infty} \underset\bP{\esssup}(1
- F_0(x))\,dx<\infty. \label{tailass2}\ee We start with these
assumptions and consider the existence of $\Psi(x,y).$

\begin{prop}  Assume $\P$ is ergodic
and satisfies  \eqref{momass}, \eqref{tailass1} and
\eqref{tailass2}. Then for all $(x,y)\in(0,\infty)^2$ the
last-passage time constant
\begin{equation}
\label{def1}
 \Psi(x,y) = \lim_{n\rightarrow \infty}\frac{1}{n} T(\lfloor
nx \rfloor, \lfloor ny \rfloor)
\end{equation}
exists as a limit both $\mP$-almost surely and in $L^1(\mP)$.
Furthermore, $\Psi(x,y)$ is a homogeneous, concave and continuous
function on $(0,\infty)^2$.  \label{pr:limit}
\end{prop}

Assumption \eqref{tailass1}  is also used for the constant
distribution case, see (2.5) in \cite{Martin2004}. Some further
control along the lines of assumption
 \eqref{tailass2} is required for our case.  For example,
suppose $1-F_j(x)=e^{-\xi_j x}$ for random $\xi_j\in(0,\infty)$.
Then  \eqref{tailass2} holds iff ${\essinf}_\bP (\xi_0)>0$. If the
distribution of  $\xi_0$ is not bounded away from zero,
  $n^{-1}T(n,n)\to\infty$ because
we can  simply collect all the weights from the row with minimal
$\xi_j$ among $\{\xi_0,\dots,\xi_n\}$. However, assumption
\eqref{tailass1} can be satisfied without bounding $\xi_0$ away from
zero.

\begin{proof}
We first prove the theorem for integer pairs $(x,y)$, and then
extend to rational numbers and finally real numbers.

\textbf{Step 1:} consider $(x,y) \in \bN^2$. Set $Z_{m,n} = - T((mx,
my),(nx, ny))$ for $0\leq m <n$, and verify that under the
distribution $\mP$, $Z_{m,n} $ satisfies assumptions (i), (ii) and
(iii) in Liggett's version of the subadditive ergodic theorem
\cite[p.~358]{durrett}. In particular,  $Z_{0,m} + Z_{m,n} \geq
Z_{0,n}$, $\{Z_{nk, (n+1)k}, n \geq 1\}$ is ergodic for each $k$,
and the distribution of $\{ Z_{m,m+k}, k \geq 1\}$ does not depend
on $m$.

We need to work harder on condition (iv), i.e. we need to show $\mE
Z^+_{0,1} < \infty$ and for each $n$, $\mE Z_{0,n} \geq \gamma n$
for some $\gamma> -\infty$. It is easy to see
\begin{align*}
\mE  Z^+_{0,1}  \leq  \mE  |T((0, 0),(x, y))| \leq \mE  \sum_{0\leq
i\le x, 0\leq j\le y}| X(i,j)| = (x+1)(y+1) \mE  |X(0,0)|<\infty.
\end{align*}

Next we show $\mE Z_{0,n} \geq \gamma n$ for some $\gamma> -\infty$
under \eqref{tailass1} and \eqref{tailass2}. This is trivially true
for a Bernoulli model where given $\{F_j\}$ the weights have
marginal distributions \be\label{eqn:bernrates}
P(X(i,j)=1)=1-F_j(u)=1-P(X(i,j)=0).\ee Therefore this Bernoulli
model satisfies all conditions of the Subadditive ergodic theorem,
and $\Psi_{Ber[1-F(u)]}(x,y)$, the limiting time constant, is
well-defined $\mP-a.s.$ and in $L^1(\mP)$ for $(x,y)\in \bN^2$. We
will see an upper bound \eqref{bernoulli4} for $\Psi(x,y)$ of the
Bernoulli model in the next chapter. We use it here without proof in
the following calculation: \be\begin{split}\label{eqn:subad}
\frac{1}{n}\mE Z_{0,n} &\ge - \frac{1}{n}\mE  \max_{\pi \in \Pi(nx,
ny)} \sum_{z\in \pi} X(z)_+ = - \frac{1}{n}\mE  \max_{\pi \in
\Pi(nx, ny)} \sum_{z\in \pi} \int_0^{\infty}
I(X(z) > u)\,du\\
& \geq - \frac{1}{n}\mE  \int_0^{\infty} \max_{\pi \in \Pi(nx, ny)}
\sum_{z\in \pi} I(X(z) > u)\,du\\
& = -\int_0^{\infty} \sup_n  \frac{1}{n}\mE  \max_{\pi \in \Pi(nx,
ny)} \sum_{z\in \pi} I(X(z) > u)\,du
= - \int_0^{\infty}  \Psi_{Ber[1-F(u)]}(x,y)\,du\\
& \geq -(y+4\sqrt{xy})\int_0^{\infty} \sqrt{1-\bE  F_0(u)}\, du- x
\int_0^{\infty} \bigl(1-\underset{\bP}{\essinf} F_0(u)\bigr) du.
\end{split}\ee
 Here $I(A)$ is the indicator function of event $A$.   By assumptions \eqref{tailass1} and \eqref{tailass2},
 $\mE Z_{0,n} \geq n\gamma$ for a constant  $\gamma > -\infty$. These
estimates  justify the application of the subadditive ergodic
theorem. So now  for $(x,y)\in  \bN^2$, we can define the following
$\mP-a.s.$ and $L^1(\mP)$ limit
\[\Psi(x,y) = \lim_{n\rightarrow \infty}\frac{1}{n} T(nx , ny).\]

\textbf{Step 2:} take $ (x,y)\in
\bigl(\mathbb{Q}\cap(0,\infty)\bigr)^2$. Let $x = \frac{x_1}{x_2}$
and $y = \frac{y_1}{y_2}$ be in their reduced forms, i.e. $x_i$ and
$y_i$ are positive for $i=1,\ 2$ and $ \gcd(x_1,x_2) =
\gcd(y_1,y_2)=1$. Let $k$ be the least common multiple of $x_2$ and
$y_2$, so $(kx, ky) \in \bN^2$.

For every positive integer $n$, write $n = Mk + r$ for integers $M$
and $r$ such that $0 \leq r \leq k-1$. Then we have
$$Mkx \leq \lfloor nx \rfloor \leq (M+1)kx, \quad Mky \leq \lfloor
ny \rfloor \leq (M+1)ky.$$

So if we denote $z_1(n) = (Mkx, Mky)$ and $z_2(n) = \bigl((M+1)kx,
(M+1)ky\bigr)$, we have the following inequalities from
superadditivity:
$$T(z_1(n)) + T\bigl((z_1(n), (\lfloor nx \rfloor, \lfloor ny
\rfloor )\bigr) \leq T(\lfloor nx \rfloor, \lfloor ny \rfloor ) \leq
T(z_2(n)) - T\bigl((\lfloor nx \rfloor, \lfloor ny \rfloor ), z_2(n)
\bigr)$$ Obviously, $$T\bigl(z_1(n), (\lfloor nx \rfloor, \lfloor ny
\rfloor )\bigr) \geq - \max_{\pi \in \Pi[z_1(n), z_2(n))} \sum_{z\in
\pi} |X(z)|,$$ and this leads to
\begin{equation}
\label{rational1} T(z_1(n)) - \max_{\pi \in \Pi[z_1(n), z_2(n))}
\sum_{z\in \pi} |X(z)| \leq T(\lfloor nx \rfloor, \lfloor ny \rfloor
).
\end{equation}

Let $\e >0$ be any small positive number,
\begin{equation}\label{borelcantelli}
\begin{split}
&\sum_{n=1}^\infty \mP(\max_{\pi \in \Pi[z_1(n), z_2(n))} \sum_{z\in
\pi}
|X(z)| \geq n\e )\\
=& \sum_{n=1}^\infty \mP(\max_{\pi \in \Pi [0, (kx, ky))}
\sum_{z\in \pi} |X(z)| \geq n\e )\\
\le & \frac{1}{\e} \int_0^\infty \mP(\max_{\pi \in \Pi [0, (kx,
ky))} \sum_{z\in \pi} |X(z)| \geq x )\, dx\\
= &\frac{1}{\e} \mE \Bigl(\max_{\pi \in \Pi [0, (kx, ky))}
\sum_{z\in \pi}
|X(z)|\Bigr)\\
\le& \frac{1}{\e} (kx+1) (ky+1) \mE |X(0,0)|<\infty.
\end{split}
\end{equation}

By Borel-Cantelli Lemma, $\frac{1}{n}\max_{\pi \in \Pi[z_1(n),
z_2(n))} \sum_{z\in \pi} |X(z)| \rightarrow 0$ $\mP-a.s.$ Dividing
through by $n$  and taking limit in \eqref{rational1} gives that
$$\frac{1}{k}\Psi(kx, ky) =  \lim_{n\rightarrow \infty} \frac{1}{n} T(z_1(n))\leq \liminf_{n\rightarrow \infty}  \frac{1}{n} T(\lfloor nx \rfloor, \lfloor ny \rfloor ) \quad \mP- a.s.$$
Similarly, we can show the other direction
$$\limsup_{n\rightarrow \infty}  \frac{1}{n} T(\lfloor nx \rfloor, \lfloor ny \rfloor ) \leq   \lim_{n\rightarrow \infty} \frac{1}{n} T(z_2(n))=\frac{1}{k}\Psi(kx, ky) \quad \mP-a.s.$$

This shows that \[ \lim \frac{1}{n} T(\lfloor nx \rfloor, \lfloor ny
\rfloor ) = \frac{1}{k}\Psi(kx,ky)\quad \mP-a.s.\]   The definition
of $\Psi(x,y)$ has now been extended from integer points to rational
points by \be \label{eqn:itor}\Psi(x,y) = \frac{1}{k}\Psi(kx,ky),\ee
where $k$ is defined at the beginning of Step 2.

From \eqref{rational1}, we get \[\Bigl(\frac{1}{n}T(z_1(n))-
\frac{1}{k}\Psi(kx,ky)\Bigr) - \frac{1}{n}\max_{\pi \in \Pi[z_1(n),
z_2(n))} \sum_{z\in \pi} |X(z)| \leq \frac{1}{n}T(\lfloor nx
\rfloor, \lfloor ny \rfloor ) - \frac{1}{k}\Psi(kx,ky). \]

Similarly, \[ \frac{1}{n}T(\lfloor nx \rfloor, \lfloor ny \rfloor )
- \frac{1}{k}\Psi(kx,ky)\le \Bigl(\frac{1}{n}T(z_2(n))-
\frac{1}{k}\Psi(kx,ky)\Bigr) + \frac{1}{n}\max_{\pi \in \Pi[z_1(n),
z_2(n))} \sum_{z\in \pi} |X(z)|. \]

Note that for $i=1,2$, because $(kx,ky)\in\bN^2,$
\[\lim_{n\rightarrow \infty}\mE \abs{\frac{1}{n}T(z_i(n))-
\frac{1}{k}\Psi(kx,ky)} = 0.\]  In addition,
\[\frac{1}{n}\mE \Bigl( \max_{\pi \in \Pi[z_1(n), z_2(n))} \sum_{z\in \pi}
|X(z)|\Bigr) \le \frac{1}{n}(kx+1)( ky+1) \mE |X(0,0)|,
\] which also converges to $0$ as $n$ goes to infinity.

Therefore the $L^1(\mP)$ convergence follows \[\lim_{n\rightarrow
\infty} \mE \abs{\frac{1}{n}
T(\fl{nx},\fl{ny})-\frac{1}{k}\Psi(kx,ky) } = 0. \]

One can fairly easily check the following properties: for positive
rational pairs $(x_1,y_1)$ and $(x_2,y_2)$
\begin{enumerate}
\item
homogeneity: $\Psi(cx_1,cy_1) = c\Psi(x_1,y_1)$ for any positive
rational number $c$.
\item
superadditivity: $\Psi(x_1+x_2,y_1+y_2) \geq \Psi(x_1,y_1) +
\Psi(x_2, y_2)$ .
\item
The above two together imply concavity: for rational $0<c<1$ and let
$x = cx_1+(1-c)x_2$, $y = cy_1+(1-c)y_2$, we have
\begin{equation}
\label{rationalconcavity} \Psi(x,y) \geq c\Psi(x_1,y_1) +(1-c)
\Psi(x_2,y_2).
\end{equation}
\end{enumerate}

\textbf{Step 3:} we extend the definition to $(x,y)\in
(0,\infty)^2$. First, we prove $ \lim_{n} \frac{T(n, \lfloor ny
\rfloor)}{n}$ exists $\mP-a.s.$ for all $y\in (0,\infty)$

If $y$ is not rational, we can pick $y_1, y_2 \in \mathbb{Q}\cap
(0,\infty)$ such that $y_1 < y < y_2$. By picking the optimal path
from the origin to $(n, \lfloor ny_1 \rfloor)$ and then moving
directly to $(n, \lfloor ny \rfloor)$, we have the following
inequality:
\begin{equation}
\label{real1} T(n, \lfloor ny_1 \rfloor) + \sum_{j=\lfloor ny_1
\rfloor  }^{\lfloor ny \rfloor-1} X(n,j) \leq  T(n, \lfloor ny
\rfloor).
\end{equation}

We now take a random subsequence $\{n_k, k=1,2,...\}$ such that
$$\frac{1}{n_k}T(n_k, \lfloor n_k y \rfloor)\rightarrow \liminf
\frac{1}{n}T(n, \lfloor ny \rfloor).$$  By the strong law of large
numbers $\frac{1}{n}\sum_{j= 1 }^{\lfloor ny \rfloor - \lfloor ny_1
\rfloor} X(0,j)$ converges to $(y-y_1)\mE X(0,0)$ $\mP-a.s.$ and in
probability. Therefore if we fix an $\e>0$, then for every $\ell=1,
2, \ldots$ we can find an integer $N(\ell)$
 such that
\be \label{eqn:subsequence}\mP \bigl(|\frac{1}{n}\sum_{j= 1
}^{\lfloor ny \rfloor - \lfloor ny_1 \rfloor} X(0,j)- (y-y_1)\mE
X(0,0)| > \e \bigr) < 2^{-\ell}\ee for all $n> N(\ell)$.

Since $\frac{1}{n}\sum_{j=\lfloor ny_1 \rfloor }^{\lfloor ny
\rfloor-1} X(n,j)$ has the same distribution as $\frac{1}{n}\sum_{j=
1 }^{\lfloor ny \rfloor - \lfloor ny_1 \rfloor} X(0,j)$,
\eqref{eqn:subsequence} implies that from $\{n_k\}$ we can select a
further subsequence $\{n_{k(l)}\}$ such that
$$\sum_{l=1}^\infty \mP \bigl(|\frac{1}{n_{k(l)}}\sum_{j= \lfloor n_{k(l)}y_1
\rfloor  }^{\lfloor n_{k(l)}y \rfloor-1 } X(n,j)- (y-y_1)\mE X(0,0)|
> \e \bigr) < \infty,$$ and this shows that $\mP-a.s.$ we have $$ \lim_{l\rightarrow \infty}\frac{1}{n_{k(l)}}\sum_{j= \lfloor n_{k(l)}y_1
\rfloor  }^{\lfloor n_{k(l)}y \rfloor-1 } X(n,j)= (y-y_1)\mE
X(0,0).$$

Hence by dividing through by $n$ and taking limits along this
subsequence $\{n_{k(l)}\}$ in  \eqref{real1} we get
\begin{equation}
\label{real2} \Psi(1, y_1) + (y-y_1)\mE X(0,0) \leq \liminf
\frac{1}{n}T(n, \lfloor ny \rfloor).
\end{equation}Similarly, we can show
\begin{equation}\label{real3}
 \limsup  \frac{1}{n}T(n, \lfloor ny
\rfloor )\leq \Psi(1, y_2) - (y_2-y)\mE X(0,0).
\end{equation} Note that both of the above inequalities hold $\mP-a.s.$

So now it is natural that we want to let $y_1$ and $y_2$ approach
$y$ from both sides. We need the following lemma.
\begin{lem}\label{lem:lrlimits}
If $f(x)$ is a concave function defined on $\bQ \cap (0,\infty)$,
then it can be extended uniquely to a continuous function on
$(0,\infty)$ by \be \label{ext-def} f(x) =
\lim_{y\in \bQ, y\rightarrow x}f(y).\ee%
\end{lem}

\begin{proof}
Let $x>0$ be a fixed real number. We first prove the one-sided limit
$\lim_{y\in \bQ, y\rightarrow x^-}f(y)$ is well-defined. If this is
not the case, then we can find two sequences of rational numbers
$\{u_n\}$ and $\{v_n\}$ such that they both approach $x$ from below,
and the limits $A_1 \equiv \lim_n f(u_n)$ and $A_2 \equiv \lim_n
f(v_n)$ both exist with $A_1> A_2$. Take $\e
>0$ small enough. We can  find three rational numbers $u, u' \in \{u_n\}$ and
$v\in \{v_n\}$ such that $u< v < u'$, and $f(u),\ f(u') > A_1 -\e$,
$f(v) < A_2+\e$. Take a rational number $0<q<1$ such that $v = qu
+(1-q)u'$ , then
$$f(v) < A_2+\e <  A_1 -\e < qf(u)
+(1-q)f(u'). $$

This contradicts \eqref{rationalconcavity}, so the left limit
$A=\lim_{y \in \mathbb{Q}, y \rightarrow x^-} f(y)$ exists.
Similarly, $B=\lim_{y \in \mathbb{Q}, y \rightarrow x^+} f(y)$ also
exists. We only need to show $A=B$.

Assume $A>B$, and choose $0<\e < \frac{A-B}{2}$. Then there exists
$\delta>0$ such that for any rational numbers $u$ and $v$ with $x-
\delta < u < x$ and $x<v<x + \delta$,
$$A-\e<f(u)< A+\e, \quad B-\e<f(v)<
B+\e.$$

Take rational numbers $u$ and $v$ such that $x-\delta< u < x < v < x
+\delta$ and $\frac{x-u}{v-u} < 1 - \frac{2\e}{A-B}$, and pick a
rational number $\alpha \in (\frac{x-u}{v-u} , 1 -
\frac{2\e}{A-B})$. It follows that
$$\alpha f(v) +(1-\alpha) f(u)>\alpha(B- \e) + (1-\alpha )(A-\e) = \alpha B +(1-\alpha)A - \e > B +\e.$$
However, by the choice of $\alpha$, $\alpha v +(1-\alpha)u$ is a
rational number in $(x, v)$, so $$f\bigl(\alpha v +(1-\alpha)
u\bigr) < B +\e < \alpha f(v) +(1-\alpha) f(u).$$

This again violates \eqref{rationalconcavity}, and by contradiction
we reject $A>B$. Similarly, we can show $A<B$ is not possible
either. Hence $A=B$ and $\lim_{y\in \mathbb{Q}, y \rightarrow x}
f(y)$ exists.

For $x\in \bQ \cap (0,\infty)$, we need to show this limit is
consistent with the original value $f(x)$. We can repeat the above
proof by contradiction and modify it when necessary. Specifically,
we let $A= f(x)$, $B=\lim_{y \in \mathbb{Q}, y \rightarrow x^+}
f(y)$ and assume $A>B$; in the following part we choose $u=x$ and $v
\in (x,x+\delta)$ such that $f(v) \in (B-\e, B+\e)$, and take a
rational number $\alpha\in (0, 1-\frac{2\e}{A-B} )$. We can check
$f\bigl(\alpha v +(1-\alpha) u\bigr) < \alpha f(v) +(1-\alpha)
f(u)$, so it contradicts \eqref{rationalconcavity} and $A>B$ is
rejected. Similarly we reject $A<B$ and get $A=B$. Since the
two-sided limit $\lim_{y\in \mathbb{Q}, y \rightarrow x} f(y)$
exists, this gives \eqref{ext-def} for $x \in \bQ\cap(0,\infty)$.

One can quickly check that the extension keeps concavity: for $x$
and $y$ in $(0,\infty)$ and $0<c<1$, take rational sequences $x_n
\rightarrow x$, $y_n \rightarrow y$ and $c_n \rightarrow c$, then
\begin{align*}f\bigl( cf(x) + (1-c)f(y) \bigr) & = \lim_{n }f\bigl( c_nf(x_n) + (1-c_n)f(y_n) \bigr)\\
&\ge \lim_n c_n f(x_n) +\lim_n(1-c_n)f(y_n)\\
&= cf(x)+(1-c)f(y). \end{align*} Since a finite concave function on
an open set is continuous by Theorem 10.1 of \cite{convex}, we get
continuity. The uniqueness is trivial because any continuous
function must satisfy \eqref{ext-def}.
\end{proof}

Let us return to the proof of Proposition \ref{pr:limit}. Since
$\Psi(1,y)$ is concave function defined on $\bQ\cap(0,\infty)$, it
can be extended to $(0,\infty)$ by \eqref{ext-def}. We let $y_1$ and
$y_2$ approach $y$ in \eqref{real2} and \eqref{real3}, and get
$$ \lim_{n\rightarrow \infty} \frac{1}{n}T(n, \lfloor ny
\rfloor ) = \lim_{u \in \mathbb{Q}, u \rightarrow y} \Psi(1,u)\quad
\mP-a.s.$$ Therefore we can extend the definition to $y\in
(0,\infty)$ \be \label{eqn:rtore1}\Psi(1,y) = \lim_{u \in
\mathbb{Q}, u \rightarrow y} \Psi(1,u).\ee

 We then extend
the definition of $\Psi(x,y)$ to any $(x,y) \in (0,\infty)^2$, and
show that
\begin{equation}\label{homogeneity}\lim_{n\rightarrow \infty} \frac{1}{n}T(\lfloor
nx \rfloor, \lfloor ny \rfloor ) = x \Psi(1, \frac{y}{x}) \quad
\mP-a.s.\end{equation}

If we write $m = \lfloor nx \rfloor$, then $m \leq nx < m+1$, hence
$\lfloor m\frac{y}{x} \rfloor \leq \lfloor ny \rfloor \le \lfloor
(m+1)\frac{y}{x} \rfloor.$ It is clear that
\begin{equation} \label{homo2} \begin{split}
 T(m, \lfloor m\frac{y}{x} \rfloor ) - \sum_{i= \lfloor
m\frac{y}{x} \rfloor }^{\lfloor ny \rfloor-1 } |X(m, i)| &\leq
T(\lfloor nx \rfloor, \lfloor ny \rfloor )\\ &\leq T(m, \lfloor
(m+1)\frac{y}{x} \rfloor ) +\sum_{i=  \lfloor ny \rfloor}^{\lfloor
(m+1)\frac{y}{x} \rfloor-1 } |X(m, i)|.\end{split}
\end{equation}

Similarly to \eqref{borelcantelli}, we can run a Borel-Cantelli
argument and claim that as $n$ goes to infinity,
$$\frac{1}{n}\sum_{i= \lfloor m\frac{y}{x} \rfloor }^{\lfloor ny
\rfloor } |X(m, i)| \rightarrow 0
 \quad \text{and} \quad \frac{1}{n}\sum_{i=  \lfloor ny
\rfloor}^{\lfloor (m+1)\frac{y}{x} \rfloor } |X(m, i)| \rightarrow 0
\quad \mP-a.s.$$

Therefore dividing through by $n$ and taking limits in \eqref{homo2}
gives that $\mP-a.s.$
$$x\Psi(1, \frac{y}{x}) \leq \liminf_{n\rightarrow \infty}
\frac{1}{n}T(\lfloor nx \rfloor, \lfloor ny \rfloor )\leq
\limsup_{n\rightarrow \infty} \frac{1}{n}T(\lfloor nx \rfloor,
\lfloor ny \rfloor ) \leq x\Psi(1, \frac{y}{x}).  $$ So
\eqref{homogeneity} is proved. We can then define for $(x,y)\in
(0,\infty)^2$ that \be \label{eqn:rtore2} \Psi(x,y) =
x\Psi(1,\frac{y}{x}).\ee

Finally, we prove $L^{1}(\mP)$ convergence for $(x,y)\in
(0,\infty)^2.$ From \eqref{real1} and its counterpart in the other
direction, we get
\begin{align*}&\Bigl(\frac{1}{n}T(n, \lfloor ny_1 \rfloor)- \Psi(1,y_1)\Bigr) +
\Bigl(\Psi(1,y_1) - \Psi(1,y)\Bigr) +
\frac{1}{n}\sum_{j=\lfloor ny_1 \rfloor }^{\lfloor ny \rfloor-1} X(n,j)\\
\leq& \,\frac{1}{n}T(n, \lfloor ny \rfloor) - \Psi(1,y)\\
\le &\Bigl(\frac{1}{n}T(n, \lfloor ny_2 \rfloor)- \Psi(1,y_2)\Bigr)
+ \Bigl(\Psi(1,y_2) - \Psi(1,y)\Bigr) + \frac{1}{n}\sum_{j=\lfloor
ny \rfloor }^{\lfloor ny_2 \rfloor-1} X(n,j).
\end{align*}

We have shown that for rational numbers $y_1$ and $y_2$,
$\lim_{n\rightarrow \infty}\mE \abs{\frac{1}{n}T(n, \lfloor ny_i
\rfloor)- \Psi(1,y_i)} =0. $ Then
\begin{align*} &\lim_{n\rightarrow \infty}\mE \abs{\frac{1}{n}T(n,
\lfloor ny \rfloor)- \Psi(1,y)}\\\le&
\Bigl(\abs{\Psi(1,y_2) - \Psi(1,y)}+(y_2-y)\mE\abs{X(0,0)} \Bigr)\\
&\vee \Bigl(\abs{\Psi(1,y_1) - \Psi(1,y)}+(y-y_1)\mE\abs{X(0,0)}
\Bigr)
\end{align*}
We let $y_1$ and $y_2$ approach $y$ and get \[\lim_{n\rightarrow
\infty}\mE \abs{\frac{1}{n}T(n, \lfloor ny \rfloor)- \Psi(1,y)} =
0.\]

We can use a very similar argument starting from \eqref{homo2} to
get
\[\lim_{n\rightarrow
\infty}\mE \abs{\frac{1}{n}T(\fl{nx}, \lfloor ny \rfloor)-
x\Psi(1,\frac{y}{x})} = 0.\] So $L^1(\mP)$ convergence is proved.

So now we have started from the definition of $\Psi(x,y)$ on integer
points and extended it to $(0,\infty)^2$ by \eqref{eqn:itor},
\eqref{eqn:rtore1} and \eqref{eqn:rtore2}. We can  immediately
extend the homogeneity, superadditivity, and concavity conditions to
real points. Again by Theorem 10.1 of \cite{convex} a finite concave
function on an open set is continuous, we get continuity.

Now we have finished the proof of Proposition \ref{pr:limit}.
\end{proof}


We may also define $\Psi(x,y)$ using supremum. If we denote $x_n =
\lfloor nx \rfloor$ and $y_n=\lfloor ny \rfloor$, by
superadditivity, we have
\[
\begin{split}
&T(x_m, y_m) +  T\bigl((x_m,y_m), (x_m + x_n, y_m+y_n)\bigr)
\\&+T\bigl( (x_m + x_n, y_m+y_n), (x_{m+n}, y_{m+n})\bigr)\leq T(x_{m+n},
y_{m+n}).
\end{split}
\]

We note that $x_{m+n}- x_m - x_n = 0$ or $1$, and so is $y_{m+n}-
y_m - y_n$. If $\mE X(0,0)\geq 0$, then we can easily check that
$\mE T(0,1)$, $\mE T(1,0)$ and $\mE T(1,1)$ are nonnegative. This
gives $\mE T(x_m, y_m)+ \mE T(x_n, y_n) \leq \mE T(x_{m+n},
y_{m+n})$, which leads to another definition that

$$\Psi(x,y) = \lim_{n\rightarrow \infty} \frac{1}{n} \mE T(x_n, y_n) = \sup_{n} \frac{1}{n} \mE T(x_n, y_n)
=\sup_{n} \frac{1}{n} \mE T(\lfloor nx \rfloor, \lfloor ny
\rfloor).$$

This alternative definition will be helpful in some settings where
we need to estimate the upper bound of last-passage times. It may
not be true if $\mE X(0,0)< 0$. An easy counterexample is the case
where $X(z)\equiv -1$ for all $z\in \bZ_+^2$. We easily see that
$\frac{1}{n} \mE T(\lfloor nx \rfloor, \lfloor ny \rfloor) =
\frac{1}{n} (-\lfloor nx \rfloor - \lfloor ny \rfloor) >\frac{1}{n}
(-nx - ny) = -x-y =\Psi(x,y)$ when $x$ and $y$ are not integers.
However, if $x$ and $y$ are both integers, $\Psi(x,y) = \sup_n
\frac{1}{n} \mE T(nx, ny)$ is a valid definition regardless of the
sign of $\mE X(0,0)$.

\chapter{Results for Bernoulli and exponential models} \label{chap:bern and exp}

\section{Bernoulli   models with strict-weak paths  in a random environment} \label{sec:bernoulli}

 As we have seen in the proof of Proposition
\ref{pr:limit}, last-passage models with Bernoulli-distributed
weights can play an important role when we study general models. As
for models with Bernoulli weights, one of the major difficulties is
that there is no explicit results so far about $\Psi(x,y)$ with the
weakly increasing paths. For this reason in this chapter we first
study Bernoulli models with two different types of admissible paths,
and eventually give an estimate of $\Psi(x,y)$ with the weakly
increasing paths.

The environment is now an i.i.d.\ sequence $\{p_j\}_{j\in\bZ_+}$ of
numbers $p_j\in[0,1]$, with distribution $\bP$.  Given $\{p_j\}$,
the weights $\{X(i,j)\}$ are independent with marginal distributions
$P(X(i,j)=1)=p_j=1-P(X(i,j)=0)$.  We consider two last-passage times
that differ by the type of admissible path:  for $z_1,
z_2\in\bZ_+^2$ \be T_{\rightarrow}(z_1, z_2) = \max_{\pi \in
\Pi_{\rightarrow}[z_1,z_2)}  \sum_{ z\in \pi} X(z)
\quad\text{and}\quad T_{\uparrow}(z_1, z_2) = \max_{\pi \in
\Pi_{\uparrow}[z_1,z_2)}  \sum_{ z\in \pi} X(z). \label{bernT}\ee
 In terms of coordinates denote the endpoints by  $z_k=(a_k,b_k)$, $k=1,2$.  Admissible paths
$\pi\in\Pi_{\rightarrow}[z_1,z_2)$ are  of the form $\pi=\{ (a_1,
y_0)(a_1+1,y_1), \dotsc, (a_2-1,y_{a_2-a_1-1})\}$ with $b_1\le
y_0\le y_1\le \dotsm\le y_{a_2-a_1-1}\le b_2$. Please see Figure
\ref{fig2}. Again note that now we always exclude the end point as
was declared in Remark \ref{rem:endp}.

Symmetrically paths $\pi\in \Pi_{\uparrow}[z_1,z_2)$ are of the form
$\pi=\{ (x_0, b_1), (x_1, b_1+1)\dotsc, \newline (x_{b_2-b_1-1},
b_2-1)\}$ with $a_1\le x_0\le x_1\le \dotsm\le x_{b_2-b_1-1}\le
a_2$. Thus paths in $\Pi_{\rightarrow}[z_1,z_2)$ increase strictly
in the $x$-direction while those in $\Pi_{\uparrow}[z_1,z_2)$
increase
strictly in the $y$-direction. 

 As before we simplify   notation with  $T_{\rightarrow}(0, z)= T_{\rightarrow}(z)$.
  The almost sure limits are denoted by
\be
 \Psi_{\rightarrow}(x,y) = \lim_{n\to \infty}\frac{1}{n} T_{\rightarrow}(\lfloor
nx \rfloor, \lfloor ny \rfloor) \quad\text{and}\quad
 \Psi_{\uparrow}(x,y) = \lim_{n\to \infty}\frac{1}{n} T_{\uparrow}(\lfloor
nx \rfloor, \lfloor ny \rfloor) \label{bernPsi}\ee
 for $(x,y)\in(0,\infty)^2$.
The proof of the existence of the above limits can be outlined as
follows: we first verify the assumptions of the subadditive ergodic
theorem for $(x,y)\in \bN^2$, and note that the moment assumption is
trivial for Bernoulli models because of the uniform boundedness;
then we extend the definition to all $(x,y)\in (0,\infty)^2$ using
the same argument as we had in the previous chapter. We will omit
the details and claim the existence of the limits.

\begin{rem}
In \eqref{eqn:subad} we used a result \eqref{bernoulli4} from this
chapter. To remove the concern for circularity, we note that the
proof of Proposition \ref{pr:limit} works for the Bernoulli models
even without knowing \eqref{bernoulli4}. The logic progression
actually should be: Proposition \ref{pr:limit} holds for Bernoulli
models, then we derive \eqref{bernoulli4}, and apply
\eqref{bernoulli4} to prove Proposition \ref{pr:limit}  for more
general models under moment conditions.
\end{rem}

The next theorem gives the explicit  limits.  \eqref{bernoulli5} is
the same as   in \cite[Thm.~1]{GTW2002}. Inside the $\bE[\,\dotsm]$
expectations below $p$ is the random Bernoulli probability.  Let
$b=\underset{\bP}{\esssup}\,p$  denote the maximal probability.

We prove the formulas and inequalities first for
$\Psi_{\rightarrow}$ and then for $\Psi_{\uparrow}$.  It is
convenient to  assume   $b< 1$.
 Results for the case  $b=1$ follow by taking a limit.

\begin{thm}
The limits in \eqref{bernPsi} are as follows for $x,y\in(0,\infty)$.
  \be \label{bernoulli5} \Psi_{\rightarrow} (x,y) =
\begin{cases}
 bx + y(1-b) \bE \Bigl[  \frac{p }{b-p}\Bigr], & {x}/{y} \ge
\bE\Bigl[ \frac{p(1-p) }{(b - p)^2}\Bigr]\\[6pt]
 yz_0^2\bE\Bigl[  \frac{1-p }{(z_0 -  p)^2} \Bigr] -y,
 &\bE\Bigl[ \frac{p }{1-p}\Bigr] < {x}/{y} < \bE\Bigl[ \frac{p(1-p) }{(b-p)^2}\Bigr]\\[6pt]
x, &  0< {x}/{y} \le \bE\Bigl[ \frac{p }{1-p}\Bigr]
\end{cases}
\ee with $z_0\in(b,1)$ uniquely defined by the equation
$${x}/{y}=   \,\bE \Bigl[ \frac{  p(1-p)}{(z_0-  p)^2}\Bigr].$$

\be \label{bernoulli6} \Psi_{\uparrow} (x,y) = \begin{cases}
y-y z_0^2 \bE\Bigl[ \frac{1-p}{(z_0 +p)^2} \Bigr], & 0 < {x}/{y} <
\bE\Bigl[\frac{1-p}{p} \Bigr]\\[6pt]
y, & {x}/{y} \ge \bE\Bigl[ \frac{1-p}{p} \Bigr]
\end{cases}\ee
with $z_0\in(0,\infty)$ uniquely defined by the equation
$${x}/{y} = \bE \Bigl[   \frac{p(1-p)}{(z_0 +p)^2} \Bigr]. $$
\label{thm:bernoulli}\end{thm}

\newpage
\begin{figure}[h]
\centering \subfigure[Approximation of the function
$\Psi_{\rightarrow}(1,\alpha)$ from simulation.]{
    \label{psi(1,alpha)} 
    \includegraphics[width=9.5cm]{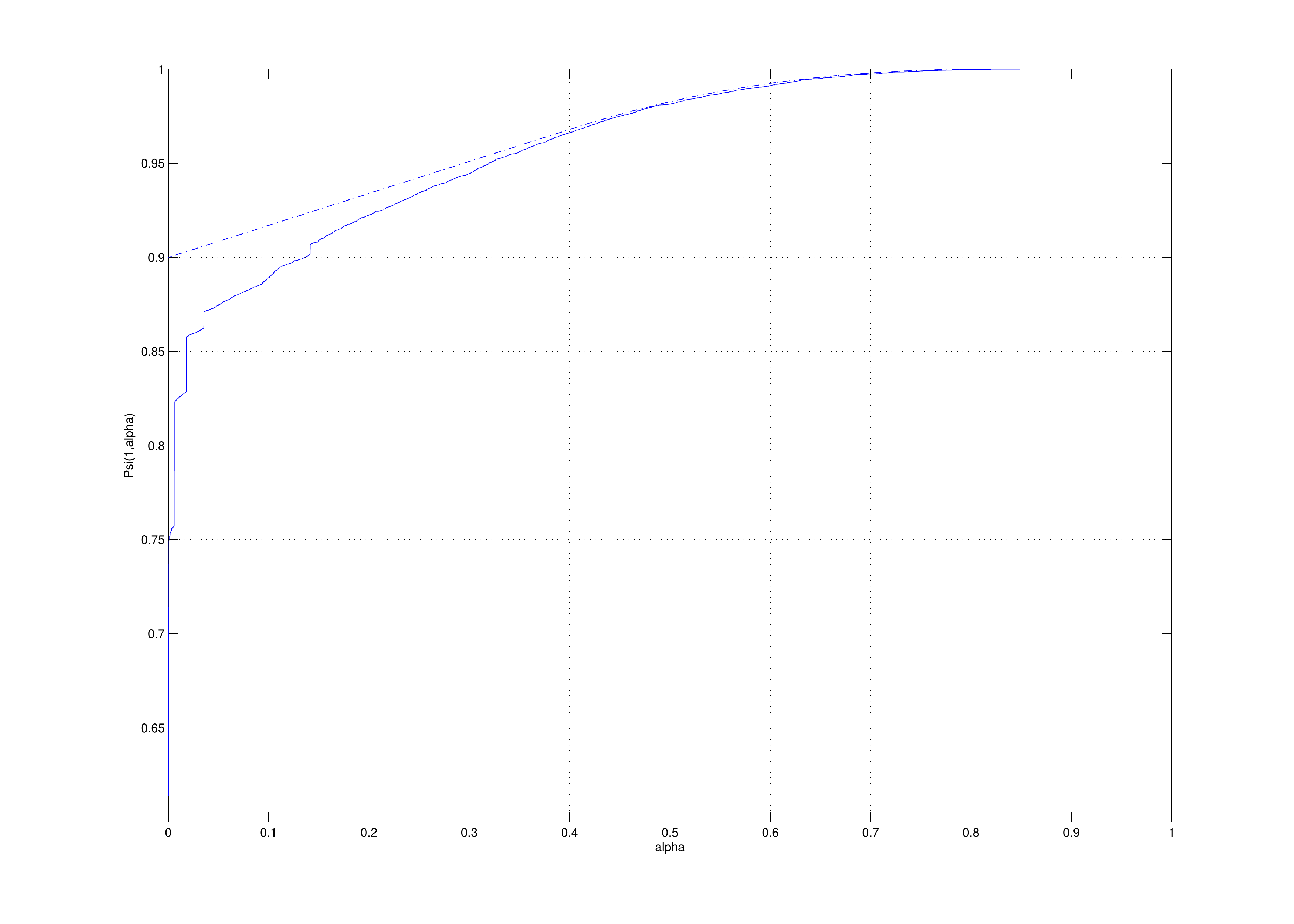}}
 \subfigure[ Approximation of the function $\Psi_\uparrow(\alpha,1)$ from
simulation.]{
   \label{psi(alpha,1)} 
 \includegraphics[width=9.5cm]{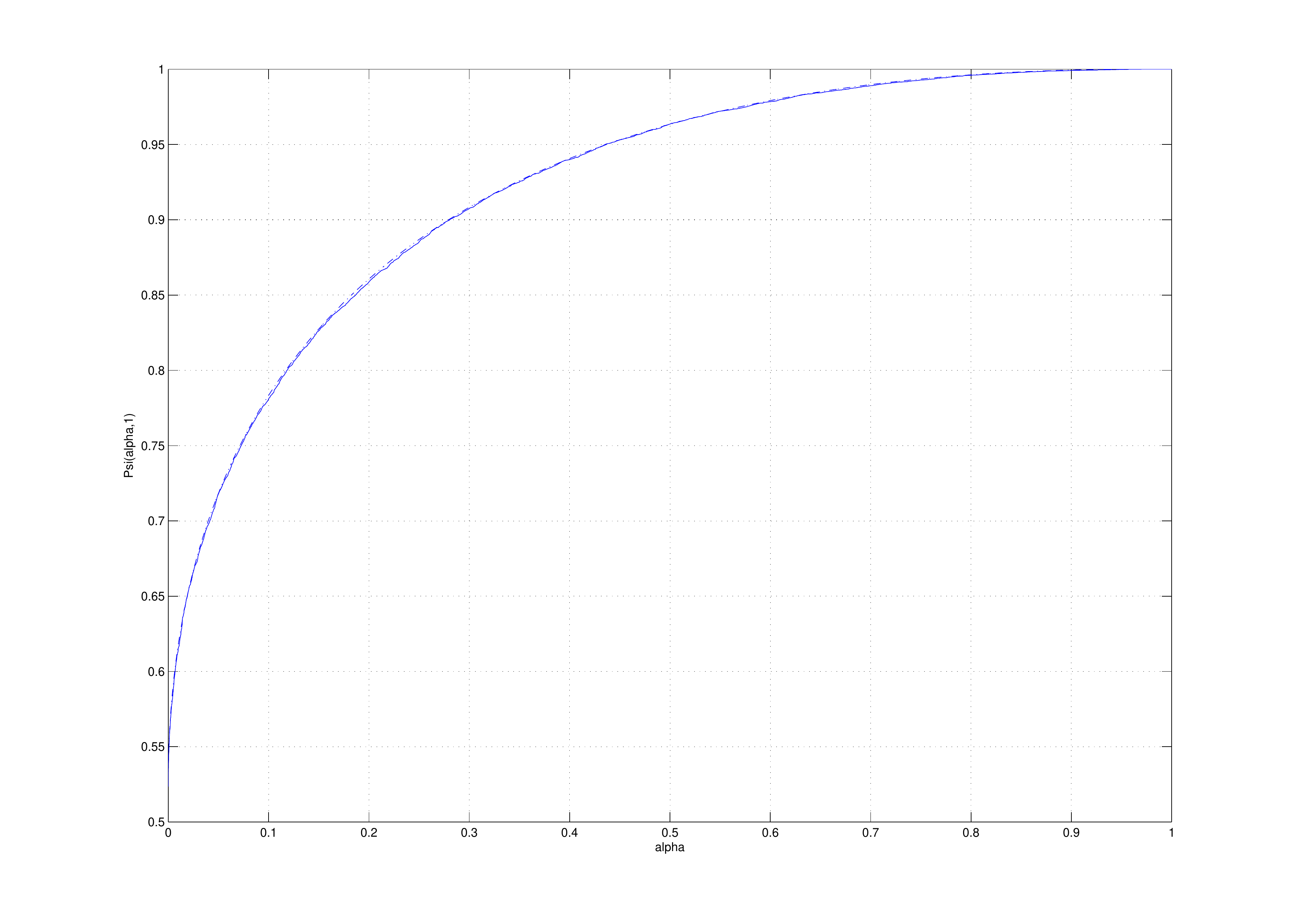}}
 \caption{simulations to verify Theorem \ref{thm:bernoulli}}
\end{figure}

\newpage
To illustrate this result, we present two plots based on
simulations: the rates $p_j$ are chosen to be i.i.d. with $\bP(p_j
\le x) = 1 - (\frac{0.9 - x}{0.5})^3$ supported on $[0.4, 0.9]$.

The dashed curves are the precise values described by Theorem
\ref{thm:bernoulli}. The solid curves are the approximations
$\frac{T_{\rightarrow}(n,\fl{n\alpha})}{n}$ and
$\frac{T_{\uparrow}(\fl{n\alpha},n)}{n}$ respectively when $n =
20,000$.

In Figure \ref{psi(1,alpha)}, the approximation is not very accurate
when $\alpha$ is close to 0 after zooming in. The reason is that the
density of Bernoulli rates near $b$ is low, so the first few rows
are not likely to have means close to $b$.  In Figure
\ref{psi(alpha,1)}, the error is almost negligible. We can also
clearly see that $\Psi_{\rightarrow}(1,\alpha)$ and
$\Psi_{\uparrow}(\alpha,1)$ approach to different limits as $\alpha
\searrow 0$.

Our second result gives  simplified  bounds   that are    useful for
the proof of the  Theorem \ref{thm:main} in the next chapter.    Let
$\bar p=\bE(p)$ be the mean of the environment.
  $\Psi(x,y)$ is the limiting time constant with weakly increasing paths defined
in Proposition \ref{pr:limit}.

\begin{thm}
The following three inequalities hold for the Bernoulli model:
\begin{equation}
\label{bernoulli1} \Psi_{\rightarrow} (x,y)  \leq bx + 2
\sqrt{\bar{p}(1-b) xy },
\end{equation}
\begin{equation}
\label{bernoulli2}
 \Psi_{\uparrow} (x,y) \leq \bar{p}y + 2 \sqrt{\bar{p}(1-\bar{p})xy}
\end{equation}
and
\begin{equation}
\label{bernoulli3}
 \Psi (x,y)\leq \bar{p}y + 4 \sqrt{\bar{p}(1-\bar{p})xy} + bx.
\end{equation}
\label{thm:bernoulli2}\end{thm}

\eqref{bernoulli3}  follows from   \eqref{bernoulli1} and
\eqref{bernoulli2}   because $ \Psi (x,y) \leq \Psi_{\rightarrow}
(x,y)+\Psi_{\uparrow} (x,y)$.
 Another loose estimate we will use later following
\eqref{bernoulli3} is
\begin{equation}
\label{bernoulli4} \Psi (x,y) \leq \bar{p}y + 4
\sqrt{\bar{p}(1-\bar{p})xy} + bx \leq (y+4\sqrt{xy})\sqrt{\bar{p}} +
bx.
\end{equation}

 \begin{proof}[Proof of \eqref{bernoulli5} and \eqref{bernoulli1}]
 We   adapt    the proof  from  \cite{Timo1998} to the random environment situation
and sketch the main points. Results from \cite{Timo1998} are
directly applicable in the random environment and will be quoted
without further explanation.

Consider now the environment $\{p_j\}$ fixed, but the weights
$X(i,j)$ random. For integers $0\le s<t$ and $a, k$ define  an
inverse to the last-passage time as
 $$\Gamma ((a,s), k, t) = \min\{l \in \bZ_+: T_{\rightarrow}((a+1,s+1),(a+l+1, t)) \geq k \}.$$
 For the special case $k=0$ we define $\Gamma ((a,s), 0, t) =0$, but $\Gamma ((a,s), k, t) >0$ for $k>0$.
 Knowing the limits of the variables $\Gamma$ is the same as knowing
 $\Psi_\rightarrow$.  By the homogeneity of $\Psi_\rightarrow$  it is enough to find $h(x)=\Psi_\rightarrow(x,1)$.  By the homogeneity and superadditivity of  $\Psi_\rightarrow$,
 $h$ is concave and nondecreasing.
Let $g$ be the inverse function of $h$ on $\bR_+$.   Then $g$ is
convex and nondecreasing, and (7.4) in \cite{Timo1998} still holds
here:
\[  tg(x/t)= \lim_{n\to\infty} \frac1n \Gamma((0,0), \fl{nx}, \fl{nt} ).  \]

To find these functions we construct an exclusion-type  process
$z(t)=\{z_k(t): k\in\bZ\}$ of labeled, ordered  particles
$z_k(t)<z_{k+1}(t)$ that jump leftward on the lattice $\bZ$, in
discrete time $t\in\bZ_+$. Given an initial configuration
$\{z_i(0)\}$ that satisfies $z_{i-1}(0) \leq z_i(0)-1 $  and
 $\liminf_{i \rightarrow -\infty} |i|^{-1} z_i(0) > -1/b$,  the evolution is defined by
 \be z_k(t) = \inf_{i: i
\leq k} \{ z_i(0)  + \Gamma((z_i(0),0),k-i, t )\}, \quad k\in\bZ, \,
t\in\bN. \label{bernz}\ee It can be checked that  $z(t)$ is a
well-defined Markov process, in particular that
  $z_k(t) > -\infty$ almost surely. These claims are identical to
  Lemma 5.2 and Lemma 5.3 in \cite{Timo1998}.

Define the process $\{\eta_i(t)\}$  of interparticle distances by
  $\eta_i (t)= z_{i+1}(t) -z_{i}(t)$ for   $i\in \bZ$ and $t\in
\bZ_+$. By Prop.~1 in \cite{Timo1998} process $\{\eta_i(t)\}$  has a
family of i.i.d.\ geometric invariant distributions indexed by the
mean  $u\in [1, {b}^{-1})$ and defined by \be P(\eta_i = n)  =
u^{-1} (1-u^{-1})^{n-1}, \quad n\in\bN. \label{berninv}\ee
 Let  $x_k(t) = z_k(t-1) - z_{k}(t)\ge 0$  be the absolute size of the jump of
the $k$th particle from time $t-1$ to $t$, and let $q_t=1-p_t$. From
(6.5) in \cite{Timo1998}, in the stationary process \be P(x_k(t) =
x) =
\begin{cases}
        (1-up_t) q_t^{-1} &
        x=0\\
         p_t(1-up_t) q_t^{-1}(u-1)^x(uq_t)^{-x} & x=1,2,3,\dotsc\\
\end{cases}\label{bern-step-distr}\ee

We track the motion of particle $z_0(t)$ in a stationary situation.
The initial state is defined by setting $z_0(0)=0$ and by letting
$\{\eta_i(0)\}$ be i.i.d.\ with common distribution \eqref{berninv}.
With $k=0$, divide by $t$ in \eqref{bernz} and take $t\to\infty$.
Apply laws of large numbers inside the braces in \eqref{bernz}, with
some simple estimation to pass the limit through the infimum, to
find the average speed of the tagged particle: \be -\, \lim_{t
\rightarrow \infty} \frac{1}{t} z_0(t)
 =\sup_{x\ge 0} \{ ux-g(x)\} \equiv f(u). \label{bern-lim-8}\ee
For further details please refer to the proof of (7.15) in
\cite{Timo1998}.

The last equality defines the speed $f$ as
   $f=g^+$, the {\sl monotone conjugate} of $g$.
   It is natural to set $f(u)=0$ for $u\in[0,1)$,  $f(b^{-1})=f((b^{-1})-)$,  and $f(u)=\infty$ for $u>b^{-1}$.
  By \cite[Thm.~12.4]{rock-ca}
\be
 g(x) = \sup_{u \geq 0} \{xu - g^+(u)\} = \sup_{1\leq u \leq {1}/{b}} \{xu - f(u)\}.
 \label{berndual} \ee

Since $z_0(t)$ is a sum of  jumps $x_0(k)$ with distribution
\eqref{bern-step-distr} we have the second moment bound
$\sup_{t\in\bN}  \mE[(t^{-1} z_0(t))^2]<\infty$, and consequently
the limit  in \eqref{bern-lim-8} holds also in expectation. From
this \be\begin{aligned} f(u)&=  -\,\lim_{t \rightarrow \infty}\mE[
{t}^{-1}  z_0(t)]
= \lim_{t \rightarrow \infty} \mE\Bigl[\,  {t}^{-1}\sum_{k=1}^{t} x_0(k)\,\Bigr] = \mE[  x_0(0)] \\
&=\bE  \sum_{x=1}^\infty x(u-1)^x (uq)^{-x} p (1-up) (1-p)^{-1} =
\bE \Bigl[ \frac{pu(u-1)}{1-up} \Bigr] .
\end{aligned}\label{bernf}\ee

Next we will find the explicit expression of $g(x)$ from
\eqref{bernf} and \eqref{berndual}. To find the supremum of $xu-
f(x)$, we compute its first derivative and find it equal to $x-
\bE\Bigl[\frac{1-p}{(b - p)^2}-1 \Big]$.

When $\bE\Bigl[\frac{p}{1-p} \Bigr] \leq x \leq \bE\Bigl[
\frac{1-p}{(b - p)^2}-1\Bigr]$, the equation \be x+1 =
\bE\Bigl[\frac{1-p}{(1 - u_0 p)^2 }\Bigr] \ee has a solution $u_0
\in [1, \frac{1}{b}]$, so $g(x) = xu_0 - f(u_0)$. If $x <
\bE\Bigl[\frac{p}{1-p} \Bigr]$, then $g(x)=x$. If $x
> \bE\Bigl[
\frac{1-p}{(b - p)^2}-1\Bigr]$, $g(x) = \frac{x}{ b} -
f(\frac{1}{b})$. Therefore,
 \be\label{gx2} g(x) = \begin{cases}
 {x}/{b} - b^{-1}(1-b) \bE\Bigl[\frac{ p}{(b-p)}\Bigr]  & x \ge b^2 \bE
\Bigl[\frac{ (1-p)}{(b - p)^2}-1\Bigr]\\[5pt]
  u_0^2 \bE \Bigl[\frac{
p(1-p)}{(1- u_0 p)^2}\Bigr]  &  \bE \Bigl[ \frac{ p}{1-p}\Bigr] < x
<
b^2\bE \Bigl[\frac{ (1-p)}{(b - p)^2} -1\Bigr]\\[5pt]
x &  0<  x \le \bE \Bigl[\frac{ p}{1-p}\Bigr]
\end{cases} \ee
where $u_0 \in (1,{b}^{-1})$ is uniquely defined by the equation \[
x+1 =\bE\Bigl[{(1-p)}{(1 - u_0 p)^{-2}}\Bigr].\]

Then we need to find  the inverse function $h(x)= g^{-1}(x)$ and
then $\Psi_{\rightarrow} (x,y) = y h( {x}/{y})$. $g(x)$ has three
different cases when $x$ takes different values. The first and last
cases $g^{-1}(x)$can be calculated directly, and for the second case
we only need to interchange the positions of $x$ and $g(x)$ in their
defining equations. Therefore \be h(x) = g^{-1}(x) =
\begin{cases}
 bx + \bE \Bigl[\frac{(1-b)p }{b-p} \Bigr]& x > \bE \Bigl[
\frac{p(1-p) }{(b - p)^2}\Bigr]\\
 \bE \Bigl[ \frac{(1-p) }{(1 - u_0 p)^2 }-1\Bigr]& \bE \Bigl[ \frac{p }{1-p}\Bigr] \leq x \leq \bE \Bigl[ \frac{p(1-p) }{(b-p)^2}\Bigr]\\
x &  0\leq x <\bE \Bigl[\frac{p }{1-p}\Bigr]
\end{cases}
\ee with
$$x= \bE \Bigl[  \frac{ u_0^2 p(1-p)}{(1- u_0 p)^2}\Bigr].$$
Since $\Psi_{\rightarrow} (x,y) = y h( {x}/{y})$, \eqref{bernoulli5}
proved.

To  prove \eqref{bernoulli1} we return to the duality
\eqref{berndual} and write \be  g(x)
  \geq \sup_{1 \leq u < 1/b}\{xu -
\tilde f(u)\}\quad \text{for}\quad \tilde f(u) = \frac{u(u-1)}{1-ub}
\bar{p}. \label{bern11}\ee
 $\tilde f'(u) = x$ is solved by  $u^* = b^{-1}\Bigl({1 - \sqrt{\frac{(1-
b)\bar{p}}{ bx+\bar{p}}}}\Bigr). $

 When $x \geq \frac{\bar{p}}{1-b}$,  we have $u^* \in [1,\frac{1}{b})$, and then
 \[\begin{split}
g(x)  \geq  xu^* - \tilde f(u^*)
       =  \frac{1}{b^2} \bigl(\sqrt{(1-b)\bar{p}} - \sqrt{bx+\bar{p}}\,\bigr)^2.
      \end{split}\]
Consequently
\[\begin{split} g^{-1}(x) &\leq \frac{1}{b} \bigl(\sqrt{b^2 x}
+\sqrt{(1-b)\bar{p}}\,\bigr)^2- \frac{\bar{p}}b  = bx -\bar{p} + 2
\sqrt{(1-b)\bar{p} x}.
 \end{split}\]
When $x < \frac{\bar{p}}{1-b}$,  the supremum in \eqref{bern11} is
attained at $u = 1$, and in this case  $$g^{-1}(x)\le x  \leq bx + 2
\sqrt{(1-b) \bar{p} x}.$$ The bound \eqref{bernoulli1} now follows
from $\Psi_{\rightarrow} (x,y) = y g^{-1}({x}/{y}).$
\end{proof}

\begin{proof}[Proof of \eqref{bernoulli6} and \eqref{bernoulli2}]
The scheme is the same as above, so we omit some more details. The
inverse of the last-passage time is now defined
\[ \Gamma ((a,s), k, t) = \min\{l \in \bZ_+: T_{\uparrow}((a,s+1),(a+l, t+1)) \geq k \}. \]
When $k=0$ we define $\Gamma ((a,s), 0, t)=0$. Vertical distance
$t-s$ allows for at most $t-s$ points with value $1$, so the above
quantity must be set equal to $\infty$ for $k>t-s$. The particle
process $\{z(t): t\in\bZ_+\}$ is defined by the same formula
\eqref{bernz} as before but it is qualitatively different. The
particles still jump to the left, but the ordering rule is now
$z_k(t)\le z_{k+1}(t)$ so particles are allowed to sit on top of
each other. Well-definedness of the dynamics needs no further
restrictions on admissible particle configurations   because the
minimum in \eqref{bernz} only considers $i\in\{k-t,\dotsc, k\}$ so
it is well-defined for all initial configurations
$\{z_i(0):i\in\bZ\}$ such that $z_i(0)\le z_{i+1}(0)$.

The following can be checked.  Under a fixed environment $\{p_j\}$,
the gap process $\{\eta_i(t)=z_{i+1}(t)-z_i(t): i\in\bZ\}$ has
i.i.d.\ geometric invariant distributions $ P(\eta_k =n) =(
\frac{1}{1+u}) (\frac{u}{1+u})^n$, $n\in\bZ_+$, indexed by the mean
$u\in\bR_+$.   In this stationary situation the successive jumps
$x_k(t) = z_k(t-1) - z_{k}(t)$ of a tagged particle  have
distribution
\[ P(x_{k}(t) =y )=
\begin{cases}
\frac{1}{1+up_t} &y=0 \\[4pt]
(\frac{u}{u+1})^y \frac{p_t}{1+up_t} & y \geq 1.
\end{cases}
\]
From here the analysis proceeds the same way as for the other model.
The speed function is defined by
 \begin{align*}
f(u)&= -\lim_{n \rightarrow \infty} \mE\bigl[\, {n}^{-1} z_0(n)
\bigr]  = \mE[ x_0(0)]
   = u(u+1)\bE \Bigl[\, \frac{p}{1+up} \, \Bigr].
              \end{align*} We now calculate $g(x) = \sup_{u\geq 0} \{xu -
f(u)\}$.

When $\bar{p} \leq x \leq 1$, there exists a non-negative solution
$u_0$ to the equation \be \label{relation1} x = f'(x)=1- \bE\Bigl[\,
\frac{1-p}{(1+ up)^2}\, \Bigr],\ee and it follows that
$$g(x) = xu_0 -f(u_0) = \bE\Bigl[\, \frac{u_0^2
p(1-p)}{(1+u_0 p)^2}\, \Bigr].$$

When $x < \bar{p}$, we can check the sup is taken at $u=0$ and thus
$g(x) = 0$. If $x>1$, $g(x) = +\infty$ so \be\label{gx1} g(x) =
\begin{cases}+\infty & x>1\\ \bE\Bigl[\, \frac{u_0^2
p(1-p)}{(1+u_0^2 p)^2}\, \Bigr] & \bar{p} \leq x \leq 1 \\
0 &  0\leq x \leq \bar{p}
\end{cases}
\ee and it has an inverse function
\begin{align*}g^{-1}(x) = \begin{cases}
1-\bE\Bigl[\, \frac{(1-p)}{(1+u_0 p)^2} \, \Bigr] & 0 < x \leq
\bE\Bigl[\,\frac{1-p}{p}\, \Bigr]\\
1 & x > \bE\Bigl[\, \frac{1-p}{p}\, \Bigr]
\end{cases}
\end{align*}
with $u_0$ defined implicitly in
$$x = \bE\Bigl[\, \frac{u_0^2
p(1-p)}{(1+u_0 p)^2} \, \Bigr]. $$ Since $\Psi_{\uparrow} (x,y)  =
yg^{-1}(\frac{x}{y})$, \eqref{bernoulli6} follows easily.

To prove \eqref{bernoulli2},   note that
\[\begin{split}
g(x) = \sup_{u\geq 0} \{xu - f(u)\}
      &\geq \sup_{u\geq 0} \Bigl\{xu -  \frac{\bar{p}u(u+1)}{1+u\bar{p}}
     \Bigr\}\\[3pt]
     &= \begin{cases}
        \frac{1}{\bar{p}}(\sqrt{1-x} - \sqrt{1 - \bar{p}})^2 &
        \bar{p} \leq x \leq 1\\
        0  & 0 \leq x \leq \bar{p}.\\
\end{cases}
\end{split}
\]
We used Jensen's inequality and  concavity of
$p\mapsto\frac{p}{1+up}$. From this
\[g^{-1}(x) \leq
\begin{cases}
\bar{p} -\bar{p}x + 2\sqrt{\bar{p}(1-\bar{p})x} & 0 \leq x \leq
\frac{1 - \bar{p}}{\bar{p}}\\
1  &  x>\frac{1 - \bar{p}}{\bar{p}}\\
\end{cases}
\]
 and  \eqref{bernoulli2}  follows.  \end{proof}

\section{Limiting shapes for exponential models}
\label{sec:exp}

\subsection{Estimate of $\Psi_G(\alpha,1)$}

In this section, we consider both cases $\Psi(1,\alpha)$ and
$\Psi(\alpha,1)$ for the exponential model where  some (partially)
explicit calculation is possible.

Let  $\{\xi_j\}_{j\in\bZ_+}$ be an i.i.d.\ sequence of random
variables  $0<c\leq \xi_j$ with common distribution $m$.  We assume
$c$ is the exact lower bound: $m[c,c+\e)>0$ for each $\e>0$. The
distribution function of exponential distribution with random
parameter $\xi_j$ is $G_j(x)=1-e^{-\xi_j x}$ for $x>0$. Its mean is
$\frac{1}{\xi_j}$ and the variance is $\frac{1}{\xi_j^2}$. Then the
essential supremum of the random  mean is $\mu^*=c^{-1}$.

We assume $X(i,j) \sim G_j(x)$ for $(i,j)\in \bZ_+^2$ and write
$\Psi_G$ for the limiting time constant defined in \eqref{def1}.
Define $\mu_G = \int_{[c,\infty)} \frac{1}{\xi} m(d\xi)$ and
$\sigma^2_G = \int_{[c,\infty)} \frac{1}{\xi^2} m(d\xi)$.

 An implicit description of the limit shape was derived in
\cite{TimoKrug} by way of studying an exclusion process with random
jump rates attached to particles. We recall the result from
\cite{TimoKrug} here. One explicit shape is needed for the proof of
Theorem \ref{thm:main} also, so this result will serve there too.

 Define first a critical value
  $u^* = \int_{[c,\infty)}
\frac{c}{ \xi -c }\, m(d\xi)\in(0,\infty]$. For $0 \leq u < u^*$
define $a = a(u)$ implicitly by
      $$u =  \int_{[c,\infty)} \frac{a}{ \xi-a}\,m(d\xi). $$
The function $a(u)$ is strictly increasing, strictly concave,
continuously differentiable and one-to-one from $0< u < u^*$ onto
$0< a <c$. We let $a(u) = c$ for $u\geq u^*$ if $u^* < \infty$. Then
define $g:\bR_+\to\bR_+$ by \be g(y) = \sup_{u\geq 0}\{ -yu +
a(u)\}, \quad y\ge0. \label{ga-dual}\ee The function $g$  is
monotone decreasing, continuous, and $g(y)=0$ for $y\ge
a'(0+)=1/\mu_G$.  It is the level curve of the time constant. The
equations connecting the two are $g(y)=\inf\{ x>0: \Psi_G(x,y)\ge
1\}$ and
  \be \Psi_G(x,y) = \inf\{t \geq 0: tg({y}/{t}) \geq x\}. \label{PsiGg}\ee

We first derive a result of the asymptotic behavior of $\Psi(x,y)$
close to the $y-$axis. From homogeneity, we can focus on the
univariate funtion $\Psi(1,\alpha)$.
\begin{thm}\label{exponential} For the random exponential
distributions defined above, as $\alpha \searrow 0$
$$\Psi_{G} (\alpha
, 1 ) = \mu_G + 2 \sigma_G \sqrt{\alpha}+ O(\alpha).$$
\end{thm}

\begin{proof}
Recall the definition of the limit shape $\Psi_G(\alpha, 1)$ from
\eqref{PsiGg}. From \eqref{ga-dual} one can read that $tg(1/t)$ is
nondecreasing in $t$. Thus by  \eqref{PsiGg}
 $\Psi_G(\alpha, 1) = t = t(\alpha)$ such that $t
g({1}/{t}) = \alpha$.

Next we argue that when $\alpha$ is close enough to $0$, $g({1}/{t})
= - u_0/t + a(u_0)$ for some $0< u_0 < u^*$ with $a'(u_0) =
{1}/{t}$. Since $a(0)=0$ and $ a(u^*-) =c$, strict concavity gives
for $0<u<u^*$
\begin{align*}
\Bigl\{\int_{[c,\infty)}\frac{\xi}{ (\xi-c)^2} m(d\xi)\Bigr\}^{-1} =
a'(u^*-) &< a'(u) =
\Bigl\{\int_{[c,\infty)} \frac{\xi}{ (\xi-a(u))^2} m(d\xi)\Bigr\}^{-1} \\
&< a'(0+) =  \Bigl\{\int_{[c,\infty)} { \xi}^{-1}
m(d\xi)\Bigr\}^{-1}= \frac{1}{\mu_G}.
\end{align*}
On the other hand, $  0<   \Psi_G(\alpha, 1) - \mu_G \leq
C\sqrt{\alpha} +  C\alpha $ where the second inequality comes from
comparing $\{G_j\}$  in \eqref{eqn:diff} with identically zero
weights. Thus  when $\alpha$ is small enough, ${1}/{t}$ is in the
range of $a'$.  Consequently   there exists  $u_0\in(0,u^*)$ such
that $a'(u_0) = {1}/{t}$, or equivalently,
\begin{equation}
\label{eqn:t1} \int_{[c,\infty)} \frac{\xi}{ (\xi-a(u_0))^2} m(d\xi)
=t.
\end{equation} From the choice of $t$, $\alpha = t g({1}/{t}) =
t\bigl(- u_0/t +a(u_0)\bigr) = -u_0 + ta(u_0)$ and so
\begin{equation} \label{eqn:t2} \Psi_G(\alpha, 1) = t =
\frac{\alpha}{a(u_0)} + \frac{u_0}{ a(u_0)} = \frac{\alpha}{a(u_0)}
+ \int_{[c,\infty)} \frac{1}{\xi-a(u_0)}  m(d\xi).
\end{equation}
Combining  \eqref{eqn:t1} and \eqref{eqn:t2}  gives
\begin{equation}
\label{eqn:alhpa/a}\alpha = a(u_0)^2 \int_{[c,\infty)}
\frac{1}{(\xi-a(u_0))^2} m(d\xi).  \end{equation} From this
 \[  a(u_0)^2\sigma^2_G= a(u_0)^2 \int_{[c,\infty)} \frac{1}{\xi^2} m(d\xi)
\leq \alpha. \] Hence we have $0\le a(u_0)\le \sqrt \alpha
/\sigma_G.$

When $\alpha$  and hence $a(u_0)$ is   small, \eqref{eqn:alhpa/a}
and the last bound on $a(u_0)$ yield
\begin{equation}\begin{split}
0\le \frac{\alpha}{a(u_0)^2} - \sigma^2_G &= \int_{[c,\infty)}
\bigl[\frac{1}{(\xi - a(u_0))^2} - \frac{1}{\xi^2}\bigr] m(d\xi)\\
 &= \int_{[c,\infty)} \frac{2\xi a(u_0)-
 a(u_0)^2}{\xi^2(\xi-a(u_0))^2}m(d\xi)\\
 &\le 2\int_{[c,\infty)} \frac{ a(u_0)}{\xi(c-a(u_0))^2}m(d\xi) = O(\sqrt \alpha).
\end{split}
\end{equation}
Consequently  \[ \frac{\sqrt \alpha}{a(u_0)}-\sigma_G
=\frac{\frac{\alpha}{a(u_0)^2} - \sigma^2_G}{\frac{\sqrt
\alpha}{a(u_0)}+\sigma_G}= O(\sqrt \alpha).\]

Now we put all the above together to prove the lemma.
 \begin{align*}
&\Psi_G(\alpha, 1) -\mu_G - 2\sigma_G\sqrt\alpha \\
&= \frac{\alpha}{a(u_0)} + \int_{[c,\infty)} \frac{1}{\xi-a(u_0)}
m(d\xi)
-\mu_G - 2\sigma_G\sqrt\alpha\\
  &= \frac{\alpha}{a(u_0)} + \int_{[c,\infty)} \Bigl[ \,\frac{1}{\xi} + \frac{1}{\xi^2} a(u_0) +
  O\bigl(a(u_0)^2\bigr)\Bigr] m(d\xi)-\mu_G - 2\sigma_G\sqrt\alpha\\
  &= \sqrt{\alpha} \Bigl( \,\frac{\sqrt\alpha}{a(u_0)}-\sigma_G\Bigr)
  + \sigma_G a(u_0)\Bigl( \sigma_G - \frac{\sqrt \alpha}{a(u_0)}    \Bigr)
   +  \alpha \cdot O\Bigl(\frac{a(u_0)^2}{\alpha}\Bigr)\\
&=O(\alpha) \quad \text{ as $\alpha\downarrow 0$. }\qedhere
 \end{align*}
\end{proof}

We may also approach this problem from a different point of view. It
uses the following theorem proved in \cite{interchange}:
\begin{thm}\label{thm:interchange}
Two $./M/1$ queues in series are interchangeable: for any common
input process $A(t)$, the output processes of $./M_1/1 \rightarrow
/M_2/1$ and $./M_2/1\rightarrow /M_1/1$ have the same distribution.
\end{thm}

In terms of the language in our last-passage models, the above
theorem simply implies that if we pick a realization of the
exponential rates $\{\xi_j \}$, the distribution of $T(m,n)$ will
not be affected if we exchange $\xi_{j_1}$ and $\xi_{j_2}$ for $0\le
j_1 < j_2 \le n$. We can start from this idea and reprove Theorem
\ref{exponential}.

\begin{proof}
We first look at the case where the distribution $m$ of the
exponential rates has a finite state space $\{a_1, a_2, \ldots a_K
\}$. Assume $m(\{a_k\}) = x_k, k = 1, 2, \ldots K$ with $\sum_k x_k
=1.$ We now try to approximate $\Psi_G( \alpha,1)$ by calculating
$T(\fl{N\alpha}, N)$ for very large $N$.

Fix an arbitrarily small $\delta>0$. It is standard result that as
$N$ grows, the number of rows that have exponential rates $a_k$ is
in $[\fl{N(1-\delta)x_k}, \fl{N(1+\delta)x_k}]$ with probability
converging to $1$ exponentially fast.

Now we temporarily assume there are $\fl{N(1+\delta)x_k}$ rows with
rate $a_k$ for each $k$. Next we calculate the last-passage time
from origin to $(\fl{N\alpha}, \sum_k \fl{N(1+\delta)x_k}-1)$.
Because exponential variables are positive, this result will be no
smaller than the actual $T(\fl{N\alpha}, N)$ if $N$ is large enough.

From Theorem $\ref{thm:interchange}$  we can rearrange the rows
without changing the distribution of the last-passage time. So
without loss of generality we let the first $\fl{N(1+\delta)x_1}$
rows have rates $a_1$, the next $\fl{N(1+\delta)x_2}$ rows have
rates $a_2$, and so on.

Now we have divided the first quadrant into $K$ horizontal strips,
each of which has underlying distribution $\exp(a_k)$. We select the
optimal path $\pi^*$ from the origin to $(\fl{N\alpha},\sum_k
\fl{N(1+\delta)x_k}-1)$, and record the points at which it exits
each strip: $(M_1-1, \fl{N(1+\delta)x_1}-1), (M_1+M_2-1,
\fl{N(1+\delta)x_1}+\fl{N(1+\delta)x_2}-1), \ldots
\newline (\fl{N\alpha}, \sum_k \fl{N(1+\delta)x_k}-1)$. Note
$M_1+M_2+\dotsm +M_K = \fl{N\alpha}+1$.

Now we fix a small $\e>0$, and define $r_k =
\ce{\frac{M_k}{N\e}}\e$. Then  $ \frac{M_k}{N}\le r_k $, and $\sum_k
r_k \le \frac{1}{N}(\fl{N\alpha}+1) + K\e.$

We look at the intersection of $\pi$ and the $k$-th strip. The
weight of this portion of $\pi$ is no more than the last-passage
time from the origin to $(\fl{Nr_k}, \fl{N(1+\delta)x_k} )$, with
weights i.i.d. from $\exp(a_k)$. Now Theorem 1.6 from \cite{joha}
can be applied and it shows that the maximal weight in this portion
is bounded from above by
$N\bigl[\frac{1}{a_k}\bigl(\sqrt{(1+\delta)x_k} +
\sqrt{r_k}\bigr)^2+\delta\bigr] $ with probability exponentially
close to $1$ in $N$. In fact for each $k$ the rate function  \[
\lim_{N\rightarrow}\frac{1}{N}\log P\Bigl(
 \abs{\frac{1}{N}T_k(\fl{Nr_k}, \fl{N(1+\delta)x_k} ) - \frac{1}{a_k}\bigl(\sqrt{(1+\delta)x_k} +
\sqrt{r_k}\bigr)^2}>\delta \Bigr)\] depends on the random variable
$r_k$. However, $\{r_1, \ldots, r_K\}$ has a finite state space
$\{\e, 2\e, \ldots, \ce{\frac{\alpha}{\e}}\e\}$, so we still have a
deterministic upper bound for the rate function.

So now we connect all strips and see that with probability
converging to $1$ exponentially fast, $T(\fl{N\alpha}, N)$ is
bounded from above by $N\sum_k
\bigl[\frac{1}{a_k}(\sqrt{(1+\delta)x_k} +
\sqrt{r_k})^2+\delta\bigr]$. By Cauchy-Schwarz inequality, \be
\label{expo-upper}
\begin{split}&\sum_k \bigl[\frac{1}{a_k}(\sqrt{(1+\delta)x_k} +
\sqrt{r_k})^2\\ =& \sum_k \frac{1}{a_k}(1+\delta)x_k +
2\sum_k\frac{1}{a_k}\sqrt{(1+\delta)x_kr_k} + \sum_k \frac{r_k}{a_k}\\
\le& \sum_k \frac{1}{a_k}(1+\delta)x_k + 2 \sqrt{\sum_k
\frac{1}{a_k^2}(1+\delta)x_k \cdot \sum_k r_k} + \sum_k \frac{r_k}{c}\\
 \underset{N\rightarrow \infty}{\longrightarrow} & \sum_k \frac{1}{a_k}(1+\delta)x_k  +2\sqrt{\sum_k
\frac{1}{a_k^2}(1+\delta)x_k} \sqrt{\alpha+K\e} +
\frac{1}{c}(\alpha+K\e).
\end{split}
\ee This shows that $\Psi_G(\alpha,1)= \lim_{N\rightarrow
\infty}T(\fl{N\alpha}, N)/N$ is $\mP-a.s.$ bounded from above by
$$\sum_k \frac{1}{a_k}x_k +2\sqrt{\sum_k \frac{1}{a_k^2}x_k}
\sqrt{\alpha} + \frac{1}{c}\alpha$$ since we can make $\delta$ and
$\e$ arbitrarily small.

Then we can run a similar argument to find the lower bound: first
assume there are $\fl{N(1-\delta)x_k}$ rows with rates $a_k$ for
each $k$ and rearrange them to get $K$ strips. We then pick a path
$\pi\in \Pi(\fl{N\alpha},\sum_k \fl{N(1-\delta)x_k}-1)$ in which we
let
\[M_k  = \fl{ N\alpha \cdot\frac{(1-\delta)x_ka^{-2}_k}{\sum_k (1-\delta)x_ka^{-2}_k}      }  \]
be the number of horizontal movements $\pi$ makes in the $k$-th
strip, and in each strip $\pi$ chooses the path that gives the
maximal weight.

With probability converging to 1 exponentially fast in $N$, the
weight of $\pi$ and hence $T(\fl{N\alpha},\sum_k
\fl{N(1-\delta)x_k}-1)$ is at least $N\sum_k
\bigl[\frac{1}{a_k}(\sqrt{(1-\delta)x_k} +
\sqrt{M_k/N})^2-\delta\bigr]$.

From our choice of $M_k$, we get
\begin{align*}&\sum_k \frac{1}{a_k}(\sqrt{(1-\delta)x_k} +
\sqrt{M_k/N})^2\\ \ge & \sum_k \frac{1}{a_k}(1-\delta)x_k +
2\sum_k\frac{1}{a_k}\sqrt{(1-\delta)x_kM_k/N} \\
 \underset{N\rightarrow \infty}{\longrightarrow} & \sum_k \frac{1}{a_k}(1-\delta)x_k  +2\sqrt{\sum_k
\frac{1}{a_k^2}(1-\delta)x_k} \sqrt{\alpha}.
\end{align*}
 Then we can use Borel-Cantelli Lemma, let
$\delta$ go to $0$, and claim the last line above is a lower bound
of $\Psi_G(\alpha,1)$.

It's worth noting that $\sum_k \frac{1}{a_k}x_k$  is actually the
same as $\mu_G$, the average of the means of the exponential
distributions, and $\sum_k \frac{1}{a_k^2}x_k$ should be understood
as $\sigma_G^2$, the average of the variances. Therefore we have
already shown that \be \mu_G + 2 \sigma_G \sqrt{\alpha} \le \Psi_{G}
(\alpha , 1 ) \le \mu_G + 2 \sigma_G \sqrt{\alpha}+\frac{1}{c}\alpha
\ee when the distribution $m$ of exponential rates are supported on
a finite space.

We can use discrete distributions to approximate any general
distribution $m$. Let $\{\xi_j, X(i,j)\}$ be a realization of the
exponential rates and the weights assigned to all lattice points.
Choose an arbitrarily small $\e>0$, couple them with $\{\rho_j,
Y(i,j)\}$ in such a way that $$\rho_j = \frac{1}{\fl{\frac{1}{\xi_j
\e}} \e}$$ and
$$Y(i,j) = \frac{\xi_j}{\rho_j}X(i,j).$$ Then $Y(i,j) \sim
\exp(\rho_j)$ and $\rho_0$ is a random variable with a finite state
space. In addition, we guarantee $Y(i,j)\le X(i,j)$ and $0\le
\frac{1}{\xi_j} - \frac{1}{\rho_j} < \e$.

Then we see \begin{align*} \Psi_G(\alpha,1)  &= \lim_{n\rightarrow
\infty} \frac{1}{n}\mE \max_{\pi \in \Pi(\fl{n\alpha}, n)}
\sum_{z\in \pi}X(z) \\ &\ge \lim_{n\rightarrow \infty}
\frac{1}{n}\mE \max_{\pi \in \Pi(\fl{n\alpha}, n)} \sum_{z\in
\pi}Y(z)\\& \ge \bE \frac{1}{\rho_0}  + 2\sqrt{\bE
\frac{1}{\rho^2_0}}\sqrt{\alpha}.
\end{align*}
Letting $\e$ go to $0$ leads to $ \Psi_G(\alpha,1) \ge \mu_G + 2
\sigma_G \sqrt{\alpha}$ by continuity.

For the other direction, we can define $\rho_j =
\frac{1}{(\ce{\frac{1}{\xi_j \e}}) \e}$ and repeat the same
argument. This gives the upper bound
\[\Psi_G(\alpha,1) \le \mu_G + 2\sigma_G\sqrt\alpha + \frac{1}{c}\alpha\]
and we have proved the theorem.
\end{proof}

\subsection{Estimate of $\Psi_G(1,\alpha)$ }
Next, we switch the two coordinates and estimate $\Psi_G(1,\alpha)$.
 Here we see how the tail of the random mean $\mu_0$ creates different orders of
magnitude for the   $\alpha$-dependent correction term. Qualitative
properties of the limit shape depend on the tail of the distribution
$m$ at $c+$, and transitions occur where the integrals
$\int_{[c,\infty)} (\xi-c)^{-2}\, m(d\xi)$ and $\int_{[c,\infty)}
(\xi-c)^{-1}\, m(d\xi)$ blow up.  ( \cite{TimoKrug} also addressed
this phenomenon.) These same regimes appear in our results below.
For the case $\int_{[c,\infty)} (\xi-c)^{-2}\, m(d\xi) = \infty$ we
make a precise assumption about the tail of the distribution of the
random rate: \be \exists  \; \nu\in[-1,1], \; \kappa>0 \ \
\text{such that} \ \ \lim_{\xi \searrow c}
\frac{m[c,\xi)}{(\xi-c)^{\nu+1}}=\kappa. \label{exp-ass1}\ee The
value $\nu=-1$ means that the bottom rate $c$ has probability
$m\{c\}=\kappa>0$. Values $\nu<-1$ are of course not possible.

\begin{thm}  For the model with exponential distributions with i.i.d.\ random
rates the limit $\Psi_G$ has these asymptotics close to the
$x$-axis.

{\sl Case 1:}   $\int_{[c,\infty)}  (\xi-c)^{-2}\, m(d\xi) <
\infty$. Then there exists $\alpha_0>0$ such that
  \be \Psi_G(1,\alpha) = {c}^{-1} + \alpha \int_{[c,\infty)}
\frac{1}{\xi-c} \,m(d\xi) \quad \text{for $\alpha\in[0,\alpha_0]$.}
\label{expcase1} \ee

{\sl Case 2:} \eqref{exp-ass1}  holds so that, in particular
$\int_{[c,\infty)}  (\xi-c)^{-2}\, m(d\xi) = \infty$. Then as
$\alpha \searrow 0$,
\begin{align}
&\text{if $\nu \in (0,1]$ then} \  \ \Psi_G(1,\alpha) = {c}^{-1}+
\alpha \int_{[c,\infty)} \frac{1}{\xi-c} \,m(d\xi) +o(\alpha)\,;
\label{casenu01} \\
  &\text{if $\nu =0$ then} \qquad   \Psi_G(1,\alpha) =  c^{-1} -\kappa\alpha\log\alpha
  +o(\alpha\log\alpha) \,;  \label{casenu0} \\
 &\text{if $\nu \in [-1,0)$ then} \   \
 \Psi_G(1,\alpha) = c^{-1} + B\alpha^{\frac1{1-\nu}} +o(\alpha^{\frac1{1-\nu}}).   \label{casenu-10}
 \end{align}
\label{exp-thm}\end{thm} In statement \eqref{casenu-10} above
$B=B(c,\kappa,\nu)$ is a constant whose explicit definition is in
equation \eqref{exp18} in the proof below. The extreme case $\nu=-1$
is the one that matches up with Theorem \ref{1,a-thm}.

\begin{proof}Equation \eqref{PsiGg} gives
  \be \Psi_G(1,\alpha) =
\inf\{t \geq 0: tg({\alpha}/{t}) \geq 1\}  = t(\alpha) =  t . \ee
That the   infimum is achieved  can be seen from \eqref{ga-dual}.

Under Case 1 the critical value $u^* = \int_{[c,\infty)} c{ (\xi
-c)^{-1} } \,{m(d\xi)}<\infty$, and also
\[  a'(u^*-)= \biggl\{ \int_{[c,\infty)}  \frac{\xi}{ (\xi-c)^2} \,m(d\xi)\biggr\}^{-1} > 0. \]
By the concavity of $a$ and \eqref{ga-dual}, for $0\le y\le
a'(u^*-)$ we have  $g(y)=-yu^*+c$. Consequently for small enough
$\alpha$
\[   1 = tg({\alpha}/{t}) = -\alpha c \int_{[c,\infty)}  \frac{1}{\xi-c}
m(d\xi) + ct\] and equation \eqref{expcase1} follows.

In Case 2 $a'(0+)> a'(u^*-)= 0$ and hence for small enough
$\alpha>0$  there exists a unique $u_0\in(0,u^*) $ such that
$a'(u_0) = {\alpha}/{t}$.  Set  $a_0 = a(u_0)\in(0,c)$. As
$\alpha\searrow 0$, both $u_0\nearrow u^*$ and $a_0\nearrow c$. We
have the equations
 \[   a'(u_0)^{-1} =\int_{[c,\infty)}
\frac{\xi}{ (\xi-a_0)^2} \,m(d\xi) = \frac{t}{\alpha} \,, \quad  1 =
t g({\alpha}/{t})  = -\alpha u_0 + ta_0\,,  \]
 \begin{equation}\label{case2}t = \frac{1}{a} + \frac{\alpha
u_0}{ a_0} = \frac{1}{a_0} + \alpha \int_{[c,\infty)}
\frac{1}{\xi-a_0} \,m(d\xi)\end{equation}and  \be
\label{excases}\frac{1}{a_0^2} = \alpha \int_{[c,\infty)}
\frac{1}{(\xi-a_0)^2}\,m(d\xi).\ee

Assuming $\eqref{exp-ass1}$,   start  with $\nu \in (-1,0) \cup
(0,1)$. For a small enough $\e>0$ there are constants
$0<\kappa_1<\kappa_2$ such that \be \kappa_1(\xi-c)^{\nu+1}\le
{m[c,\xi)}\le\kappa_2{(\xi-c)^{\nu+1}} \quad\text{ for $\xi
\in[c,c+\e]$ }\label{exp9}\ee and as $\e\searrow 0$ we can take
$\kappa_1,\kappa_2\to \kappa$. First   we estimate $c-a_0$.  Fix
$\e>0$.
\begin{align*}
\frac{1}{\alpha} &= a_0^2 \int_{[c,\infty]} \frac{1}{(\xi-a_0)^2}
\,m(d\xi)
= 2{a_0^2}\int_c^\infty  \frac{m[c,\xi)}{(\xi-a_0)^3}\,d\xi\\
&=2{a_0^2}\int_c^{c+\e}  \frac{m[c,\xi)}{(\xi-a_0)^3}\,d\xi  +
C_1(\e)
\end{align*}
for a quantity $C_1(\e)=O(\e^{-2})$.  The first term above can be
bounded above and below by \eqref{exp9}, and we develop both bounds
together for $\kappa_i$, $i=1,2$,  as \be\begin{aligned}
&2\kappa_i {a_0^2}\int_c^{c+\e}  \frac{(\xi-c)^{\nu+1}}{(\xi-a_0)^3}\,d\xi  +  C_1(\e)  \\
 &= 2\kappa_i a_0^2 \int_{c}^{c+\e} \frac{\bigl[(\xi-a_0)-(c-a_0)\bigr]^{\nu+1}}{(\xi-a_0)^3} d\xi
  +  C_1(\e) \\
 & =  2\kappa_i  a_0^2 \sum_{k=0}^\infty  \binom{\nu+1}{k}  (-1)^k(c - a_0)^k
 \int_{c}^{c+\e} (\xi-a_0)^{\nu - k-2} d\xi +  C_1(\e) \\
& = 2\kappa_i  a_0^2 \sum_{k=0}^\infty  \binom{\nu+1}{k}  (-1)^k(c - a_0)^k \frac{(c - a_0)^{\nu-k-1} - (c+\e-a_0)^{\nu-k-1} }{k-\nu+1} +  C_1(\e) \\
& = 2\kappa_i  a_0^2 A_\nu (c - a_0)^{\nu-1} -  2\kappa_i  a_0^2
\sum_{k=0}^\infty  \binom{\nu+1}{k} \frac{(-1)^k}{k-\nu+1}
(c - a_0)^k (c+\e-a_0)^{\nu-k-1}   +  C_1(\e) \\
& = 2\kappa_i  a_0^2 A_\nu (c - a_0)^{\nu-1} +  C_1(\e).
\end{aligned} \label{1overalpha}\ee
$C_1(\e)$ changed of course in the last equality. In the next to
last equality above we defined
\[   A_\nu= \sum_{k=0}^\infty  \binom{\nu+1}{k} \frac{(-1)^k}{k-\nu+1}. \]
 Rewrite the above development in the form
\[   (c - a_0)^{1-\nu} =  2\kappa c^2 A_\nu\alpha  +
\alpha[ 2A_\nu(\kappa_i  a_0^2-\kappa c^2)+C_1(\e)(c -
a_0)^{1-\nu}].  \] Now choose $\e=\e(\alpha)\searrow 0$ as
$\alpha\searrow 0$ but slowly enough so that $C_1(\e)(c -
a_0)^{1-\nu}\to 0$ as $\alpha\searrow 0$.  Then also $\kappa_i
a_0^2\to\kappa c^2 $ and we can write \be c-a_0 =
B_0\alpha^{\frac1{1-\nu}} +o(\alpha^{\frac1{1-\nu}})
\label{exp12}\ee with a new constant $B_0=(2\kappa c^2
A_\nu)^{\frac1{1-\nu}}$.

Now consider the case   $\nu \in (0,1)$ which also guarantees
$\int_{[c,\infty)} (\xi-c)^{-1} \,m(d\xi)<\infty$.  From
\eqref{case2} and \eqref{exp12} as  $\alpha\searrow 0$
 \be \begin{aligned}&\Psi_G(1,\alpha)  = \frac{1}{a_0}+ \alpha \int_{[c,\infty)} \frac{1}{\xi-a_0}\, m(d\xi)\\
&=\frac{1}{c}+ \alpha \int_{[c,\infty)} \frac{1}{\xi-c} m(d\xi)
+O(\alpha^{\frac1{1-\nu}})  + \alpha  \biggl( \,  \int_{[c,\infty)}
\frac{1}{\xi-a_0}\, m(d\xi)
-   \int_{[c,\infty)} \frac{1}{\xi-c} \,m(d\xi) \biggr) \\
&=\frac{1}{c}+ \alpha \int_{[c,\infty)} \frac{1}{\xi-c} \,m(d\xi)
+o(\alpha).
 \nn\end{aligned}\ee

Next the case $\nu \in (-1,0)$.  The steps are similar to those
above
 so we can afford to be
sketchy.
 \begin{align*}
&\Psi_G(1,\alpha)= \frac{1}{a_0} +  \alpha \int_{[c,\infty)}
\frac{1}{\xi-a_0}
\,m(d\xi)\\
&=\frac1c + \frac{c-a_0}{c^2}   + \frac{(c-a_0)^2}{c^2a_0} +\alpha
\int_c^{c+\e}\frac{m[c,\xi)}{(\xi-a_0)^2}\,d\xi  + \alpha  C_1(\e).
\end{align*}
Again, using \eqref{exp9} and proceeding as in \eqref{1overalpha},
we develop an upper and a lower bound for the quantity above  with
distinct constants $\kappa_i$, $i=1,2$. After bounding $m[c,\xi)$
above and below with $\kappa_i(\xi-c)^{\nu+1}$  in the integral,
write $(\xi-c)^{\nu+1}=((\xi-a_0)-(c-a_0))^{\nu+1}$ and expand in
power series.
 \begin{align*}
&\frac1c + B_0c^{-2}\alpha^{\frac1{1-\nu}}
+o(\alpha^{\frac1{1-\nu}})
+\alpha \kappa_i \int_c^{c+\e}\frac{(\xi-c)^{\nu+1}}{(\xi-a_0)^2}\,d\xi  + \alpha  C_1(\e) \\
&=\frac1c + B_0c^{-2}\alpha^{\frac1{1-\nu}}
+o(\alpha^{\frac1{1-\nu}})
+\alpha \kappa_i  (c-a_0)^\nu \sum_{k=0}^\infty  \binom{\nu+1}{k}  \frac{(-1)^k}{k-\nu} \\
&\qquad \qquad \qquad + \alpha \kappa_i  (c-a_0+\e)^\nu
\sum_{k=0}^\infty  \binom{\nu+1}{k}  \frac{(-1)^k}{\nu-k}
\biggl(\frac{c-a_0}{c-a_0+\e}\biggr)^k
+ \alpha  C_1(\e) \\
&= \frac1c + B\alpha^{\frac1{1-\nu}} +o(\alpha^{\frac1{1-\nu}}) +
A_{\nu, 2}  \alpha (\kappa_i-\kappa)   (c-a_0)^\nu   + \alpha
C_1(\e).
 \end{align*}
In the last equality the next to last term with the
$\sum_{k=0}^\infty $ sum was subsumed in the $ \alpha  C_1(\e)$
term.  Then
 we introduced   new constants
\be A_{\nu, 2} = \sum_{k=0}^\infty  \binom{\nu+1}{k}
\frac{(-1)^k}{k-\nu} \quad\text{and}\quad
  B= B_0c^{-2} + \kappa B_0^\nu A_{\nu, 2}. \label{exp18}\ee
  As before, by letting $\e=\e(\alpha)\searrow 0$  slowly enough as $\alpha\searrow 0$
we can extract  $\Psi_G(1,\alpha)=c^{-1} + B\alpha^{\frac1{1-\nu}}
+o(\alpha^{\frac1{1-\nu}})$ from the above bounds.

It remains to treat the cases $\nu=-1, 0, 1$   where integration of
the type done in \eqref{1overalpha} is elementary. We omit the
details.  \end{proof}

\chapter{Universality results} \label{chap:universality}

\section{Limiting shape near the $y-$axis}
\label{sec:thm:main}

Now we turn to the results on   the   form of  the limit shape at
the boundary for a general process $\{F_j\}_{j\in\bZ_+}$.   As
explained in the introduction, for
 $\Psi(\alpha,1)$  we find a universal form as  $\alpha\searrow 0$.
In addition to the earlier assumptions \eqref{tailass1} and
\eqref{tailass2}, we need similar control of the left tail of the
distributions: \be \int_{-\infty}^0 \bigl(\bE
[F_0(x)]\bigr)^{{1}/{2}} dx< \infty \label{tailass3}\ee and \be
\int_{-\infty}^0 \underset\bP{\esssup}\, F_0(x)\,dx<\infty.
 \label{tailass4}\ee
Let us point out
  that \eqref{tailass1} and \eqref{tailass3} together
guarantee   $\mE |X(z)|^2 < \infty$.   Let  $\mu_j = \mu(F_j)$ and
$\sigma_j^2 = \sigma^2(F_j)$
 denote the mean and variance of
distribution $F_j$. These are random variables under $\P$ with
expectations
  $\mu = \E( \mu_0)$ and   $\sigma^2 =\E(\sigma^2_0)$. Here is our
  main theorem of this chapter:
\begin{thm}  Assume the process $\{F_j\}$ is i.i.d.\ under $\P$, and satisfies 
tail assumptions \eqref{tailass1},   \eqref{tailass2},
\eqref{tailass3} and    \eqref{tailass4}. Then,  as
$\alpha\downarrow 0$, $\Psi(\alpha,1) = \mu + 2\sigma \sqrt{\alpha}+
o(\sqrt{\alpha\,}).$ \label{thm:main}\end{thm}

\subsection{Proof of Theorem \ref{thm:main}}
We start the proof with the first lemma, which enables us to compare
the last-passage time limits with different underlying distribution
sequences. let $\{F_j\}$ and $\{G_j\}$ be ergodic sequences of
distributions defined on a common probability space under
probability measure $\P$.  In a later step of the proof we need to
assume $\{F_j\}$ i.i.d.
 Assume that both processes $\{F_j\}$
and $\{G_j\}$  satisfy the assumptions made in Theorem
\ref{thm:main}. With some abuse of notation we label the time
constants, means, and even random weights associated to the
processes $\{F_j\}$ and $\{G_j\}$ with subscripts $F$ and $G$.  So
for example $\mu_F=\bE(\int x\,dF_0(x))$. The symbolic subscripts
$F$ and $G$ should not be confused with the random distributions
$F_j$ and $G_j$ assigned to the rows of the lattice. We write
$\Psi_{Ber([G(x) -F(x)]_+)}$ for the limit of a Bernoulli model with
weight
   distributions $P(X(i,j)=1)=(G_j(x)-F_j(x))_+=1-P(X(i,j)=0)$ where $x$ is a fixed
   parameter.  An  analogous
convention will be used  for other Bernoulli models  along the way.

\begin{lem}\label{lem:diff}
Assume $\{F_j\}$ and $\{G_j\}$ satisfy \eqref{tailass1},
\eqref{tailass2}, \eqref{tailass3} and \eqref{tailass4}. Then for
$\alpha>0$,
\begin{equation}
\label{eqn:diff}
\begin{split}
&\abs{\Psi_F(\alpha,1) - \Psi_G(\alpha,1) - (\mu_F - \mu_G)}\\
&\qquad \leq 8 \sqrt{\alpha} \int_{-\infty}^{+\infty}
\Bigl(\bE\abs{G_0(x)-F_0(x)}\Bigr)^{1/2} dx
 +  \alpha \int_{-\infty}^{+\infty}
\underset{\bP}{\esssup} |F_0(x) - G_0(x)|\, dx.
\end{split}
\end{equation}
\end{lem}

\begin{proof}
The right-hand side of \eqref{eqn:diff} is finite under the
assumptions on  $\{F_j\}$ and $\{G_j\}$. Couple the $F_j$ and $G_j$
distributed weights
 in a standard way. Let $\{u(z):z =
(i,j) \in \bZ_+^2\} $ be i.i.d.\ Uniform$(0,1)$ random variables.
Set  $X_F(z) = F_j^{-1}(u(z))$, where $F_j^{-1}(u) = \sup\{x :
F_j(x)< u\}$, and similarly  $X_G(z) = G_j^{-1}(u(z))$. Write $\mE$
for expectation over the entire probability space of distributions
and weights.

The following equality will be useful when we relate arbitrary
random variables to Bernoulli variables:
$$X_F(z) - X_G(z) = \int_{-\infty}^{+\infty} \Bigl\{I \bigl(X_G(z)\le x<
X_F(z)\bigr) - I\bigl(X_F(z)\le x< X_G(z)\bigr) \Bigr\} dx.$$

Now we compare $\Psi_F(\alpha,1)$ and $\Psi_G(\alpha,1)$:
\begin{align*}
&\quad\Psi_F(\alpha,1) - \Psi_G(\alpha,1) \\
& = \lim_{n \rightarrow \infty} \frac{1}{n}\mE\max_{\pi \in
\Pi(\lfloor \alpha n \rfloor, n )} \sum_{z \in \pi} X_F(z) - \lim_{n
\rightarrow \infty} \frac{1}{n} \mE\max_{\pi \in \Pi(\lfloor \alpha
n \rfloor,  n  )} \sum_{z \in \pi}
X_G(z) \\
&\leq \varlimsup_{n \rightarrow \infty} \frac{1}{n} \mE \max_{\pi
\in \Pi(\lfloor \alpha n \rfloor, n )} \sum_{z \in \pi}\bigl(
X_F(z)-X_G(z) \bigr)\\
&=\varlimsup_{n \rightarrow \infty}   \frac{1}{n} \mE \max_{\pi \in
\Pi(\lfloor\alpha n \rfloor, n )} \sum_{z \in \pi}
\int_{-\infty}^{+\infty}
\Bigl\{I \bigl(X_G(z)\le x< X_F(z)\bigr) - I\bigl(X_F(z)\le x< X_G(z)\bigr) \Bigr\} dx\\
&\leq \varlimsup_{n \rightarrow \infty}   \frac{1}{n} \mE
\int_{-\infty}^{+\infty} \max_{\pi \in \Pi(\lfloor\alpha n \rfloor,
n )} \sum_{z \in \pi}  \Bigl\{I \bigl(X_G(z)\le x< X_F(z)\bigr) -
I\bigl(X_F(z)\le x< X_G(z)\bigr) \Bigr\} dx.
\end{align*}

We check that Fubini allows us
 to interchange the integral and the expectation.
Since $F$ and $G$ are interchangeable  it  is enough to consider the
first indicator function from above. Let $a$ be an integer
$\ge\alpha$.
\begin{align*}
&\quad\int_{-\infty}^{+\infty} \frac{1}{n}  \mE   \max_{\pi \in
\Pi(\lfloor\alpha n \rfloor, n )} \sum_{z \in \pi}
I \bigl(X_G(z)\le x< X_F(z)\bigr) dx\\
&\leq \int_{-\infty}^{+\infty}   \sup_n \frac{1}{n}\mE \max_{\pi \in
\Pi(an, n )} \sum_{z \in \pi}  I
\bigl(X_G(z)\le x< X_F(z)\bigr) \, dx= \int_{-\infty}^{+\infty}  \Psi_{Ber([G(x) -F(x)]_+)} (a,1)\, dx\\
&\leq  \int_{-\infty}^{+\infty} \Bigl( \bE \abs{G_0(x) -F_0(x)}  +
4\sqrt{a} \bigl( \bE \abs{G_0(x) -F_0(x)} \bigr)^{1/2} + a\,
\underset{\bP}{\esssup}\abs{G_0(x) - F_0(x)}
\Bigr)\, dx\\
& <\infty
\end{align*}
by estimate \eqref{bernoulli3} and
 the finiteness of the right-hand side of \eqref{eqn:diff}.
Continue from the limit above by applying Fubini. Then take the
limit inside the $dx$-integral by dominated convergence,  justified
by the $n$-uniformity in the bound above. Finally apply again the
Bernoulli  estimate \eqref{bernoulli3}.
\begin{align*}
&\quad \Psi_F(\alpha,1) - \Psi_G(\alpha,1) \\
 &\le\varlimsup_{n \rightarrow \infty}
\int_{-\infty}^{+\infty}   \frac{1}{n}
\mE \max_{\pi \in \Pi(\lfloor\alpha n \rfloor, n )} \sum_{z \in \pi}  \Bigl\{I \bigl(X_G(z)\le x< X_F(z)\bigr) - I\bigl(X_F(z)\le x< X_G(z)\bigr) \Bigr\}\, dx\\
&\le  \int_{-\infty}^{+\infty}  \lim_{n \rightarrow \infty}
\frac{1}{n}    \Bigr\{ \mE \max_{\pi \in
\Pi(\lfloor\alpha n \rfloor, n )} \sum_{z \in \pi} I\bigl(X_G(z)\le x< X_F(z)\bigr)\\
&\qquad\qquad\qquad +\; \mE\max_{\pi \in \Pi(\lfloor \alpha n
\rfloor, n )} \sum_{z \in \pi}\bigl(1 - I\bigl(X_F(z)\le x<
X_G(z)\bigr)\bigr) -
\sum_{z \in \pi}1 \Bigr\}\, dx\\
 &= \int_{-\infty}^{+\infty}
\bigl\{\Psi_{Ber([G(x)-F(x)]_+)}(\alpha,1) +
\Psi_{Ber(1-[F(x)-G(x)]_+)}(\alpha,1)
-(1+ \alpha) \bigr\}dx\\
&\leq\int_{-\infty}^{+\infty} \biggl\{\bE\bigl(G_0(x)-F_0(x)\bigr)_+
+1-
\bE\bigl(F_0(x)-G_0(x)\bigr)_+\\
&\qquad+ 4 \sqrt{\alpha} \Bigl(\,\sqrt{\bE\bigl(G_0(x)-F_0(x)\bigr)_+}+ \sqrt{\bE\bigl(F_0(x)-G_0(x)\bigr)_+}\,  \Bigr) \\
&\qquad\qquad  +\alpha\,\Bigl(\, \underset{\bP}{\esssup}
[G_0(x)-F_0(x)]_+ + 1-\underset{\bP}{\essinf}
[F_0(x)-G_0(x)]_+\Bigr)
  -(1 + \alpha)    \biggr\} dx\\
&\leq (\mu_F-\mu_G) +8\sqrt{\alpha} \int_{-\infty}^{+\infty}
\sqrt{\bE|F_0(x)-G_0(x)|} \,dx + \alpha\int_{-\infty}^{+\infty}
\underset{\bP}{\esssup} |G_0(x)-F_0(x)|\,dx.
\end{align*}
Interchanging $F$ and $G$ in the above inequality gives the bound
from the other direction and concludes the proof.
\end{proof}

For a while we make a convenient assumption that the weights are
 uniformly bounded, so for a constant  $M<\infty$, \be  \bP\{ \text{$F_0(-M)=0$ and
$F_0(M)=1$}\} =1, \label{M-bd}\ee then it's easy to see \be
\sigma^2(F_0) \leq M^2 \quad \text{$\bP$-a.s}\label{var-2}\ee
 and the conditions assumed
for Theorem \ref{thm:main} are trivially satisfied by the uniform
boundedness.


Henceforth $r=r(\alpha)$ denotes a positive integer-valued function
such that $r(\alpha) \nearrow \infty$ as $\alpha\searrow 0$. Tile
the lattice with $1\times r$ blocks
 $B_r(x,y) = \{(x,
ry+k): k=0,1,..., r-1\}$
 for $(x,y)\in \bZ_+^2$. A coarse-grained last-passage model
is defined by adding up the weights in each block:
\[  X_r(z) = \sum_{v\in B_r(z)} X(v). \]
The distribution of the new weight $X_r(i,j)$
 on row $j\in\bZ_+$ of the rescaled
lattice is the convolution $  F_{r,\,j}  =  F_{rj}* F_{rj+1}*\dotsm
*F_{rj+r-1}.  $

We repeat  Lemma 4.4 from \cite{Martin2004} with a sketch of the
argument.

\begin{lem}\label{lem:blocks}
 Let $\Psi_F(x,y)$
and $\Psi_{F_r}(x,y)$ be the last-passage time functions obtained by
using  $F_j$ and $F_{r,j}$ as the distributions on the $j$th row,
respectively. If $r\to\infty$ and
 $r\sqrt{\alpha} \rightarrow 0$ as
$\alpha \downarrow 0$, then $$\lim_{\alpha \downarrow 0}
\frac{1}{\sqrt{\alpha}} \bigl\lvert\Psi_F(\alpha,1) -
\frac{1}{r}\Psi_{F_r}(\alpha r,1)\bigr\rvert =0. $$
\end{lem}

\begin{proof}
 Given a path  $\pi \in \Pi(m, nr-1)$, consider all the blocks
that it intersects; this gives a path  $\tilde{\pi} \in \Pi(m,n-1)$
in the rescaled lattice. This path contains almost all the points in
$\pi$, with the possible exception at the end point. For example,
$\pi$ may contain the point $(m, nr-2)$, but $\tilde{\pi}$ does not
if $r\ge 2$. So there are at most $r(m+n-1) - (m+nr-1) =(m-1)(r-1)$
points in $\tilde{\pi}$ but not in $\pi$, and at most $r$ points in
$\pi$ but not in $\tilde{\pi}$.

So $\bigl\lvert (\cup_{z\in \tilde{\pi}} B_r(z))\triangle \pi
\bigr\rvert \leq mr$ when $r$ is large. Then by \eqref{M-bd}
$$\big|\max_{\pi \in \Pi(m, nr-1)} \sum_{z\in \pi} X(z) - \max_{\tilde{\pi} \in \Pi(m, n-1)} \sum_{z\in \tilde{\pi}} X_r(z)\, \big| \leq mrM.$$
Let  $m = \lfloor \alpha nr\rfloor$, divide through
 by $nr$, and take limits, finally we arrive at
 \be \label{error1}
\lim_{\alpha \downarrow 0}\frac{1}{\sqrt\alpha}\abs{\Psi(\alpha,1) -
\frac{1}{r} \Psi(\alpha r, 1)} \le \lim_{\alpha \downarrow 0}
Mr\sqrt\alpha = 0
 \ee
\end{proof}

 Let
$\mu_{r,y}$ and $V_{r,y} $ be the mean and variance of $F_{r,y}$:
$$\mu_{r,y} = \sum_{i=0}^{r-1} \mu_{ry+i},\quad \text{ and}\quad
V_{r,y} = \sum_{i=0}^{r-1} \sigma^2_{ry+i}. $$ Let $\Phi_{r,y}$ be
the distribution function of the normal $\cN(\mu_{r,y}, V_{r,y} )$
distribution,
 and $\wt\Phi_{r,y}$ the
distribution function of  $\cN(r\mu_F, V_{r,y} )$. The
 difference between $\Phi_{r,y}$ and
$\wt\Phi_{r,y}$ is that the latter has a non-random mean. We shall
also find it convenient to use $\{X_j\}$ as a sequence of
independent variables with (random) distributions $X_j\sim F_j$. For
the next lemma we need to assume $\{F_j\}$ an i.i.d.\ sequence under
$\P$.

 As in \cite{Martin2004}, a key step in
the proof is the replacement of the rescaled  weights with Gaussian
weights, which is undertaken in the next lemma.

\begin{lem}\label{lem:N-replace} Assume  $\{F_j\}$
i.i.d.\ under $\P$. If $r\to\infty$ and
 $r\sqrt{\alpha} \rightarrow 0$ as $\alpha \downarrow 0$, then
\be \lim_{\alpha \downarrow 0} \frac{1}{r \sqrt{\alpha}}
|\Psi_{F_r}(\alpha r,1)-\Psi_{\Phi_r}(\alpha r,1)|=0.
\label{N-replace}\ee
\end{lem}

\begin{rem}
The following proof contains many inequalities where we need to use
letters to denote proper constants. For simplicity all constants
that does not depend on $r, \alpha$ and $M$ are subsumed in a single
notation $C$.

Technically we could also subsume $M$ into $C$ and the proof of
Theorem \ref{thm:main} is not affected. However, we write $M$
explicitly to help us prove the next theorem \ref{thm:betterorder}.
\end{rem}

\begin{proof}
We will use \eqref{eqn:diff} for the processes
$\{F_{r,y}\}_{y\in\bZ_+}$ and $\{\Phi_{r,y}\}_{y\in\bZ_+}$ and with
$\alpha$ replaced by $\alpha r$.  On the right-hand side there are
two terms. We will handle the second term first.

To estimate the second integral, we discuss over the value of $x$.
For $x\ge 2rM$, we note that $F_{r,y}(x) =1$ since we assume
\eqref{M-bd}. Now we turn to $\Phi_{r,y}(x)$. We quote Theorem 1.4
from \cite{durrett} and get \be
\begin{split}
\Phi_{r,y}(x) = &1 - \int_{x}^\infty \frac{1}{\sqrt{2\pi V_{r,y}}}
\exp\bigl(-\frac{(s - \mu_{r,y})^2}{2V_{r,y}}\bigr)ds\\
\ge & 1 - \frac{\sqrt{V_{r,y}}}{\sqrt{2\pi}(x-\mu_{r,y})}
\exp\bigl(- \frac{(x-\mu_{r,y})^2}{2V_{r,y}}\bigr)\\
\ge& 1 - \frac{\sqrt{rM^2}}{\sqrt{2\pi} (2rM-rM)}\exp
\bigl(-\frac{(x-rM)^2}{2rM^2} \bigr)\\
=&1-\frac{1}{\sqrt{2\pi r} }\exp \bigl(-\frac{(x-rM)^2}{2rM^2}
\bigr).
 \end{split}\ee
This gives \be \label{taildiff1} \abs{F_{r,y}(x) - \Phi_{r,y}(x)}
\le \frac{1}{\sqrt{2\pi r} }\exp \bigl(-\frac{(x-rM)^2}{2rM^2}
\bigr)\ee for $x > 2rM$. From symmetry a similar inequality \be
\label{taildiff2} \abs{F_{r,y}(x) - \Phi_{r,y}(x)} \le
\frac{1}{\sqrt{2\pi r} }\exp \bigl(-\frac{(x+rM)^2}{2rM^2} \bigr)\ee
 holds for
$x<-2rM$.

Now we are ready to claim \begin{align*} &\alpha
r\int_{-\infty}^{+\infty} \underset{\bP}\esssup \abs{ F_{r,0}(x) -
\Phi_{r,0}(x)}\, dx \\
 \leq & \alpha r \biggl\{\int_{-\infty}^{-2rM} \frac{1}{\sqrt{2\pi
r} }\exp \bigl(-\frac{(x+rM)^2}{2rM^2} \bigr) dx +\int_{-2rM}^{2rM} 1\cdot dx\\
&\qquad   + \int_{2rM}^{+\infty} \frac{1}{\sqrt{2\pi r} }\exp
\bigl(-\frac{(x-rM)^2}{2rM^2} \bigr)
 dx\biggr\}\\
\le & C \alpha r \biggl\{M +4rM \biggr\} \le CM\alpha r^2.
\end{align*}

Next we estimate the first term on the right hand side of
\eqref{eqn:diff}. We will need Theorem 5.17 of \cite{Petrov}, which
states that if independent mean 0 random variables
 $X_1, X_2, X_3, \dotsc$ all have finite third moments, then they satisfy the estimate
\[ \Bigl\lvert P\Bigl\{ B_r^{-1/2}{\sum_{i=1}^r X_i}\le x\Bigr\}- \Phi(x)
\,\Bigr\rvert  \leq A\frac{\sum_{i=1}^r
E|X_i|^3}{B_r^{{3}/{2}}}(1+|x|)^{-3}, \quad x\in\bR,
\]
where  $B_r = \sum_{i=1}^r \Var(X_i) $, $\Phi$ is the standard
normal
 distribution function, and
$A$ is  a constant that is independent of the distribution functions
of $X_1, X_2, \dotsc , X_r$.

Recall that we assume $\{F_j\}$ i.i.d.\ under $\P$, so
$\sigma^2(F_j)$ are i.i.d. random variables. For an arbitrary
$0<\e<1$, say $\e = \frac{1}{2}$, we define $U_r$ as the event
$\sum_{i=0}^{r-1} \sigma_{ry+i}^2 \ge r(1-\e)\bE \sigma^2_0=
\frac{1}{2}r\bE \sigma^2_0$. Then it is standard result that
$\P(U_r)$ converges to $1$ exponentially fast as $r$ goes to
infinity.

With probability $\P(U_r)$ we get
\begin{equation}\label{thm:petrov}\begin{aligned}
| F_{r,y}(x) - \Phi_{r,y}(x)| &\leq A \frac{\sum_{i=0}^{r-1}
E|X_{ry+i} - \mu_{ry+i}|^3}{(\sum_{i=0}^{r-1}
\sigma_{ry+i}^2)^{{3}/{2}}}\bigl(1+V_{r,y}^{-1/2}|x -
\mu_{r,y}|\bigr)^{-3}\\
&\le \frac{CM^3}{\sqrt r} \bigl(1+M^{-1}r^{-1/2}|x -
\mu_{r,y}|\,\bigr)^{-3}
\end{aligned}\end{equation}
where the second inequality used the assumptions
 $P(|X_i | \leq M) =1$ and the property of $U_r$, and $C$ is a proper constant that depends on $A$ and
$\E \sigma^2_0$.

Next we note this trick using Cauchy-Schwarz inequality: for a
probability density $f$ on $\bR$ and a function $H\ge 0$,
\[
\int \sqrt H\,dx = \int f^{1/2} \sqrt{f^{-1}H}\,dx \le \biggl(\,
\int f^{-1} H\,dx\biggr)^{1/2}. \] Then we get \be
\label{eqn:normal1}\begin{split} &\sqrt{\alpha r}
\int_{-\infty}^{+\infty}
\Bigl(\bE\abs{F_{r,0}(x)-\Phi_{r,0}(x)}\Bigr)^{1/2} dx \\
&\le \sqrt{\alpha r} \Bigl\{ \int_{-\infty}^{+\infty}
\frac{1}{f(x)}\bE\abs{F_{r,0}(x)-\Phi_{r,0}(x)}dx \Bigr\}^{1/2}\\
&=\sqrt{\alpha r} \Bigl\{\E \int_{-\infty}^{+\infty}
\frac{1}{f(x)}\abs{F_{r,0}(x)-\Phi_{r,0}(x)}dx \Bigr\}^{1/2}.\\
\end{split}
\ee

For the calculation below take $\delta>0$ and $f(x)=
c_1(1+\abs{x-r\mu_F}^{1+\delta})^{-1}$ for the right constant
$c_1=c_1(\delta)$ to make $\int_{-\infty}^\infty f(x)dx =1$. Again
factors that depend on $\delta$  are subsumed in a constant $C$ in
each of the following steps.

Over the event $U_r$, \be \label{on-Br} \begin{split}
&\int_{-\infty}^{+\infty}
\frac{1}{f(x)}\abs{F_{r,0}(x)-\Phi_{r,0}(x)}dx\\&\le
\frac{CM^3}{\sqrt{r}}\int_{-\infty}^{+\infty} \bigl(
1+\abs{x-r\mu_F}^{1+\delta}\bigr)  \biggl(1+\frac{|x -
\mu_{r,0}|}{M\sqrt r}\,\biggr)^{-3} dx\\
&\text{by a change of variables $x=\mu_{r,0}+yM\sqrt{r}$}\\ &= CM^4
\int_{-\infty}^{+\infty} \frac{1+\abs{\mu_{r,0}-r\mu_F+yM\sqrt{r}
}^{1+\delta}}
 {(1+\abs{y})^3} \, dy  \\
&\le CM^4\bigl( \abs{\mu_{r,0}-r\mu_F}^{1+\delta} + M^{1+\delta}
r^{(1+\delta)/2}\bigr). 
\end{split}\ee

Over the event $U_r^c$ we use \eqref{taildiff1} and
\eqref{taildiff2} to bound the integral
\begin{align*} &\int_{-\infty}^{+\infty}
\frac{1}{f(x)}\abs{F_{r,0}(x)-\Phi_{r,0}(x)}dx\\
 \le &\int_{-\infty }^{-2rM} \bigl(
1+\abs{x-r\mu_F}^{1+\delta}\bigr)\frac{1}{\sqrt{2\pi r} }\exp
\bigl(-\frac{(x+rM)^2}{2rM^2} \bigr) dx+ \int_{-2rM}^{ 2rM } \bigl(
1+\abs{x-r\mu_F}^{1+\delta}\bigr) dx\\
+& \int_{2rM}^\infty \bigl( 1+\abs{x-r\mu_F}^{1+\delta}\bigr)
\frac{1}{\sqrt{2\pi r} }\exp \bigl(-\frac{(x-rM)^2}{2rM^2} \bigr)
dx.
\end{align*}

We use a change of variables $y=\frac{x + rM}{\sqrt{r} M}$ and the
first term above
\begin{align*}  &= \int_{-\infty }^{-\sqrt r} \bigl(
1+\abs{\sqrt r M y - rM -r\mu_F}^{1+\delta}\bigr)\frac{M}{\sqrt{2\pi
} }\exp
\bigl(-\frac{y^2}{2} \bigr) dy \\
&\le C \int_{-\infty }^{-\sqrt r} \bigl( \abs{\sqrt r M
y}^{1+\delta} +\abs{rM +r\mu_F}^{1+\delta}\bigr)\frac{M}{\sqrt{2\pi
} }\exp
\bigl(-\frac{y^2}{2} \bigr) dy\\
&\le C \int_{-\infty }^{-\sqrt r} M^{1+\delta} r ^{\frac{1+\delta}{2}} \abs{y} ^{1+\delta} \frac{M}{\sqrt{2\pi } }\exp(-\frac{y^2}{2}) dy \\
& + C \int_{-\infty }^{-\sqrt r}
r^{1+\delta}M^{1+\delta}\frac{M}{\sqrt{2\pi } }\exp(-\frac{y^2}{2})
dy\\
&\le C( r^{(1+\delta)/2}M^{2+\delta}+r^{1+\delta}M^{2+\delta} ) \le
Cr^{1+\delta}M^{2+\delta}
\end{align*}
The third term follow the same upper bounds.

The second term is simply bounded by $C\cdot rM\cdot (rM)^{1+\delta}
= CM^{2+\delta}r^{2+\delta}. $ Then \be\int_{-\infty}^{+\infty}
\frac{1}{f(x)}\abs{F_{r,0}(x)-\Phi_{r,0}(x)}dx \le
CM^{2+\delta}r^{2+\delta}.\ee

Continue from \eqref{eqn:normal1}, and keep in mind that $\P(U_r^c)$
decays exponentially as $r$ grows, we have
 \be \begin{split} &\sqrt{\alpha r}
\int_{-\infty}^{+\infty}
\Bigl(\bE\abs{F_{r,0}(x)-\Phi_{r,0}(x)}\Bigr)^{1/2} dx \\
&\le \sqrt{\alpha r} \Bigl\{\E \int_{-\infty}^{+\infty}
\frac{1}{f(x)}\abs{F_{r,0}(x)-\Phi_{r,0}(x)}dx \Bigr\}^{1/2}\\
&\le C\sqrt{\alpha r} \Bigl\{M^4 \bE
\abs{\mu_{r,0}-r\mu_F}^{1+\delta}+ M^{5+\delta} r^{(1+\delta)/2} +
M^{2+\delta}r^{2+\delta}\P(U_r^c) \Bigr\}^{1/2}\\ &\le
CM^{(5+\delta)/2}\alpha^{1/2}r^{(3+\delta)/4}.
\end{split}
\ee

Here we used the fact that $\mu_{r,0}-\mu_F$ is a sum of independent
bounded mean-zero variables, so \begin{align*} \bE
\abs{\mu_{r,0}-r\mu_F}^{1+\delta} &\le \Bigl(\bE
\abs{\mu_{r,0}-r\mu_F}^2\Bigr)^{(1+\delta)/2}\\ =&\bigl[ r \bE
(\mu_0-\mu_F)^2\bigr]^{(1+\delta)/2} =
CM^{1+\delta}r^{(1+\delta)/2}.\end{align*}

 To summarize, with these estimates and \eqref{eqn:diff} we have
\be  \label{N-replace:conclusion}\begin{split}  &\frac{1}{r
\sqrt{\alpha}} |\Psi_{F_r}(\alpha r,1)-\Psi_{\Phi_r}(\alpha r,1)|\\
\le&   \frac{C}{r \sqrt{\alpha}} (M^{(5+\delta)/2}\alpha^{1/2}
r^{(3+\delta)/4} + M\alpha r^{2}) \\=& C(M^{(5+\delta)/2}
r^{(-1+\delta)/4} + Mr\sqrt\alpha). \end{split}\ee By choosing
$\delta <1$, then assumptions $r\to\infty$ and $r\sqrt\alpha\to 0$
make this vanish as $\alpha\to 0$. The proof is
completed.\end{proof}

\begin{rem} The proof presented above distinguished the two events $U_r$
and $U_r^c$ because \eqref{thm:petrov} only works when
$V_{r,y}=\sum_{i=0}^{r-1}\sigma^2_j$ can be bounded away from zero.
Therefore on $U_r^c$ we used a different approach.

We can actually find an alternative proof of this lemma. In addition
to \eqref{M-bd} let's also assume that variances are uniformly
bounded away from zero, i.e. for a constant $0<c_0<\infty$, \be
\bP\{ \sigma^2(F_0)\ge c_0\}=1. \label{var-bd}\ee
 Note that then $ c_0\le \sigma^2(F_0) \leq M^2$. A direct
 consequence of this is that now \eqref{thm:petrov}
and \eqref{on-Br} holds $\P-a.s.$, so \[\sqrt{\alpha r}
\int_{-\infty}^{+\infty}
\Bigl(\bE\abs{F_{r,0}(x)-\Phi_{r,0}(x)}\Bigr)^{1/2} dx\le
C\alpha^{1/2}r^{(3+\delta)/4}.\]

Note that in equations above and below we subsume $M$ into $C$ for
simplicity since we will not come back to them any more. For the
second term on the right in \eqref{eqn:diff},
\begin{align*}
&\alpha r\int_{-\infty}^{+\infty} \underset{\bP}\esssup \abs{
F_{r,0}(x) - \Phi_{r,0}(x)}\, dx \leq C\alpha r^{1/2}
\int_{-\infty}^{+\infty} \underset{\bP}\esssup \biggl(1+\frac{|x -
\mu_{r,y}|}{M\sqrt r}\,\biggr)^{-3}   dx\\
&\qquad \leq  C\alpha r^{1/2} \biggl\{\int_{-\infty}^{-rM} \biggl(
1+ \frac{-rM-x}{M\sqrt{r}}\biggr)^{-3} dx + \int_{-rM}^{rM}  dx +
\int_{rM}^{+\infty} \biggl(1+ \frac{x-rM}{M\sqrt{r}}\biggr)^{-3}
 dx\biggr\}\\
& \qquad = C\alpha r^{1/2} \biggl\{ M\sqrt{r} \int_1^{\infty} u^{-3}
du + 2rM +   M\sqrt{r} \int_1^{\infty} u^{-3} du \biggr\}\\
& \qquad = C\alpha r^{1/2}( M\sqrt{r} + 2rM )\le C\alpha r^{3/2}.
\end{align*}

Therefore we see \eqref{N-replace:conclusion} still holds and Lemma
\ref{lem:N-replace} is proved under \eqref{var-bd}. However, we
eventually we have to show Theorem \ref{thm:main} without
\eqref{var-bd}, so then we try to lift this assumption.

For $\e>0$, let $\{W(z)\}$ be i.i.d weights with distribution $H$
defined by $P(W(z)=\pm\e)=1/2$.  Let $\wt F_j=F_j*H$ be the
distribution of the weight $ \wt X(i,j)=X(i,j)+W(i,j)$. Let $\Psi_H$
and $\Psi_{\wt F}$ be the time constants of the last-passage models
with weights $\{W(z)\}$  and $\{\wt X(z)\}$, respectively. The
Bernoulli bound \eqref{bernoulli3} gives the estimate $
\Psi_{H}(x,y)\le 4\e\sqrt{xy}$. The corresponding last-passage times
satisfy
\[  T_{\wt F}(z) -T_H(z)\le T_F(z) \le   T_{\wt F}(z) +\hat{T}_H(z) \]
where $\hat{T}_H(z)$ uses the weights $-W(z)$. In the limit \be
\Psi_{\wt F}(\alpha, 1) -  4\e\sqrt{\alpha}
 \le \Psi_F(\alpha,1) \le
\Psi_{\wt F}(\alpha, 1) +  4\e\sqrt{\alpha}. \label{eqn:addH}\ee
Since $\sigma^2(\wt F_j)=\sigma^2(F_j) + \e^2$, $\{\wt F_j\}$
satisfies \eqref{var-bd}. Once $\{\wt F_j\}_{j\in\bZ_+}$ satisfies
\eqref{aux30}, then so does $\{F_j\}_{j\in \bZ_+}$, because
$\mu_{\wt F}=\mu_F$ and $\e>0$ can be arbitrarily small.
\end{rem}

After the discussion of Lemma \ref{lem:N-replace} we make a further
approximation that puts us in the situation where
 all sites have  normal variables with the
same mean.

\begin{lem}
Let $\Psi_{\Phi_r}$ and $\Psi_{\wt\Phi_r}$ be defined as before, and
again $r\sqrt{\alpha} \rightarrow 0$ as $\alpha \rightarrow 0$. Then
 $$\lim_{\alpha \downarrow 0}
\frac{1}{r\sqrt{\alpha}}|\Psi_{\Phi_r} (\alpha r, 1 ) -
\Psi_{\wt\Phi_r} (\alpha r, 1 )|=0. $$
\end{lem}

\begin{proof}
For  $z=(i,j) \in \bZ^2_+$, let $X^{(r)}(z)$ have distribution
$\Phi_{r,j}$ so that $\wt{X}^{(r)}(z) = X^{(r)}(z) - \mu_{r,j} +
r\mu_F$
 has distribution
$\wt\Phi_{r,j}$.  Now estimate:
\begin{align*}
\Psi_{\wt\Phi_r} (\alpha r, 1 )& = \lim_{n \rightarrow \infty}
\frac{1}{n} \max_{\pi \in \Pi(\lfloor \alpha nr \rfloor,  n )}
\sum_{z \in \pi} \wt{X}^{(r)}(z)\\
& \leq \lim_{n \rightarrow \infty} \frac{1}{n} \max_{\pi \in
\Pi(\lfloor \alpha nr \rfloor,  n )} \sum_{z \in \pi} X^{(r)}(z) +
\lim_{n \rightarrow \infty} \frac{1}{n} \max_{\pi \in \Pi(\lfloor
\alpha nr \rfloor,  n )} \sum_{z\in \pi}\bigl( - \mu_{r,j} +
r\mu_F\bigr)\\
&\leq \Psi_{\Phi_r} (\alpha r, 1 ) + \lim_{n \rightarrow \infty}
\frac{1}{n} \sum_{j=0}^n\bigl( - \mu_{r,j} + r\mu_F\bigr)+
\lim_{n \rightarrow \infty} \frac{1}{n} 2Mr\cdot \lfloor \alpha nr\rfloor\\
&= \Psi_{\Phi_r} (\alpha r, 1 ) + 2M\alpha r^2.
\end{align*}
Note that in the second to last step we used the fact that when
$(i,j)$ is assigned with $- \mu_{r,j} + r\mu_F$, every admissible
path in $\Pi[\lfloor \alpha nr \rfloor,  n )$ contains at least one
of each $- \mu_{r,j} + r\mu_F$ for every $j = 0,1,\ldots, n$. Their
average converges to $0$ by the law of large numbers. There are also
$\lfloor \alpha nr \rfloor$ additional points all bounded from above
by $2M$, which contribute to the third term.

The opposite bound $\Psi_{\wt\Phi_r} (\alpha r, 1 )\geq
\Psi_{\Phi_r} (\alpha r, 1 )  - 2M\alpha r^2 $  comes similarly. So
\be \label{error2} \frac{1}{r\sqrt\alpha}\abs{\Psi_{\wt\Phi_r}
(\alpha r, 1 )- \Psi_{\Phi_r} (\alpha r, 1 )} \le 2Mr \sqrt\alpha
\ee and the lemma follows.
\end{proof}


Let us separate the mean by letting
 $\overline{\Phi}_{r,y}$ denote
 the  $N(0, \sum_{i=0}^{r-1} \sigma^2_{ry+i} )$
 distribution
function. Since the last-passage functions of the normal
distributions satisfy \be \label{error3} \Psi_{\wt\Phi_r} (\alpha r,
1 ) = r\mu_F(1+\alpha r)+\Psi_{\overline{\Phi}^{(r)}} (\alpha r, 1
),\ee
  we can  summarize the effect of the last three lemmas as
follows.

\begin{lem}   Assume  $\{F_j\}$
i.i.d.\ under $\P$, and assume $r=r(\alpha)$ satisfies
 $r\to\infty$ and
 $r\sqrt{\alpha} \rightarrow 0$ as $\alpha \downarrow 0$.
 Under assumptions \eqref{M-bd}
\begin{equation}\label{eqn:main1}
\lim_{\alpha \downarrow 0} \frac{1}{\sqrt{\alpha}}|\Psi_{F} (\alpha
, 1 ) - \mu_F- \frac{1}{r}\Psi_{\overline{\Phi}^{(r)}} (\alpha r, 1
)| = 0.
\end{equation}
\label{summ-lm}\end{lem}

In order to deduce a  limit from  \eqref{eqn:main1}   we utilize the
explicitly computable   case   of exponential distributions from
\cite{TimoKrug}, and use the results proved in Chapter
\ref{sec:exp}.  We need to match up the random variances of the
exponentials with the variances $\sigma_j^2$ of the sequence
$\{F_j\}$.  Thus, given the i.i.d.\ sequence of quenched variances
$\sigma_j^2=\sigma^2(F_j)$ that we have worked with up to now under
condition \eqref{var-2}, let $\xi_j=1/\sigma_j$ and
$G_j(x)=1-e^{-\xi_jx}$   the rate $\xi_j$ exponential distribution.
Then $\{\xi_j\}_{j\in\bZ_+}$ is  an i.i.d.\ sequence of random
variables  $\xi_j>0$ with distribution $m$. Since we assume
\eqref{M-bd}, the sequence $\{\xi_j\}_{j\in\bZ_+}$ is bounded away
from zero. We can assume $c$ is the exact lower bound: $m[c,c+\e)>0$
for each $\e>0$. $G_j$ has mean and variance $\mu(G_j)=\xi_j^{-1}$
and $\sigma^2(G_j)=\xi_j^{-2}=\sigma_j^2$.

Assumptions  \eqref{tailass1} and \eqref{tailass2} are easily
checked, and so the last-passage function   $\Psi_{G}$ is
well-defined. We would like to apply Lemma \ref{summ-lm} to this
exponential model, but obviously assumption \eqref{M-bd} is not
satisfied. To get around this difficulty we  do the following
approximation which leaves the quenched means and variances intact.
We learned this trick from \cite{Martin2004}.

Let  $Y_j$ denote a $G_j$-distributed
 random variable.  For a fixed $\tau>0$, let $$m_j =
E(Y_j| Y_j >\tau ) \quad\text{and}\quad  w_j = E(Y_j^2| Y_j >\tau
).$$ The quantities
\[\ppb_j = \frac{(m_j-\tau )^2}{(m_j-\tau )^2 + w_j - m_j^2}\quad \textrm{and} \quad u_j = \frac{w_j -
\tau ^2}{ m_j -\tau }-\tau \] satisfy the equations
\[(1-\ppb_j)\tau  + \ppb_j u_j = m_j \quad \textrm{and} \quad  (1-\ppb_j)\tau ^2 + \ppb_j u_j^2 = w_j.\]
Then  $0\leq \ppb_j \leq 1 $, $u_j \geq \tau $ and $ w_j \geq
\tau^2$. Define distribution functions
\begin{equation}\label{deftildeF}
\wt{G}_j(x) =
\begin{cases}
 G_j(x) & 0\leq x < \tau \\
 1-\ppb_j[1-G_j(\tau )] & \tau \leq x < u_j\\
 1 &  x\geq u_j.
\end{cases}
\end{equation}
$\wt Y_j \sim \wt G_j$ satisfies  $EY_j = E \wt{Y}_j$ and $EY_j^2 =
E \wt{Y}_j^2$.
 Moreover, for any fixed  $\tau >0$, $$u_j= \frac{E(Y_j^2|
Y_j>\tau )-\tau^2}{E(Y_j| Y_j>\tau )-\tau }-\tau  = \frac{2}{\ppa_j}
+ \tau  \leq \frac{2}{c} + \tau ,$$ so   $\{\wt G_j\}$ are all
supported on  the   non-random bounded interval $[0, 2/c+\tau ]$.
Consequently Lemma \ref{summ-lm} applies to $\wt G$. We can draw the
same conclusion for $G$ once we have the next estimate:

\begin{lem}\label{lem:tobdd}
Given $\e>0$, we can select $\tau $ large enough and define
$\wt{G}_j$ as in \eqref{deftildeF} so that
$$\varlimsup_{\alpha \downarrow 0} \frac{1}{ \sqrt{\alpha}}
|\Psi_{G}(\alpha ,1)-\Psi_{\wt{G}}(\alpha ,1)|< \e.$$
\end{lem}

\begin{proof}
This comes from an application of Lemma \ref{lem:diff}. $G_j=\wt
G_j$ on $(-\infty, \tau )$ and
 $1 - \wt{G}_j \leq 1 - G_j$ on all of $\bR$.
\be\label{error4} \begin{split}  &\abs{\Psi_G(\alpha,1) - \Psi_{\wt{G}}(\alpha,1)}\\
&\qquad \leq 8 \sqrt{\alpha} \int_{-\infty}^{+\infty}
\Bigl(\bE\abs{G_0(x)-\wt{G}_0(x)}\Bigr)^{1/2} dx
 +  \alpha \int_{-\infty}^{+\infty}
\underset{\bP}{\esssup} |G_0(x) - \wt{G}_0(x)|\, dx\\
&\qquad \leq 8 \sqrt{\alpha} \int_{\tau}^{+\infty}
\Bigl(\bE\abs{1-G_0(x)} + \bE\abs{1-\wt{G}_0(x)}\Bigr)^{1/2} dx\\
&\qquad \qquad +  \alpha \int_{\tau}^{+\infty}\bigl(
\underset{\bP}{\esssup} |1-G_0(x)|+ \underset{\bP}{\esssup}
|1-\wt{G}_0(x)|\bigr) dx\\
&\qquad \leq 8 \sqrt{2\alpha} \int_{\tau}^{+\infty}
\Bigl(\bE\abs{1-G_0(x)}\Bigr)^{1/2} dx+ 2\alpha
\int_{\tau}^{+\infty} \underset{\bP}{\esssup} |1-G_0(x)| dx\\
&\qquad \le 8 \sqrt{2\alpha} \int_{\tau}^{+\infty}
\exp(-\frac{cx}{2})\, dx + 2\alpha\int_{\tau}^{+\infty} \exp(-cx)\,
dx\\
& \qquad = \frac{16\sqrt{2\alpha} }{c}\exp(-\frac{c\tau}{2}) +
\frac{2\alpha}{c} \exp(-c\tau).
\end{split}\ee

Now we see $\varlimsup_{\alpha \downarrow 0} \frac{1}{
\sqrt{\alpha}} |\Psi_{G}(\alpha ,1)-\Psi_{\wt{G}}(\alpha ,1)|$ can
be made arbitrarily small by choosing $\tau $ large.
\end{proof}

So Lemma \ref{summ-lm} and Lemma \ref{lem:tobdd} together show that
\begin{equation}\label{eqn:main2}
\lim_{\alpha \downarrow 0} \frac{1}{\sqrt{\alpha}}|\Psi_{G} (\alpha
, 1 ) - \bE \sigma_{0}- \frac{1}{r}\Psi_{\overline{\Phi}_{r}}
(\alpha r, 1 )| = 0.
\end{equation}

It remains to perform
 an explicit calculation on $\Psi_G(\alpha,1)$.  As before,
utilize the notation $\mu_G =\bE \xi_0^{-1}$ and  $\sigma_G^2
 = \bE\xi_0^{-2} $.  In Theorem \ref{exponential} We have already computed that
$$\lim_{\alpha \downarrow 0} \frac{1}{\sqrt{\alpha}}|\Psi_{G}
(\alpha , 1 ) - \mu_G - 2 \sigma_G \sqrt{\alpha}| = 0.$$ This result
combined with \eqref{eqn:main2} gives
\[\lim_{\alpha \downarrow 0} \frac{1}{\sqrt{\alpha}}\bigl\lvert
\frac{1}{r}\Psi_{\overline{\Phi}_{r}} (\alpha r, 1 ) - 2\sigma_G
\sqrt{\alpha}\bigr\rvert=0.
\]
Substitute this back into  \eqref{eqn:main1} and recall  that
$\sigma_F = \sigma_G$. The conclusion we get is
 \be \lim_{\alpha \downarrow 0}
\frac{1}{\sqrt{\alpha}}|\Psi_{F} (\alpha , 1 ) - \mu_F- 2 \sigma_F
\sqrt{\alpha}| = 0. \label{aux30}\ee

So far we  have proved Theorem \ref{thm:main} under the  assumption
\eqref{M-bd}. As the last item of the proof of Theorem
\ref{thm:main} we remove this uniform boundedness assumption.
Suppose $\{F_j\}$
 satisfy  the conditions required for
  Theorem \ref{thm:main}, but there is no common bounded support.
For a fixed $M>0$ define the truncated distributions
\[F_{j,M}(x) =
\begin{cases} 1 & x \geq M\\
F_j(x) & -M \leq x <M\\
0 & x< -M.  \end{cases}\] Let $\mu_{M}$, $\sigma^2_{M}$ and
$\Psi_{F_{M}}(x,y)$  be quantities associated to  $\{F_{j,M}\}$.

From \eqref{eqn:diff} and the conditions assumed in Theorem
\ref{thm:main},  \be \label{eqn:removeM}  \begin{split}
 &\frac{1}{\sqrt{\alpha}}\abs{\Psi_F(\alpha,1) - \Psi_{F_{M}}(\alpha,1) - (\mu - \mu_M)}\\
& \leq   8 \int_{-\infty}^{+\infty}
\Bigl(\bE\abs{F_0(x)-F_{0,M}(x)}\Bigr)^{1/2} dx
 +
 \sqrt{\alpha} \int_{-\infty}^{+\infty}
\underset{\bP}{\esssup} |F_0(x) - F_{0,M}(x)|\, dx\\
&=8  \Bigl [ \int_{-\infty}^{-M} \Bigl(\bE\abs{F_0(x)}\Bigr)^{1/2}
dx + \int_{M}^{\infty}
\Bigl(\bE\abs{1-F_0(x)}\Bigr)^{1/2} dx\Bigr]\\
&\quad +  \sqrt{\alpha} \Bigl[\int_{-\infty}^{-M}
\underset{\bP}{\esssup} |F_0(x)|\, dx+ \int_{M}^{+\infty}
\underset{\bP}{\esssup} |1-F_0(x) |\, dx\Bigr]  \le \e.
\end{split} \ee
The last inequality comes from choosing $M$ large enough, and is
valid for all   $\alpha \leq 1$. Since $\bE (EX^2(0,0))   < \infty$,
dominated convergence gives
  $\sigma_M \rightarrow \sigma$  and so
 we can pick $M$ so that
$\abs{\sigma - \sigma_M} <\e$.  Now
\begin{align*}
\frac{1}{\sqrt{\alpha}}\abs{\Psi_{F} (\alpha , 1 ) - \mu- 2 \sigma
\sqrt{\alpha}}
 \leq   \frac{1}{\sqrt{\alpha}}\abs{\Psi_{F_M}
(\alpha , 1 ) - \mu_M- 2 \sigma_{M} \sqrt{\alpha}}    + 2\e.
 \end{align*}
 Since $\e$ is arbitrary  and limit \eqref{aux30} holds for
$\{F_{j,M}\}$,  we  get the conclusion for the sequence $\{F_j\}$.
This concludes the proof of Theorem \ref{thm:main}.

\subsection{An improvement on the error $o(\sqrt\alpha)$}
The error term $o(\sqrt\alpha)$ can still be improved. As was shown
in Theorem \ref{exponential}, in the exponential model $O(\alpha)$
is a more accurate estimate.  In this section we will take a closer
look at the approximations we used along the proof of Theorem
\ref{thm:main} and analyze the order of $\alpha$.

\begin{thm}\label{thm:betterorder}
 Assume the process $\{F_j\}$ satisfies the assumptions in Theorem
 \ref{thm:main}, i.e. it is i.i.d.\ under $\P$, and satisfies
tail assumptions \eqref{tailass1},   \eqref{tailass2},
\eqref{tailass3} and \eqref{tailass4}.

In addition, assume  the uniform bound \eqref{M-bd}, then for any
$\e>0$, as $\alpha \searrow 0$ \be \Psi_F(\alpha,1) = \mu_F +
2\sigma_F\sqrt\alpha + o(\alpha^{\frac{3}{5}-\e}).\ee
\end{thm}

\begin{proof}
 We give a list of all approximations in the previous
subsection: when $\{F_j\}_{j\in \bZ_+}$ satisfies \eqref{M-bd},
\begin{align*}
\eqref{error1}: &\frac{1}{\sqrt\alpha}\abs{\Psi(\alpha,1) -
\frac{1}{r} \Psi(\alpha r, 1)} \le Mr\sqrt\alpha,\\
\eqref{N-replace:conclusion}:& \frac{1}{r \sqrt{\alpha}}
|\Psi_{F_r}(\alpha r,1)-\Psi_{\Phi_r}(\alpha r,1)|  \le
C(M^{(5+\delta)/2}
r^{(-1+\delta)/4} + Mr\sqrt\alpha),\\
\eqref{error2}:&\frac{1}{r\sqrt\alpha}\abs{\Psi_{\wt\Phi_r} (\alpha
r, 1 )- \Psi_{\Phi_r} (\alpha r, 1 )} \le 2Mr \sqrt\alpha,\text{ and}\\
\eqref{error3}:& \Psi_{\wt\Phi_r} (\alpha r, 1 ) = r\mu_F(1+\alpha
r)+\Psi_{\overline{\Phi}^{(r)}} (\alpha r, 1 ).\\
\end{align*}

To sum up, \eqref{eqn:main1} can be rewritten as  \be
\label{error-sum} |\Psi_{F} (\alpha , 1 ) - \mu_F-
\frac{1}{r}\Psi_{\overline{\Phi}^{(r)}} (\alpha r, 1 )| \le
C(M^{(5+\delta)/2}r^{(-1+\delta)/4}\sqrt\alpha + Mr\alpha) \ee for a
proper constant $C.$

Now we recall the approximation for the exponential model. For
exponential distributions $\{G_j\}_{j\in\bZ_+}$, if we define $\{\wt
G_j\}_{j\in \bZ_+}$ uniformly bounded by $M$ as in
\eqref{deftildeF}, then \eqref{error4} gives:
\[|\Psi_{G}(\alpha ,1)-\Psi_{\wt{G}}(\alpha ,1)| \le C\exp(-cM) \sqrt \alpha\]
for a proper constant $C$.

The above equation, together with \eqref{error-sum} applied to
$\{\wt{G_j}\}_{j\in \bZ_+}$, implies \be\label{error-sum2}
\abs{\Psi_G(\alpha, 1)- \mu_G -
\frac{1}{r}\Psi_{\overline{\Phi}^{(r)}} (\alpha r, 1 ) }\le
C\bigl(\exp(-cM) \sqrt \alpha + M^{(5+\delta)/2}r^{(-1+\delta)/4}
\sqrt\alpha + Mr\alpha\bigr).\ee Here $C$ does not depend on $M, r,$
and $\alpha$.

In the above equation the left-hand side is independent of $M$, so
we make both $M$ and $r$ functions of $\alpha$. Let $r(\alpha) =
\alpha^{-\frac{2}{5-\delta}} $ and $M (\alpha) = - \frac{0.1}{c}\ln
\alpha$. Make $\delta$ small enough, then the right hand side of
\eqref{error-sum2} is bounded by $C (- \frac{0.1}{c}\ln
\alpha)^{5+\delta} \alpha^{\frac{3-\delta}{5-\delta}} =
o(\alpha^{\frac{3}{5}-\e})$ as $\alpha \searrow 0$.

Recall Theorem \ref{exponential}, we get \be
\frac{1}{r}\Psi_{\overline{\Phi}^{(r)}} (\alpha r, 1 ) = 2\sigma_G
\sqrt\alpha + o(\alpha^{\frac{3}{5}-\e}). \ee

Now let us come back to $\Psi(\alpha,1)$. If $\{F_j\}_{j\in \bZ_+}$
are uniformly bounded, then $M$ is a fixed constant and the right
hand side of \eqref{error-sum} is just $o(\alpha^{\frac{3}{5}-\e})$,
and this shows that $\forall \e>0$,
\[ \Psi(\alpha,1) = \mu_F + 2\sigma_F\sqrt\alpha  + o(\alpha^{\frac{3}{5}-\e}).\]
\end{proof}

\begin{rem}
 The proof did not treat the case when $\{F_j\}_{j\in\bZ_+}$ is not uniformly
 bounded. The difficulty is that \eqref{eqn:removeM} does not give
 precise computability on how the right-hand side would change according
to $M$. We surely need additional assumptions in this case.
\end{rem}

\section{Estimates for limiting shape near $x$-axis}
\label{sec:1,a-thm}

We turn to the case  $\Psi(1,\alpha)$. As we have seen in
\ref{exp-thm},  the results will be
  qualitatively  different from  Theorem \ref{thm:main}.
The leading term will be  the essential supremum of the mean instead
of the averaged mean and we will see different orders for the first
$\alpha$-dependent correction term. Universality results are desired
but much more difficult to achieve. In this section we will see some
estimates of $\Psi(1,\alpha)$ under various conditions.

We first present a result which gives an upper bound of
$\Psi(1,\alpha)$  in a very general setting. We will use $F$ as
superscripts or subscripts when we want to stress the dependence (of
a probability, or expectation, etc.) on a distribution function $F$
that is in the state space of $\P$. Again, $\mu(F)$ is the mean of
the distribution $F$ and $\sigma^2(F)$ is the variance.

\begin{thm}
Assume $\{F_j\}_{j\in\bZ_+}$ is i.i.d.\ and there exists constants
$t_0
>0$ and $K_0 < \infty$ such that
\[\esssup_{\bP}E^F e^{t_0\abs{X-\mu(F)}} < K_0.\] Let $\mu^* =
\esssup_{\bP}\mu(F)$, then\be \Psi(1,\alpha)= \mu^* + O(\sqrt{\alpha
\log \frac{1}{\alpha}}). \ee
\end{thm}
\begin{proof}
We will use $\omega$ to denote a realization of
$\{F_j\}_{j\in\bZ_+}$, and use $P^\omega$ and $E^\omega$ as the
corresponding quenched probability and expectation. Recall the
definition
$$\Psi(1,\alpha) = \lim_{n \rightarrow \infty} \frac{1}{n} \max_{\pi
\in \Pi(n,\fl{n\alpha} )} \sum_{z \in \pi} X(z).$$

For any specific path $\pi$, we write $T(\pi) = \sum_{z \in \pi}
X(z)$. We also use $m_j$ as the leftmost site traveled by $\pi$ in
the $j$-th row. Then we derive the following estimate: for any
positive $u$ and $t$,

\be \label{generalupper1}
\begin{split}
&\mP(T(\pi) \geq nu) = \bE P^\omega (T(\pi) \geq nu) \leq  \bE(e^{-ntu} E^\omega(e^{tT(\pi)}))\\
&= e^{-ntu} \bE \bigl[ \exp ( \sum_{z\in \pi} \log E^\omega
(e^{tX(z)}))\bigr]\\
&= e^{-ntu} \prod_{j=0}^ {\fl{n\alpha}} \bE \exp\bigl[ \sum_{i= m_j}
^{m_{j+1}} \log E^\omega (e^{tX(i,j)})   \bigr]\\
& = \exp \Bigl\{-n t u  + \sum_{j=0} ^ {\fl{n\alpha}} \log \bE
\bigl[ \exp\bigl( (1+m_{j+1} - m_j) \log E^{F_j} e^{tX}\bigr) \bigr]
\Bigr\}.
\end{split} \ee

We need the following inequality in order to proceed: if we have a
random variable $Y$ and positive real numbers $t_1, ..., t_m$ with
$\sum_j t_j =t$, then \be \sum_j \log E e^{t_iY} \leq \sum_j \log [
E(e^{tY})]^{ \frac{t_j}{t} } = \sum_j \frac{t_j}{t} \log [
E(e^{tY})]=\log [ E(e^{tY})].\ee

Continuing from \eqref{generalupper1}, as $n \nearrow \infty$ \be
\label{generalupper2}
\begin{split}
&\mP(T(\pi) \geq nu) \leq \exp \Bigl\{-n t u +
 \log \bE \exp\bigl(n(1+\alpha) \log E^F e^{tX}\bigr) \Bigr\}\\
=& \exp \Bigl\{-n[ t u -
 (1+\alpha) \frac{1}{n(1+\alpha)}\log \bE \exp\bigl(n(1+\alpha) \log E^F e^{tX}\bigr)]
 \Bigr\}\\
\le & \exp \Bigl\{-n[ t u -
 (1+\alpha)\Lambda(t)] \Bigr\},
\end{split}
\ee where $\Lambda(t) = \esssup_{\bP} \Lambda^F(t) $ and
$\Lambda^F(t) = \log E^F e^{tX}$.

From the assumption that there is $t_0 >0$ and $K_0 < \infty$ such
that $E^F e^{t_0\abs{X-\mu(F)}} < K_0$, it follows that for
$t\in(0,t_0)$, \be\begin{split} \Lambda^F(t) = &\mu(F) t+
\log\Bigl(1+ \frac{\sigma^2(F)}{2}t^2 + E^F \sum_{k=3}^\infty
\frac{t^k}{k!} (X-
\mu(F))^k\Bigr)\\
\leq &\mu(F) t+ \frac{\sigma^2(F)}{2}t^2 +
t^3E^F e^{t_0 \abs{X-\mu(F)}}\\
\leq  &\mu(F) t +  \frac{\sigma^2(F)}{2}t^2  + K_0 t^3,\\
\end{split}\ee and therefore if we denote $\sigma^* = \esssup_{\bP} \sigma(F)$, $$\Lambda(t) \leq \mu^* t + \frac{\sigma^{*2}}{2}t^2 +
K_0t^3,$$ from which we obtain \be
\label{generalupper3}\begin{split}\mP(T(\pi) \geq nu) \leq &\exp
\Bigl\{-n[ t u -
 (1+\alpha) (\mu^* t + \frac{\sigma^{*2}}{2}t^2 +
K_0t^3)]\Bigr\}\\
= & \exp \Bigl\{-n[ \bigl( u -
 (1+\alpha) \mu^* \bigr) t-(1+\alpha) \frac{\sigma^{*2}}{2}t^2-
(1+\alpha)K_0t^3 ]\Bigr\}.
\end{split} \ee

We take $u = \mu^* (1+\alpha)+\e$ and  $t = \frac{\e}{
\sigma^{*2}(1+\alpha)}$, where $\e = 4 \sigma^* \sqrt{\alpha \log
\frac{1}{\alpha}}$.  When $\alpha$ is small enough, $\e$ is small
and we get $t \in (0,t_0)$. Then \eqref{generalupper3} becomes \be
\mP(T(\pi) \geq nu) \leq
\exp\Bigl[-n\frac{\e^2}{2{\sigma^*}^2(1+\alpha)}\bigl(1 - \frac{2K_0
\e}{{\sigma^*}^4(1+\alpha)}\bigr)\Bigr] \leq
\exp\Bigl[-n\frac{\e^2}{4\sigma^{*2}(1+\alpha)}\Bigr] \ee when
$\alpha$ and thus $\e$ is small enough.

The last-passage time $T(n, \fl{n \alpha}) = \max_{\pi \in
\Pi(n,\fl{n\alpha} )} T(\pi)$. Here the maximum is taken over a pool
of $\binom{n + \fl{n\alpha}}{n}$ paths, so we can use Stirling's
formula to get \be
\begin{split}\mP\bigl(T(n, \fl{n \alpha})\geq nu\bigr) \leq& \binom{n +
\fl{n\alpha}}{n} \exp\Bigl[-n\frac{\e^2}{4{\sigma^*}^2(1+\alpha)}\Bigr]\\
\leq & \frac{C}{\sqrt{\alpha}}\exp\Bigl[-n\bigl(\alpha \log \alpha
-(1+\alpha)\log(1+\alpha)
 +\frac{\e^2}{4{\sigma^*}^2(1+\alpha)}\bigr)\Bigr].\\
\end{split}\ee

Plug in the expression for $\e$ we have \[ \alpha \log \alpha
-(1+\alpha)\log(1+\alpha)
 +\frac{\e^2}{4{\sigma^*}^2(1+\alpha)}\\ = \frac{3-\alpha}{1+\alpha}\alpha \log
\frac{1}{\alpha} - (1+\alpha)\log(1+\alpha).\] As $\alpha \searrow
0$, the first term on the right-hand side above has order $\alpha
\log \frac{1}{\alpha}$, whereas the second term has order $\alpha$.
So when $\alpha$ is small,
\[\alpha \log \alpha
-(1+\alpha)\log(1+\alpha)
 +\frac{\e^2}{4{\sigma^*}^2(1+\alpha)}>0\] and hence $$\sum_n \mP\bigl(T(n, \fl{n
\alpha})\geq nu\bigr) < \infty,
$$ which by Borel-Cantelli lemma gives \be \Psi(1,\alpha)\leq u= \mu^* (1+\alpha) + 4 \sigma^* \sqrt{\alpha \log \frac{1}{\alpha}},\ee and  we can claim as
$\alpha \searrow 0$, \be \Psi(1,\alpha)= \mu^* + O(\sqrt{\alpha \log
\frac{1}{\alpha}}). \ee \end{proof}

The order $\sqrt{\alpha\log\frac{1}{\alpha}}$ should be a rather
conservative estimate. In Theorem \ref{exp-thm} and in some other
settings we find that the first  $\alpha$-dependent term in
$\Psi(1,\alpha)$ is no more than $\sqrt{\alpha}$. There has not been
a very general statement about the necessary condition for the order
$\sqrt\alpha$ so far, but next we will see two sufficient
conditions. The first one is uniform boundedness.

\begin{thm}\label{upperbound} Let $\{F_j\}_{j\in \bZ_+}$ be an ergodic sequence of distribution functions
satisfying the conditions listed in Proposition \ref{pr:limit}.
Assume the existence of $M>0$ such that \be
\label{unifbdness}\bP\{F_0(-M)=0, F_0 (M)=1\}=1.\ee 
Again define $\mu^* = \esssup_{\bP}\mu(F)$, then as $\alpha \searrow
0$, \be \label{orderPsi}\Psi(1,\alpha) = \mu^* + O(\sqrt{\alpha}).
\ee
\end{thm}

\begin{proof}
We first prove an upper bound for $\Psi(1,\alpha)$. We start by
increasing all the weights $X(z)$  by moving their means to $\mu^*$,
so that  $\Psi(1,\alpha)$ for the shifted weights gets no smaller.
Then we subtract the common mean $\mu^*$ from the weights. Therefore
we can assume $\mu(F) = 0$ for all $F$. The weights $X(z)$ are still
uniformly bounded, and without loss of generality we still assume
\eqref{unifbdness} for the shifted weights with the bounds still
denoted by $M$.

 Fix a realization of $\{F_j\}_{j\in
\bZ_+}$, and the lattice point $z_0=(0,0)$. Let $N$ be a positive
integer. According to whether the path $\pi$ goes through $z_0$ or
not, and in case it goes we also separate the weight at $z_0$, we
write \be\label{split} \max_{\pi \in \Pi(n, \lfloor N\alpha  \rfloor
)} \sum_{z \in \pi} X(z) =
 A \vee(B + X(z_0))  = B + (A-B) \vee X_(z_0),
\ee where $A = \max_{z_0 \not \in \pi} \sum_{z\in \pi} X(z)$ and $B
= \max_{z_0 \in \pi } \sum_{z\in \pi\setminus\{z_0\}} X(z).$ Both
$A$ and $B$ look complicated but we only need to treat them as some
random variables. Let $G(y)$ denote the distribution of $A-B$, then
the quenched expectation \be E \bigl[(A-B)\vee X(z_0)\bigr] =
\int_{\bR \times \bR} x\vee y \,dF_0(x) dG(y) = \int_\bR
\Bigl(\int_{-M}^M x \vee y\, dF_0(x)\Bigr) dG(y)
.\label{meanofmax}\ee

We now take a closer look at $\int_{-M}^M  x\vee y\, dF_0(x)$. The
only nontrivial case is when $y \in [-M,M]$, integration by parts
gives \be
\begin{split} \int_{-M}^M x\vee y\, dF_0(x)&=yF_0(y) +
\int_y^M x \,dF_0(x)\\ & = yF_0(y) + \bigl(M - yF_0(y)\bigr)-
\int_y^M F_0(x)\, dx
\\ &= M - \int_y^M F_0(x)\, dx.
\end{split}\ee

Next we try to minimize the integral $\int_y^M F(x) dx$ when $F(x)$
is selected from mean zero distribution functions supported on $[-M,
M]$. Suppose the value $F(y)$ is given and $F(y) \ge \frac{1}{2}$,
then obviously $\int_y^M F(x) dx \ge (M-y)F(y)$, in which the equal
sign can be achieved for \be F(x) =
\begin{cases} 0 & x < -M \\ \frac{(1-F(y))M + y F(y)}{M+y}& -M \le x
<  y \\ F(y)&   y\le x < M\\ 1 & x \ge M.
\end{cases} \ee If $F(y)$ is known and $F(y) < \frac{1}{2}$ , the above function $F(x)$ would not work since it is then not non-decreasing.
Now we have \be\begin{split}\int_y^M F(x)dx &= \int_{-M}^M F(x) dx -
\int_{-M}^y
F(x) dx\\ & = M - \int_{-M}^M x dF(x)- \int_{-M}^y F(x) dx\\
&= M - EX- \int_{-M}^y F(x) dx\\
 &\ge M - F(y)(y+M) . \end{split}\ee The equality holds when \be
F(x)
= \begin{cases} 0 & x < -M \\F(y)& -M \le x <  y \\  \frac{(1-F(y))M - y F(y)}{M-y}&   y\le x < M\\
1 & x \ge M.
\end{cases} \ee

We now summarize the above two cases and see $\int_y^M F(x) dx \ge
\frac{1}{2}(M-y)$, with equality when $F(y) = \frac{1}{2}$, which
corresponds to the distribution function \be F^M(x) =
\begin{cases}
0 & x < -M\\
\frac{1}{2} & -M \le x < M\\
1  & M\le x.
\end{cases}
\ee Notice that $F^M(x)$ puts half probability on $M$ and $-M$ each
and does not depend on the value $y$. Back to \eqref{meanofmax}, we
see that for any random variables $A$ and $B$, $E\bigl[(A-B)\vee
X(z_0)\bigr]$ is maximized as long as we let $X(z_0)$ follow
$F^M(x)$. Running this argument for all $z \in \{0,..., N\} \times
\{0,..., \fl{N \alpha }\}$, we obtain \be \label{chooselarge} E
\max_{\pi \in \Pi(n, \lfloor N\alpha \rfloor )} \sum_{z \in \pi}
X(z)\le E \max_{\pi \in \Pi(n, \lfloor N\alpha \rfloor )} \sum_{z
\in \pi} X_{F^M}(z). \ee

Taking limits and using \eqref{martin-1} gives\be \Psi(1,\alpha) \le
\Psi_{F^M}(1,\alpha) = 2M\sqrt{\alpha} + o (\sqrt{\alpha}).\ee If we
consider the effect of $\mu^*$, we get \eqref{orderPsi}.

\end{proof}

The next two theorems also prove that $\Psi(1,\alpha)$ is a constant
plus order $O(\sqrt \alpha)$. They relax the assumption of uniform
boundedness and use the ones from \cite{Martin2004}. The finiteness
of the state space of $\bP$ plays an important role in the proofs of
both theorems, but it does not seem to be a necessary condition for
the results. It would be great if we can design a different approach
and remove this finiteness condition.

\begin{thm}
Assume  the process $\{F_j\}$ of probability distributions is
stationary, ergodic, and  has a state space of finitely many
distributions $H_1,\dotsc, H_L$  each of which satisfies Martin's
\cite{Martin2004} hypothesis \be  \int_0^\infty
(1-H_\ell(x))^{1/2}\,dx  +   \int_{-\infty}^0
H_\ell(x)^{1/2}\,dx<\infty. \label{martin8}\ee Let $\mu^*=\max_\ell
\mu(H_\ell)$ be the maximal mean of the $H_\ell$'s.
 Then there exist
  constants $0<c_1<c_2<\infty $   such that,  as $\alpha\downarrow 0$,
\be   \mu^* + c_1  \sqrt{\alpha}+ o(\sqrt{\alpha}\,)\; \le \;
\Psi(1, \alpha) \; \le \; \mu^* + c_2  \sqrt{\alpha}+
o(\sqrt{\alpha}\,).  \label{1,a-bd}\ee \label{1,a-thm}\end{thm}

\begin{proof}
   The lower bound  in \eqref{1,a-bd}  can be  proved by applying Martin's result
\eqref{martin-1}  to the homogeneous  problem where a maximal path
is constructed by using only those rows $j$  where $F_j=H_{i^*}$,
the distribution with the maximal mean $\mu^*=\mu(H_{i^*})$.  This
is fairly straightforward.

To prove the upper bound  in \eqref{1,a-bd}, we again start by
increasing all the weights $X(z)$  by moving their means to $\mu^*$.
Then we subtract the common mean $\mu^*$ from the  weights, so that
for the proof we can assume that all distributions $H_1,\dotsc, H_L$
have {\sl mean zero.}

Create the following coupling.  Independently of the process
$\{F_j\}$, let $\{X_\ell(z): 1\le \ell\le L, \, z\in\bZ_+^2\}$ be a
collection of independent weights such that  $X_\ell(z)$ has
distribution  $H_\ell$.  Then define the weights used for computing
$\Psi(1,\alpha)$ by
\[   X(z)= \sum_{\ell=1}^L   X_\ell(z) I_{\{F_j=H_\ell\}}
\quad\text{for $z=(i,j)\in\bZ_+^2$.}  \] Begin with  this elementary
bound: \be\begin{aligned} \Psi(1,\alpha) &= \lim_{n \rightarrow
\infty} \frac{1}{n}\mE\Bigl[ \;\max_{\pi \in
\Pi(n, \lfloor \alpha n \rfloor)} \sum_{z \in \pi} X(z)\,\Bigr] \\
&\le \sum_{\ell=1}^L   \lim_{n \rightarrow \infty}
\frac{1}{n}\mE\Bigl[ \;\max_{\pi \in \Pi(n, \lfloor \alpha n
\rfloor)} \sum_{z \in \pi} X_\ell(z) I_{\{F_j=H_\ell\}}\,\Bigr].
\end{aligned}  \label{line-b5} \ee
The next lemma contains a convexity argument that will remove the
indicators from the last-passage values above.

\begin{lem} Let $\cD$ be a sub-$\sigma$-field on a probability space $(\Omega, \cF, P)$,
$D$ an event in $\cD$,  and $\xi$ and $\eta$ two integrable random
variables.  Assume that $E\eta=0$, $\eta$ is independent of $\cD$,
and   $\xi$ and $\eta$ are independent conditionally on $\cD$. Then
$E[\,\xi\vee (\eta I_D)\,]\le E[\,\xi\vee\eta\,]$.
\label{conv-lem}\end{lem}
\begin{proof}   By Jensen's inequality, for any fixed $x\in\bR$,
\[   x\vee E(\eta\,\vert\,\cD) \le  E( x\vee \eta\,\vert\,\cD). \]
Since $\eta$ is independent of $\cD$ and mean zero,
\[   x\vee 0 \le  E( x\vee \eta\,\vert\,\cD). \]
Integrate this against the conditional distribution $P(\xi\in
dx\,\vert \,\cD)$ of $\xi$, given $\cD$, and use the conditional
independence of $\xi$ and $\eta$:
\[   E( \xi\vee 0 \,\vert\,\cD)  \le  E( \xi\vee \eta\,\vert\,\cD). \]
Next integrate this over the event $D^c$:
\[  E\bigl[ I_{D^c} \cdot\,\xi\vee (\eta I_D)\, \bigr]   =E\bigl[ I_{D^c} \cdot \, \xi\vee 0\,\bigr]
\le  E\bigl[ I_{D^c} \cdot\,\xi\vee \eta \, \bigr].
\]
The corresponding integral over the event $D$ needs no argument.
\end{proof}

Fix a lattice point $z_0=(i_0,j_0)$ for the moment.  We split the
maximum in \eqref{line-b5} like the way we did in \eqref{split}:
\[ \max_{\pi
\in \Pi(n, \lfloor n\alpha  \rfloor )} \sum_{z \in \pi} X_\ell(z)
I_{\{F_j=H_\ell\}}
 = B +\bigl(A-B\bigr)\vee \bigl(X_\ell(z_0) I_{\{F_{j_0}=H_\ell\}} \bigr)
\] where
\[ A = \max_{\pi\not\ni z_0 } \sum_{z\in \pi} X_\ell(z)I_{\{F_j=H_\ell\}}
\quad \text{and}\quad
 B = \max_{\pi\ni z_0  } \sum_{z\in \pi\setminus\{z_0\} } X_\ell(z)I_{\{F_j=H_\ell\}} .\]
Now apply Lemma \ref{conv-lem} with $\xi=A-B$, $\eta=X_\ell(z_0)$,
and $D=\{F_{j_0}=H_\ell\}$.  Given $F_{j_0}$,  $A-B$ does not look
at $X_\ell(z_0)$, so the independence assumed in Lemma
\ref{conv-lem}  is satisfied.  The outcome from that lemma is the
inequality
\begin{align*}
\mE\Bigl[ \;\max_{\pi \in \Pi(n, \lfloor \alpha n \rfloor)} \sum_{z
\in \pi} X_\ell(z) I_{\{F_j=H_\ell\}}\,\Bigr] \le \mE\bigl[  A\vee
(B + X_\ell(z_0)) \bigr].
\end{align*}
This is tantamount to replacing the weight $X_\ell(z_0)
I_{\{F_{j_0}=H_\ell\}} $ at $z_0$  with $X_\ell(z_0)$.

 We can repeat this at all lattice points $z_0$ in  \eqref{line-b5}.  In the end we have
 an upper bound in terms of homogeneous last-passage values, to which we
 can apply Martin's result \eqref{martin-1}:
\begin{align*}
\Psi(1,\alpha) & \le \sum_{\ell=1}^L   \lim_{n \rightarrow \infty}
\frac{1}{n}\mE\Bigl[ \;\max_{\pi \in \Pi(n, \lfloor \alpha n
\rfloor)} \sum_{z \in \pi} X_\ell(z)  \,\Bigr]= \sum_{\ell=1}^L
\Psi_{H_\ell}(1,\alpha)
\\
&=  2\sqrt\alpha  \sum_{\ell=1}^L \sigma(H_\ell)  + o(\sqrt\alpha).
\end{align*}
This completes the proof of Theorem \ref{1,a-thm}.
\end{proof}

The gist of the above theorem is the inequality \be\Psi(1,\alpha)
\le \sum_{\ell=1}^L \Psi_{H_\ell}(1,\alpha). \label{eqn:1,a-thm}\ee
We will use it repeatedly in the proof of the following theorem.

The following theorem also  shows an order of $\sqrt{\alpha}$ given
that the state space of $\bP$ is finite. It uses a similar approach
as in the proof of Theorem \ref{thm:main} and is much lengthier than
the previous result, but it gives a better coefficient of
$\sqrt{\alpha}$ in the sense that it does not depend on the size
$L$. Therefore it gives some insight on the possibility to remove
the finiteness condition.

\begin{thm}\label{upperbound} Let $\{F_j\}_{j\in \bZ_+}$ be an i.i.d. sequence under $\bP$ from a finite set of distributions $\{H_1,..., H_L\}$.
 Again assume for
each $\ell$, \be \label{uppercond1}\int_{-\infty}^0
H_\ell(x)^{1/2}\, dx + \int_0^\infty \bigl(1-H_\ell(x)\bigr)^{1/2}\,
dx < \infty.\ee Then  as $\alpha \searrow 0$: \be
\label{upper1}\Psi(1,\alpha)\leq \mu^* + 2\sigma^*\sqrt{\alpha} +
o(\sqrt{\alpha}),\ee where $\mu^* = \max_\ell\{\mu_(H_\ell)\}$ and
$\sigma^* = \sqrt{\max_\ell \{\sigma^2(H_\ell) \}}$.
\end{thm}

\begin{proof}
Again if we assume $\mu(H_\ell) = \mu^*$ for all $\ell=1,..., L$,
$\Psi(1,\alpha)$ will get no smaller. Without loss of generality, we
will assume $\mu(H_\ell) \equiv 0$ hereafter unless specified
otherwise.

In a similar way as we did in the proof of Theorem \ref{thm:main},
we let $r=r(\alpha)$ be a positive integer-valued function such that
$r(\alpha) \nearrow \infty$ and $r\sqrt{\alpha} \searrow 0$ as
$\alpha\searrow 0$. Define the $r\times 1$ blocks as
 $B_r(x,y) = \{(rx+i,y): i=0,1,..., r-1\}$
 for $(x,y)\in \bZ_+^2$.

For every point $z=(i,j)\in \bZ^2_+$, write $X_r(z) = \sum_{v\in
B_r(z)} X(v)$. The distribution function of $X_r(z)$ is denoted as
$F_{r,j}(x)$, and the corresponding last-passage time function for
$X_r(z)$ is then denoted as $\Psi_r(x, y)$. Also, write $H_r^\ell$
for the convolution $H_\ell*H_\ell*...*H_\ell$ with $H_\ell$
repeated $r$ times.

 For $j\in \bZ_+$, let $\Phi_{r,j}$ be the distribution function of the normal
distribution $\cN(0, r\sigma^2_j )$, where $\sigma^2_j = \Var(F_j)$.
For each $z= (i,j) \in \bZ^2_+$, let $Y_{r}(z)$ be a random variable
with distribution $\Phi_{r,j}$. Write
$$\Psi_{\Phi_r}(x,y) =  \lim_{n \rightarrow \infty} \frac{1}{n} \mE\max_{\pi \in
\Pi(\fl{nx}, \fl{ny})} \sum_{z \in \pi} Y_r(z).$$ Also, let
$\Phi_r^\ell$ be the distribution function of $\cN(0, rV_\ell)$,
where $V_\ell$ is the variance of $H_\ell$. We make an approximation
first:
\begin{lem}\be \label{est1}\lim_{\alpha \downarrow
0}\frac{1}{\sqrt{\alpha}}\abs{\Psi(1,\alpha) -
\frac{1}{r}\Psi_{\Phi_r}(1,  r\alpha)  } = 0.\ee
\end{lem}
\begin{proof}
For each $\ell= 0,1, ..., L$ and each $z = (i,j) \in \bZ^2_+$,
 we do the following coupling. Define $\{u(z):z = (i,j) \in
\bZ^2_+\} $ be i.i.d.\ Uniform$(0,1)$ random variables. Set $X_r(z)
={F_{r,j}}^{-1}(u(z))$, where ${F_{r,j}}^{-1}(u) = \sup\{x :
F_{r,j}(x)< u\}$. Similarly define $Y_r(z) =\Phi_{r,j}^{-1}(u(z))$.
Also define $X_{r,\ell}(z) = {F_{r,j}}^{-1}(u(z))I_{\{F_j =
H_\ell\}}$ and $Y_{r,\ell}(z) = \Phi_{r,j}^{-1}(u(z))I_{\{F_j =
H_\ell\}}$.

We first assume $\{H_1, H_2, \ldots, H_L\}$ are uniformly bounded
and directly quote the computation from Lemma 4.2 and Lemma 4.5 in
\cite{Martin2004}: \be \label{approx1}
\begin{split}& \lim_{n \rightarrow \infty} \frac{1}{n} E \max_{\pi
\in \Pi(n, \fl{nr\alpha} )} \sum_{z \in \pi}\bigl(
(H_r^\ell)^{-1}(u(z))- (\Phi_r^\ell)^{-1}(u(z))\bigr)\\
\leq& 2\sqrt{r\alpha(1+r\alpha)} \int_{-\infty}^\infty |H_r^\ell(s)- \Phi_r^\ell(s)|^\frac{1}{2}\,ds\\
\leq &2\sqrt{r\alpha(1+r\alpha))} \int_{-\infty}^\infty \sqrt{C
r^{-\frac{1}{2}}(1+|\frac{s}{\sqrt{r}\sigma_\ell}|)^{-3}} ds\\
\leq & C r^{\frac{3}{4}} \sigma_\ell \sqrt{\alpha(1+r\alpha)}
\end{split}\ee
for proper constant $C$ that are independent of $r$, $\ell$ and
$\alpha$.

Then we apply \eqref{eqn:1,a-thm}: \be \label{approx2}\begin{split}
&\Psi_r(1,r\alpha) -
\Psi_{\Phi_r}(1,r\alpha)\\
\leq & \varlimsup_{n \rightarrow \infty} \frac{1}{n} \mE \max_{\pi
\in \Pi(n, \fl{nr\alpha})} \sum_{z \in \pi}\bigl(
X_r(z)-Y_r(z) \bigr)\\
\leq & \sum_{\ell=0}^L \varlimsup_{n \rightarrow \infty} \frac{1}{n}
\mE \max_{\pi \in \Pi(n,\fl{nr\alpha} )} \sum_{z \in \pi}\bigl(
X_{r,\ell}(z)-Y_{r,\ell}(z) \bigr)\\
\leq &\lim_{n \rightarrow \infty} \frac{1}{n} \mE \max_{\pi \in
\Pi(n, \fl{nr\alpha} )} \sum_{z \in \pi}\bigl(
(H_r^\ell)^{-1}(u(z))- (\Phi_r^\ell)^{-1}(u(z))\bigr)\\
=&\sum_{\ell=0}^L Cr^{\frac{3}{4}} \sigma_\ell \sqrt{\alpha(1+r
\alpha)}.\end{split} \ee If we switch the two terms on the left hand
side, we can repeat the calculation and obtain \be \label{approx}
\frac{1}{r\sqrt{\alpha} }|\Psi_r(1, r\alpha) - \Psi_{\Phi_r}(1,
r\alpha)| \leq \sum_{k=0}^K C\sigma_\ell
r^{-\frac{1}{4}}\sqrt{1+r\alpha}, \ee which goes to 0 as $\alpha
\searrow 0$.

Lemma 4.4 in \cite{Martin2004} still applies here since $X(z)$'s are
uniformly bounded, so we have
\[\lim_{\alpha \downarrow 0}\frac{1}{\sqrt{\alpha} }\abs{\Psi(1,\alpha)- \frac{1}{r}\Psi_r(1, r\alpha)} = 0.\]
Therefore \be \label{est1}\lim_{\alpha \downarrow
0}\frac{1}{\sqrt{\alpha}}\abs{\Psi(1,\alpha) -
\frac{1}{r}\Psi_{\Phi_r}(1,r\alpha)  } = 0.\ee

If we do not have uniform boundedness, Lemma 4.3 of
\cite{Martin2004} says that for any $\e>0$, we can find distribution
functions $\widetilde{H}_\ell$ for each $k$ with bounded support,
and with mean and variance equal to those of $H_\ell$, and
$\int_{-\infty}^\infty \abs{\widetilde{H}_\ell(s)-
H_\ell(s)}^{\frac{1}{2}}ds < \e$. Hence \eqref{est1} holds for
$\widetilde{\Psi}(1,\alpha)$, the corresponding last-passage time
function associated with $\{\wt H_1, \ldots, \wt H_L\}$.

We repeat the argument leading to \eqref{approx1} and
\eqref{approx2} and get
\be \label{aprox3}\begin{split} &\abs{\widetilde{\Psi}(1,\alpha) - \Psi(1,\alpha)}\\
\leq & \sum_{\ell=1}^L 2\sqrt{\alpha(1+\alpha)}
\int_{-\infty}^\infty
\abs{\widetilde{H}_\ell(s)- H_\ell(s)}^{\frac{1}{2}}ds\\
\leq & 2L\e \sqrt{\alpha(1+\alpha)}.
\end{split}\ee
We let $\e$ approach 0 and it follows that \eqref{est1} is also
valid for $\{H_\ell\}$. The lemma is proved.
\end{proof}

With Lemma \ref{est1}, we know that $\Psi(1,\alpha)$ has the same
coefficient of the term $\sqrt{\alpha}$ with
$\frac{1}{r}\Psi_{\Phi_r}(1,r\alpha)$. $\Psi_{\Phi_r}(1,r\alpha)$ is
the last-passage constant of a model where all $X(z)$ follow mean
zero normal distributions. The following lemma looks at the role
played by the variances of normal distributions.

\begin{lem}\label{normal} Let $X$ and $Y$ be independent random
variables  and $X \sim \cN(0, \sigma^2)$, then $E (X\vee Y)$ is an
increasing function of $\sigma$.
\end{lem}
\begin{proof} Firstly, we have
\[E (X\vee Y) = \int_{-\infty}^\infty \bigl\{\int_{-\infty}^y y \frac{1}{\sqrt{2\pi}\sigma} \exp(-\frac{x^2}{2\sigma^2}) dx + \int_{y}^\infty x \frac{1}{\sqrt{2\pi}\sigma} \exp(-\frac{x^2}{2\sigma^2})dx \bigr\} dG(y).\]
Note that $$\int_{-\infty}^y x \frac{1}{\sqrt{2\pi}\sigma}
\exp(-\frac{x^2}{2\sigma^2})dx + \int_y^\infty x
\frac{1}{\sqrt{2\pi}\sigma} \exp(-\frac{x^2}{2\sigma^2})dx = 0,$$ we
get
\begin{align*}
&\int_{-\infty}^y y \frac{1}{\sqrt{2\pi}\sigma}
\exp(-\frac{x^2}{2\sigma^2}) dx + \int_{y}^\infty x
\frac{1}{\sqrt{2\pi}\sigma} \exp(-\frac{x^2}{2\sigma^2})dx \\
=& \int_{-\infty}^y y\frac{1}{\sqrt{2\pi}\sigma}
\exp(-\frac{x^2}{2\sigma^2})dx -  \int_{-\infty}^y x
\frac{1}{\sqrt{2\pi}\sigma}
\exp(-\frac{x^2}{2\sigma^2})dx\\
=& \int_{-\infty}^y (y-x)\frac{1}{\sqrt{2\pi}\sigma}
\exp(-\frac{x^2}{2\sigma^2})dx \\
=& \int_{-\infty}^{\frac{y}{\sigma}} (y-\sigma x)
\frac{1}{\sqrt{2\pi}} \exp(-\frac{x^2}{2})dx.
\end{align*}

The last line can be viewed as a function of $\sigma$ and its
derivative is
$$-\frac{1}{\sqrt{2\pi}}\int_{-\infty}^{\frac{y}{\sigma}}
 x \exp(-\frac{x^2}{2})dx>0,$$
because when $y\le 0$, it is easy to see
$\int_{-\infty}^{\frac{y}{\sigma}}
 x \exp(-\frac{x^2}{2})dx <0$; when $y>0$, \newline $\int_{-\infty}^{\frac{y}{\sigma}}
 x \exp(-\frac{x^2}{2})dx < \int_{-\infty}^{\infty}
 x \exp(-\frac{x^2}{2})dx=0$.

 Therefore $E(X\vee Y)$ is an increasing function of $\sigma$.
\end{proof}

With Lemma \ref{normal}, we can run a similar argument with the one
leading to \eqref{chooselarge}. That is, we fix a $z_0\in \bZ_+^2$,
and claim that $E \max_{\pi \in \Pi(n, \fl{n\alpha})}\sum_{z\in
\pi}Y_r(z)$ is maximized when $Y_r(z_0)$ has the largest possible
standard variance $r\sigma^{*2}$. Repeat this reasoning we see that
an upper bound for $\Psi_{\Phi_r}(1,r\alpha)$ is given if we let all
sites $z\in \bZ_+^2$ have the largest possible variance. So
\eqref{martin-1} can be applied here: as $\alpha \searrow 0$,
$$\Psi_{\Phi_r}(1,
r\alpha) \leq 2\sqrt{r}\sigma^* \sqrt{r\alpha}+o(\sqrt{r\alpha}).$$
From \ref{est1}, as $\alpha \searrow 0,$ we have
\be\label{tempresult}\Psi(1,\alpha) \leq 2\sigma^*\sqrt{\alpha} +
o(\sqrt{\alpha}).\ee

Consider the case $\mu^* \not = 0$, the last result becomes
\be\Psi(1,\alpha) \leq \mu^*+2\sigma^*\sqrt{\alpha} +
o(\sqrt{\alpha}).\ee
\end{proof}

\begin{rem}
This theorem did not remove the finiteness of the state space
because the approximations \eqref{approx2} and \eqref{aprox3} depend
on the size $L$. A more accurate method of approximation is needed
in order to lift this assumption.\end{rem}


%
%
\bibliographystyle{plain}
\bibliography{cite,growthrefs}

\end{document}